\documentclass[a4paper, 10pt, twoside, notitlepage]{amsart}
\usepackage[utf8]{inputenc}
\usepackage{xcolor}
\usepackage{amsmath} 
\usepackage{amssymb} 
\usepackage{amsthm}
\usepackage{graphicx}
\usepackage{esint}
\usepackage{subcaption}
\usepackage{amsmath, amssymb, graphics, setspace}

\usepackage{soul}

\newcommand{\mathsym}[1]{{}}
\newcommand{\unicode}[1]{{}}

\usepackage[colorlinks=true,linkcolor=blue]{hyperref}
\usepackage[export]{adjustbox}
\usepackage{thmtools}

\usepackage{enumitem}

\theoremstyle{plain}

\newtheorem{prop}{Proposition}[section]

\newtheorem{lem}[prop]{Lemma}
\newtheorem{cor}[prop]{Corollary}

\newtheorem{defi}[prop]{Definition}
\newtheorem{rmk}[prop]{Remark}

\newcommand {\R} {\mathbb{R}} \newcommand {\Z} {\mathbb{Z}}
 \newcommand {\N} {\mathbb{N}}
\newcommand {\C} {\mathbb{C}} 
\newcommand {\p} {\partial}

\newcommand {\diam} {\text{diam}}

\newcommand{\beq}[0]{\begin{equation}}
\newcommand{\eeq}[0]{\end{equation}}

\DeclareMathOperator{\diag}{diag}
\DeclareMathOperator{\Id}{Id}
\DeclareMathOperator{\tr}{tr}

\DeclareMathOperator {\dist} {dist}

\DeclareMathOperator {\inte} {int}

\begin{document}

\title[From elastic crystals to nematic elastomers]{Exact constructions in the (non-linear) planar theory of elasticity: From elastic crystals to nematic elastomers}
\author[P. Cesana]{Pierluigi Cesana}

\address{Institute of Mathematics for Industry, Kyushu University, Fukuoka, Japan
and
Department of Mathematics and Statistics, La Trobe University, Australia}
\email{cesana@math.kyushu-u.ac.jp}

\author[F. Della Porta]{Francesco Della Porta}
\address{Max-Planck-Institute for Mathematics in the Sciences, Inselstr. 22, 04103 Leipzig}
\email{DellaPorta@mis.mpg.de}

\author[A. Rüland]{Angkana Rüland}
\address{Max-Planck-Institute for Mathematics in the Sciences, Inselstr. 22, 04103 Leipzig}
\email{rueland@mis.mpg.de}

\author[C. Zillinger]{Christian Zillinger}
\address{Department of Mathematics
University of Southern California
3620 S. Vermont Avenue
Los Angeles, CA 90089-2532 }
\email{zillinge@usc.edu}

\author[B. Zwicknagl]{Barbara Zwicknagl}
\address{Technische Universität Berlin
Institut für Mathematik,
Sekretariat MA 6-4,
Straße des 17. Juni 136,
10623 Berlin}
\email{zwicknagl@math.tu-berlin.de}

\begin{abstract}
In this article we deduce necessary and sufficient conditions for the presence of ``Conti-type'', highly symmetric, exactly-stress free constructions in the geometrically non-linear, planar $n$-well problem, generalising results of \cite{CKZ17}. Passing to the limit $n\rightarrow \infty$, this allows us to treat solid crystals and nematic elastomer differential inclusions simultaneously. In particular, we recover and generalise (non-linear) planar tripole star type deformations which were experimentally observed in \cite{MA80,MA80a,KK91}. Further we discuss the corresponding geometrically linearised problem.
\end{abstract}

\maketitle

\section{Introduction}

It is the purpose of this article to discuss certain specific, stress-free constructions for two-dimensional models of shape-memory alloys and nematic liquid crystal elastomers in a unified mathematical framework. Both of these physical systems can be described by highly non-quasi-convex energies  within the calculus of variations, which formally share important features and give rise to complex and wild microstructures. Before turning to our mathematical results, let us thus first describe the physical background of these models, discussing their common features and the problems we are interested in.

\subsection{Elastic crystals}
\label{sec:elastic}

Shape-memory alloys are solid, elastic crystals which undergo a first order, diffusionless solid-solid phase transformation in which symmetry is reduced upon the passage from the high temperature phase, \emph{austenite}, to the low temperature phase, \emph{martensite}. Due to the loss of symmetry there are typically various, energetically equivalent \emph{variants of martensite} in the low temperature phase. Mathematically, shape memory alloys have been very successfully modelled within a variational framework introduced by Ball and James \cite{B3}, where it is assumed that the observed deformations of a material minimise an energy functional of the form
\begin{align}
\label{eq:energy}
\int\limits_{\Omega} W(\nabla u, \theta) dx.
\end{align}
Here $\Omega$ denotes the reference configuration, which is typically chosen to be the material in the austenite phase at a fixed temperature, $u:\Omega \subset \R^d \rightarrow \R^d$ is the deformation of the material, $\theta: \Omega \rightarrow (0,\infty)$ represents temperature and $W: \R^{d\times d} \times \R_+ \rightarrow \R_+$ is the stored energy density. Physical requirements on the stored energy density are
\begin{itemize}
\item \emph{frame indifference}, which implies that 
\begin{align*}
W(Q F) = W(F) \mbox{ for all } Q \in SO(d),
\end{align*}
\item \emph{invariance under material symmetries}, by which
\begin{align*}
W(F H) = W(F) \mbox{ for all } H \in \mathcal{P}.
\end{align*} 
Here $\mathcal{P}$ denotes the point group of austenite, which is a (discrete) subgroup of $O(d)$.
\end{itemize}
These two conditions render the described models for martensitic phase transformations highly \emph{non-linear}, \emph{non-quasi-convex} and give rise to rich microstructures \cite{B}. 
The above two conditions on $W$ in particular determine the associated \emph{energy wells} $K(\theta)$, which are characterised by the condition 
\begin{align*}
W(F,\theta) = 0 \mbox{ iff } F \in K(\theta).
\end{align*}
Typically, $K(\theta)$ is of the form
\begin{align}
\label{eq:n_wells}
\begin{split}
K(\theta) = 
\left\{
\begin{array}{ll}
\alpha(\theta) SO(d) \mbox{ for } \theta > \theta_c,\\
\alpha(\theta) SO(d)\cup \bigcup\limits_{j=1}^{m}SO(d)U_j(\theta) \mbox{ for } \theta= \theta_c,\\
\bigcup\limits_{j=1}^{m}SO(d)U_j(\theta) \mbox{ for } \theta< \theta_c,
\end{array} 
\right.
\end{split}
\end{align}
where $\theta_c \in (0,\infty)$ denotes the transformation temperature, $\alpha(\theta):(0,\infty)\rightarrow \R_+$ is a thermal expansion coefficient, $\alpha(\theta)SO(d)$ models the austenite phase (taken as the reference configuration at the critical temperature, i.e. $\alpha(\theta_c)=1$) and $SO(d)U_j(\theta)$ represents the respective variants of martensite, where $U_j(\theta)\in \R^{d\times d}$, see \cite{Ball:ESOMAT}.
Here the matrices $U_j(\theta)$ are obtained through the action of the symmetry group from $U_1(\theta)$, i.e. for each $j\in\{1,\dots,m\}$ there exists $P\in \mathcal{P}$ such that
\begin{align*}
U_j(\theta) = P U_1(\theta)P^T.
\end{align*}
Due to the complicated and highly non-linear and non-convex structure of the energies in \eqref{eq:energy}, a commonly used first step towards the analysis of low energy microstructures in martensitic phase transformations is the analysis of the differential inclusion 
\begin{align}
\label{eq:ex_stress_free}
\nabla u \in K(\theta),
\end{align}
which corresponds to the determination of \emph{exactly stress-free} states. A class of particularly symmetric, exactly stress-free deformations had been studied by Conti \cite{C} in specific set-ups (we will also refer to these as ``Conti constructions''), see also the precursors in \cite{MS,CT05}. It is the purpose of this article to study these structures systematically in the sequel, following and extending ideas from \cite{CKZ17} and treating elastic and nematic liquid crystal elastomers in a unified framework.

\subsection{From elastic crystals to nematic elastomers}

Nematic liquid crystal elastomers (NLCEs) are a class of \textit{soft} shape-memory alloys where shape-recovery is accompanied by the emergence of soft modes and mechanical and optical instabilities. Constitutively, NLCEs are rubber-like elastic materials composed of cross-linked polymeric chains incorporating molecules of a nematic liquid crystal. We refer to \cite{WT03} for an extensive description of the synthesis and physical properties of NLCEs. The complicated interaction between orientation of the liquid crystal molecules (described by $\hat{n}(x)$, a unit vector field called the director) and the macroscopic strain field generated by the polymeric chains may induce optical isotropy, low-order states of the nematic molecules and shear-banding of martensitic type. As a typical signature of the nematic-elastic coupling, NLCEs spontaneously deform when an assigned orientation is imposed (for instance, by an external electric field) to the liquid crystal molecules. Conversely, a macroscopic deformation induces a rotation and re-orientation of the nematic molecules in a way that the director tends to be parallel to the direction of the largest principal stretch associated with the deformation.

Let us comment on the passage from solid to nematic liquid crystal elastomers.
Despite the profound differences in the nature of elastic crystals (martensite) and nematic-elastomers
it turns out that the morphology of the microstructures observed in both these materials  may   be modelled with the language of continuum mechanics by means of multi-well energies of a similar -- at least formally -- structure and shape yielding in both cases highly non-quasi-convex variational problems.

In the context of NLCE typical stored energy densities   may be considered in the general form 
\cite{ADMDS15}
\begin{align}\label{1902051753} 
W(F):= \sum\limits_{j=1}^N \frac{d_j}{\gamma_j} \left[ 
\left( \frac{\lambda_1(F)}{c_1} \right)^{\gamma_j} + \left( \frac{\lambda_2(F)}{c_2}\right)^{\gamma_j} +\left( \frac{\lambda_3(F)}{c_3}\right)^{\gamma_j}  -3
\right], \ \textrm{if } \det(F)=1,
\end{align}
and $+\infty$, if $\det F\neq 1$.
The matrix $F\in \R^{3\times 3}$ denotes the deformation gradient of the material and $\lambda_k(F)$ are its ordered singular values, that is, the square root of the eigenvalues of the matrix $FF^T$, under the assumption $0<\lambda_1\le \lambda_2\le\lambda_3$. 
Finally, $0< c_1 \leq  c_2 \leq c_3 < \infty$ as well as 
$d_j$ and $\gamma_j\in [2,\infty)$  
 are constants.

Stored energy densities of the form 
(\ref{1902051753})    
comprise the classical energy model for NLCEs of Bladon, Warner and Terentjev (BWT) \cite{BWT94} which is obtained by setting 
$N=1,\gamma_j=2,d_j=\mu$ (shear modulus)
and $c_1=c_2=r^{-1/6}, c_3=r^{1/3}$ 
(where $r>1$ is the backbone anisotropy parameter) into (\ref{1902051753}).
By operating this substitution 
we obtain the BWT energy density which we write -- with some abuse of notation -- as
\begin{eqnarray}\label{1901311846}
W(F)=\frac{\mu}{2}\left[ r^{1/3}\lambda_1^2(F) +r^{1/3} \lambda_2^2(F) +\frac{\lambda_3^2(F)}{r^{2/3} } -3\right].
\end{eqnarray}
Moreover, $W(F)=\min_{\hat{n}\in\mathbb{S}^2}\tilde{W}(F,\hat{n})$, 
where
\begin{align}\label{1901311844}
\tilde{W}(F, \hat{n})=\frac{\mu}{2}\left(r^{1/3}\left[\tr(FF^T)-\frac{r-1}{r}FF^T \hat{n}\cdot \hat{n} \right]-3\right)
\quad \textrm{ if } \det(F)=1, \hat{n}\in\mathbb{S}^2
\end{align}
(and extended to $+\infty$ if $\det(F)\neq 1$ or $\hat{n}\notin\mathbb{S}^2$) and $\hat{n}$ is the nematic director.
Notice in (\ref{1901311844}) the energy density is constant if we replace $\hat{n}$ with $-\hat{n}$: this is the so-called head-tail symmetry of nematic liquid crystals, a fundamental physical property which is incorporated in all the most typical models of both nematic liquid and solid-liquid crystals including the ones discussed here.

Similarly as in the elastic crystal setting in shape-memory materials, in studying minimisers of \eqref{1902051753} or \eqref{1901311846} a first commonly used approach is to consider the associated differential inclusion describing exactly stress-free states. In the case of \eqref{1902051753} this leads to the study of the following problem:
\begin{align}
\label{eq:K_liq}
\nabla u \in K_{\infty}:= \{F\in \R^{3\times 3}; \ \det(F) =1, \lambda_k(F) = c_k, k \in \{1,2,3\}\},
\end{align}
where $K_{\infty}$ corresponds to the zero-energy level of   $W$.
Observe that $W\ge 0$ and  $W(F)=0$ if and only if $F \in K_{\infty}$.
In contrast to the finite number of wells in the elastic crystal case, one is now confronted with an infinite number of energy wells. 

This is evident if we investigate the zero-energy level of $\tilde{W}(F,\hat{n})$. Simple algebraic computations show that  
$
 \min_{F,\hat{n}}\tilde{W}(F,\hat{n})=0
$
and that the minimum is achieved by any pair $(\overline{F},\overline{n})$ such that $\lambda_1=\lambda_2=r^{-1/6}, \lambda_3=r^{1/3}$ and $\overline{n}$ coincides with the eigenvector associated with the largest eigenvalue of $\overline{F}\overline{F}^T$ or, equivalently, by any pair $(U_{\hat{n}},\hat{n})$ where $\hat{n}$ is any vector in $\mathbb{S}^2$ and 
\begin{eqnarray}\label{1901312036}
U_{\hat{n}}=r^{1/3}\hat{n}\otimes \hat{n}+r^{-1/6}(Id-\hat{n}\otimes \hat{n}),
\end{eqnarray}
where $Id \in \R^{3\times 3}$ is the identity matrix. Deformations of the form stated in equation  (\ref{1901312036}), which are the equivalent of the bain strain in martensite, correspond to a spontaneous distortion of a ball of radius one into a prolate ellipsoid whose major axis (of length $r^{1/3}$) is parallel to $\hat{n}$. 
 For the NLCE model of (\ref{1901311846}) the energy well is obtained by plugging 
$c_1=c_2=r^{-1/6}, c_3=r^{1/3}$
 into $K_{\infty}$ (see \eqref{eq:K_liq}) which leads to the differential inclusion
\begin{eqnarray}\label{1902010046}
\nabla u \in \tilde{K}_{\infty}(r)=
\bigcup_{\hat{n}\in\mathbb{S}^2}SO(3)U_{\hat{n}}.
\end{eqnarray}
Equation (\ref{1902010046}) is resemblant of the situation described by the equations
\eqref{eq:n_wells}-\eqref{eq:ex_stress_free} for martensite, where one has replaced $\mathcal{P}$, the (discrete) point group of the material with the full group $SO(3)$. This is indeed the striking property of NLCE models: The stored energy is invariant under rotations in the ambient space as well as under the action of $SO(3)$.

This formal similarity of the two problems suggests that they can be analysed in similar frameworks. In Lemma \ref{lem:diff_incl_n_infty} we show that the set \eqref{eq:K_liq} can be obtained as the limit $n\rightarrow \infty$ of sets of the type \eqref{eq:n_wells}. Moreover, even for finite $n \in \N$ the sets from \eqref{eq:n_wells} are always subsets of the set $K_{\infty}$, hence any solution obtained for finite $n$ is also always a solution to the differential inclusion problem for \eqref{eq:K_liq} in a corresponding $n$-gon domain. This could for instance be exploited in numerical benchmarking (see the discussion in Section \ref{sec:solids_liquids}). Due to these similarities, in the sequel we  seek to discuss the two physical systems simultaneously.

A series of experiments and technological implementations which appeared over the last three decades have inspired and motivated an extensive body of work on the modelling and design of microstructure formation in NLCEs. Special focus has been given on the formation of martensitic-type stripe-domains (experimentally observed in \cite{KF95}, analysed  under the assumptions of large non-linear deformations in \cite{DD02} and infinitesimal displacements in \cite{Ces10}), respectively; complex configurations where optical microstructure interacts in a collaborative fashion with 
instabilities induced by geometrical constraints, such as wrinkling (modeled in \cite{CPB15}, images of the prototypes designed at NASA Langley Research Center are reported in \cite{PB16}) and actuation of soft structures made of NLCEs via thermal activation (see \cite{WMWTW15} and \cite{PKWB18} and
also supplementary videos available online).

Although planar and radial configurations such as the one in Figure \ref{1901112220} to the best of our knowledge have not been observed in NLCEs, they are common in liquid crystals where they are associated with topological defects (see \cite{Vir94, YAFO15}). In nematic elastomers instead, although radial -- and even spiral-like -- director configurations have been induced in membranes,
they are typically accompanied by large 3D stretches and out-of plane director re-arrangements \cite{GSK+18, KMG+18}. 
We hope the theoretical results and constructions described in this article will inspire further experimental investigation of complex microstructure morphology in NLCEs.

\subsection{Main results}

The objective of this note is the unified study of a specific class of planar solutions to differential inclusions of the forms \eqref{eq:ex_stress_free}, \eqref{eq:K_liq} and \eqref{1902010046} at a fixed temperature $\theta>0$ and for planar geometries. These type of deformations had been introduced by Conti \cite{C}, see also \cite{Pompe} and the constructions in \cite{K1}. Deformations and materials allowing for this class of constructions are of particular interest due to various reasons. Indeed, from a physical point of view
\begin{itemize}
\item materials which allow for these deformations are candidates for low hysteresis materials;
\item the constructions are motivated by specific deformation fields observed experimentally (e.g. tripole star deformations).
\end{itemize}
Moreover, in addition to these physical sources of interest, also from a purely mathematical point of view these constructions are relevant, as
\begin{itemize}
\item they can be used as building block constructions in convex integration schemes,
\item the deformations occur both in the theory of elastic crystals and also in models for nematic liquid crystal elastomers. This allows for a unified mathematical discussion of both systems.
\end{itemize}
Let us comment on some of these aspects in more detail:

On the one hand, these specific solutions are of particular interest as not only their bulk energy vanishes, but also their surface energy, measured for instance in terms of the $BV$ norm of $\nabla u$ is finite (see Section \ref{sec:stars} for some remarks on energetics). As a consequence, materials which exhibit such structures are candidates for materials with \emph{low hysteresis} as nucleation has low energy barriers (both in purely bulk but also in bulk and surface energy models) \cite{CKZ17}, see also \cite{zjm:09} for more information on hysteresis in shape-memory alloys.

On the other hand, in addition to their relevance in the analysis of hysteresis, microstructures of this type are often used as key building blocks in the construction of convex integration solutions. As the energies in \eqref{eq:energy}, \eqref{1902051753} and \eqref{1901311846} are typically highly non-quasi-convex and thus in particular not immediately amenable to the direct method in the calculus of variations, it came as a surprise, when it was discovered (first in the context of shape-memory alloys, later -- see \cite{ADMDS15} -- also in the context of nematic liquid crystal elastomers) that for a large set of possible boundary conditions exact solutions to \eqref{eq:ex_stress_free}, \eqref{eq:K_liq} and \eqref{1902010046} exist (see \cite{MS,K1} and the references therein). 
These solutions are obtained through iterative procedures in which oscillatory building blocks successively improve the construction, pushing it to become a solution to \eqref{eq:ex_stress_free} in the limit. For more information on this we refer to \cite{DaM12, D, MS1, CDK, K1, KMS03, ADMDS15, R16a} and the references therein. The solutions which we discuss below are frequently used as building blocks \cite{C,K1} in this context; they can even be applied in the \emph{quantitative} analysis of convex integration solutions \cite{RZZ16, RZZ17, RTZ18}.

Motivated by these considerations, in this note we seek to: 
\begin{itemize}
\item Extend the necessary and sufficient conditions for the presence of planar Conti type constructions derived in \cite{CKZ17} to arbitrary $n\in \N$. In particular, we reproduce the experimentally observed tripole star structures (both in the geometrically linearised and the non-linear theories). As a consequence, we also underline the observation from \cite{KK91} that within a geometrically non-linear theory tripole stars in shape-memory alloys are not exactly stress-free. An interesting aspect from the modeling point of view, these microstructures are planar and therefore fully covered by the 2D analysis we develop. However, in contrast to the experimentalists' point of view who interpret these microstructures as \emph{disclinations}, we offer an interpretation of these configurations as \emph{stressed microstructures} with low (elastic and surface) energy (see the discussion in Section \ref{sec:stars}).
\item Pass to the limit $n\rightarrow \infty$. Physically this limit corresponds to the passage from solid crystals to nematic elastomers. Our results can hence also be read as predictions on microstructure formation for experiments on nematic elastomers in highly symmetric domains.
\end{itemize}

To this end, we rely on the geometrically non-linear constructions from \cite{CKZ17} which we investigate for a general $n$-well problem before passing to the limit $n\rightarrow \infty$. As in \cite{CKZ17} we obtain necessary and sufficient conditions on the wells in order for the corresponding Conti constructions to exist. We remark that in the context of the two-dimensional, geometrically linearised hexagonal-to-rhombic phase transformation by completely different methods (relying on the characterisation of homogeneous deformations involving four variants of martensite) necessary conditions had been derived in an OxPDE summer project by Stuart Patching \cite{P14}. The sufficiency of the necessary conditions had previously been established in \cite{CPL14} in the geometrically linearised hexagonal-to-rhombic phase transformation. The results in \cite{CPL14} are also complemented by numerical simulations of possible solutions, which match the experimentally observed solutions in \cite{MA80, MA80a, KK91} well.

As side results of our discussion of the geometrically non-linear $n$-well model, we also show that by linearisation one directly obtains some solutions to the geometrically linearised problem and that for odd $n$ this requires fewer wells (only $n$) than in the geometrically non-linear setting (where a single layer construction already requires $2n$ wells). 
In addition to this, we report on attempts at producing analogous constructions in the geometrically non-linear three-dimensional $n$-well problem, in which we had originally also sought to construct Conti type solutions. Here however, we only obtained negative results showing that the two-dimensional situation allows for significantly more flexibility than its three-dimensional analogue.

\subsection{Organisation of the article}
The remainder of the article is organised as follows: In the main part of the article (Sections \ref{sec:nonlinear_n_well} and \ref{sec:lin_main}), we discuss the two-dimensional geometrically non-linear and linearised settings: In Section \ref{sec:nonlinear_n_well} we discuss the geometrically non-linear $n$-well construction in a regular $n$-gon, generalising the ideas from \cite{CKZ17}. Here we
discuss necessary and sufficient conditions (see Sections \ref{sec:necessary} and \ref{sec:single_layer}) for single layers of Conti-type structures. We then discuss their iterability, which turns out to be rather delicate in the geometrically non-linear situation and gives rise to the presence of stresses in geometrically non-linear tripole star structures (see Section \ref{sec:iteration_layer}). We then also pass to the limit $n\rightarrow \infty$ (Section \ref{sec:infty}) and discuss consequences for models of nematic liquid crystal elastomers (Section \ref{sec:solids_liquids}). In Section \ref{sec:lin} we linearise these constructions and observe that for the geometrically linearised constructions fewer variants of martensite are needed than for the geometrically non-linear ones. In particular, tripole star deformations are exactly stress-free in the geometrically linear theory (see Section \ref{sec:finite_n}). In Section \ref{sec:lin_infty}, we then also address constructions for geometrically linear models of liquid crystals. In this context we relate the special boundary conditions which had been chosen in \cite{ADMDS15} (see Section \ref{sec:lin_infty}) to our differential inclusions.
Finally, in Section \ref{sec:3D} we comment on our  (negative) results on analogous three-dimensional constructions.

\section{The non-linear construction in a regular $n$-gon}
\label{sec:nonlinear_n_well}

In this section we present the necessary and sufficient conditions for geometrically non-linear Conti constructions in a setting involving $n$ wells. 

Here we pursue the following objectives. We seek to:
\begin{itemize}
\item[(i)] Provide \emph{necessary} and \emph{sufficient} conditions for a geometrically non-linear, ``single layer'' Conti construction associated with a phase transformation for general $n\in \N$ (see Sections \ref{sec:necessary} and \ref{sec:single_layer}). This builds on and generalises the argument from \cite{CKZ17}.
\item[(ii)] Discuss the \emph{iterability} of the single layer constructions from (i). As a main observation, we show that, in general, this is not possible without allowing for a larger set of wells (see Proposition \ref{prop:non_iterable} in Section \ref{sec:iteration_layer}). Physically, the iteration of the construction in (i) reproduces for instance the tripole star deformations which are observed experimentally, see Section \ref{sec:stars}. We offer an interpretation of these in terms of slightly stressed low energy states (instead of viewing them as disclinations as in the experimental literature).
\item[(iii)] Pass to the \emph{limit} $n\rightarrow \infty$. This corresponds to the nematic liquid crystal elastomer limit (see Proposition \ref{prop:n_infty} in Section \ref{sec:infty}).
\end{itemize}

\subsection{Set-up and precise problem formulation}
\label{sec:problem}

In the sequel, we seek to identify necessary and sufficient conditions for the existence of a specific low energy nucleation mechanism associated with highly symmetric deformations. Let us describe this informally. We are interested in studying a class of deformations which satisfy the following properties:
\begin{itemize}
\item Outside of a large regular $n$-gon $\Omega_n^E$ and inside of a small regular $n$-gon $\Omega_n^I$, both with the same barycenter, the deformation is equal to a rotation (without loss of generality, we may assume it to be equal to the identity in the outside domain and a non-trivial rotation in the inner $n$-gon). Without loss of generality, we further assume that the barycenter of both $n$-gons is the origin.
\item In the set $\Omega_n^E\setminus \Omega_n^I$ the deformation is piecewise constant on a set of triangles formed by connecting the vertices of $\Omega_n^E$ and $\Omega_n^I$ (see Figures \ref{fig:polygons} and \ref{fig:2}).
\item We require that the deformation is associated with a phase transformation, i.e. that the piecewise constant deformation gradients in $\Omega_n^E\setminus \Omega_n^I$ only attain values in the set $\bigcup\limits_{j=1}^{m} SO(2)U_j$, where $U_j = P U_1 P^T$ for some $P \in \mathcal{P}_n$ and $U_1 \in \R^{2\times 2}_{sym}$ and where $\mathcal{P}_n\subset O(2)$ denotes the point group of the transformation at hand.
\item We require that the deformation is volume preserving.
\end{itemize}
Having fixed the outer $n$-gon $\Omega_n^E$, the condition on the volume preservation together with the fact that the deformation gradient has a constant determinant in $\Omega_n^E \setminus \Omega_n^I$ implies that after fixing a single vertex with coordinates $(x_1,x_2)$ of the inner $n$-gon, the deformation $u$ is already determined. Indeed, in order to ensure the volume preservation constraint, under the deformation $u$ the vertex has to be mapped to the deformed vertex $R_{\ast}(x_1,x_2)$, where $R_{\ast}$ is a rotation by $\frac{2\pi}{n}(1-2\alpha)$ and $\alpha\in(0,1)$ denotes the angle of rotation of the inner $n$-gon with respect to the outer one (see Figures \ref{fig:flip} and \ref{fig:2}). Hence, in principle, the deformation $u$ is determined by two parameters (e.g. the coordinates $(x_1,x_2)\in \R^2$). As in \cite{CKZ17}, we thus consider the two-parameter family of deformations given by

\beq
\label{DefU}
(\bar a,\psi) \mapsto Id + \bar a \begin{pmatrix} \sin(\psi) \\ \cos(\psi)\end{pmatrix} \otimes \begin{pmatrix} - \cos(\psi) \\ \sin(\psi) \end{pmatrix},
\eeq
where $Id \in \R^{2\times 2}$ denotes the identity matrix and $\psi \in (0,2\pi]$,
which is motivated by investigating the described  deformations with austenite boundary conditions corresponding to low hysteresis deformations (in fact to allow for simpler computations, in the sequel, we will often replace the identity boundary conditions by boundary conditions given by a fixed rotation). As in \cite{CKZ17} we will prove that the requirement that the deformation is associated with a phase transformation reduces the degrees of freedom from two parameters to a single parameter.\\

After this informal discussion of our problem, we present the formal problem set-up.
We start by introducing the following definitions. We remark that, here and below, for any set $A\subset \R^2$ we denote by $A^{co}$ its convex hull and by $\inte A$ its interior. Furthermore, by $\{e_1,e_2\}$ we denote an orthonormal basis of $\R^2.$

\begin{defi}
\label{DefOmega}
Let $n\in\mathbb{N}, n\geq 2$, $\alpha\in(0,1]$ and $r_I,r_E\in (0,+\infty)$ with $r_I<r_E.$ We say that $\Omega_n\subset\R^2$ is an \emph{$n-$gon configuration} if, given
\begin{align*}
E_i &= r_E\cos\Bigl(\frac{2\pi}{n}(i-1) \Bigr) e_1 + r_E\sin\Bigl(\frac{2\pi}{n}(i-1) \Bigr)  e_2,\qquad i=1,\dots, n,\\
I_i &= r_I\cos\Bigl(\frac{2\pi}{n}(i-1) +\alpha\frac{2\pi}{n} \Bigr) e_1 + r_I\sin\Bigl(\frac{2\pi}{n}(i-1) +\alpha\frac{2\pi}{n}\Bigr) e_2,\qquad i=1,\dots, n,
\end{align*}
and 
$$
\Omega_n^E = \{E_1,\dots,E_n\}^{co},\qquad \Omega_n^I= \{I_1,\dots,I_n\}^{co},
$$
we have $\Omega_n := \inte\bigl(\Omega_n^E\setminus \Omega_n^I\bigr)$. 
\end{defi}

\begin{figure}[t]
  \centering
  \includegraphics[width=0.4\linewidth,page=1]{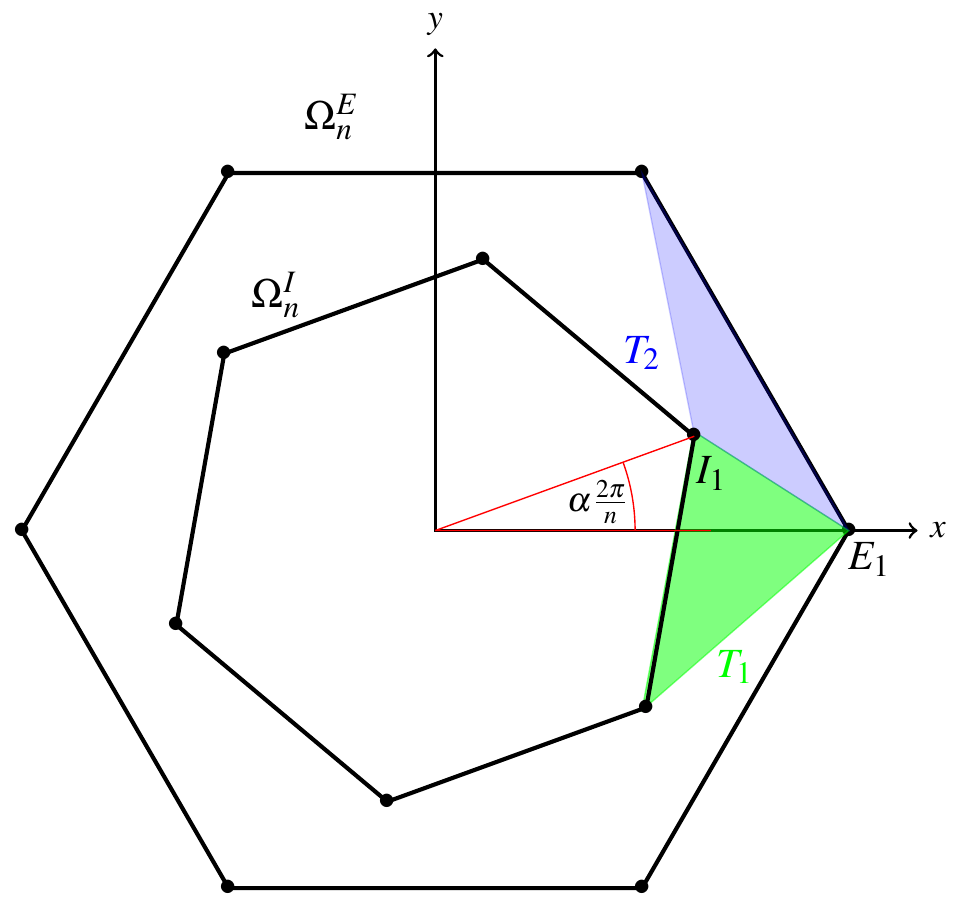}
  \caption{The inner and outer polygon are rotated by an angle
    $\frac{2\pi}{n}\alpha$ with respect to each other.}
  \label{fig:polygons}
\end{figure}

Given three points $p_1,p_2,p_3\in\R^2$, we denote by $\widehat{p_1p_2p_3}$ 
the open triangle $\widehat{p_1p_2p_3} = \inte\{p_1,p_2,p_3\}^{co},$ and by $\overline{p_1p_2}$ the vector $\overline{p_1p_2} = p_2 - p_1$. Finally, we denote by $e_{ij}$ the unit vector $e_{ij} = \frac{\overline{E_jI_i}}{|\overline{E_jI_i}|}$. Now, given an $n-$gon configuration $\Omega_n$ as in Definition \ref{DefOmega}, we define the internal triangles $T_i$ as
\[
T_i=
\begin{cases}
\widehat{E_{\frac{i+1}{2}}I_{\frac{i-1}{2}} I_{\frac{i+1}{2}}},\qquad&\text{if $i$ odd},\\[5pt]
\widehat{E_{\frac{i}{2}}I_{\frac{i}{2}}E_{1+\frac{i}{2}}},\qquad&\text{if $i$ even},
\end{cases}
\]
where we use the convention that $E_{n+1}=E_1$ and $I_0 = I_n$.\\

With this notation in hand, we now consider the following problem:\\
\paragraph{\textbf{Problem}:} Find $u\in W_{loc}^{1,\infty}(\R^2;\R^2)$ such that
\begin{enumerate}[label=(\roman*)]
\item \label{H1} for every $i=1,\dots,2n$, $u$ is affine on $T_i$;
\item \label{H2} $u=\Id$ on $\R^2\setminus \Omega_n^E$, where $\Id$ denotes the identity map;
\item \label{H3} $\nabla u(x) \in \bigcup_{P\in\mathcal{P}_n} SO(2) P U P^T$ for
  some $U\in\R^{2\times 2}$ and  for almost every $x\in \Omega_n^E \setminus \Omega_n^I$, where $\mathcal{P}_n \subset O(2)$ denotes the discrete (to be determined) symmetry group of our problem;
\item \label{H4}$u (x) = R_* x$ in $\Omega_n^I$, for some $R_*\in SO(2)$ of angle $\rho_n = \frac{2\pi}{n}\bigl(1-2\alpha \bigr)$. As a consequence, $R_{\ast} I_n = \Bigl(\cos\Bigl(\alpha\frac{2\pi}{n}\Bigr) e_1 - \sin\Bigl(\alpha\frac{2\pi}{n}\Bigr) e_2  \Bigr).$
\end{enumerate}

We remark that these conditions formalise the requirements of a ``Conti construction'' with symmetry. These are piecewise affine deformations (as stated in \ref{H1}) with specific linear boundary conditions \ref{H2} such that all involved deformation gradients are symmetry related as in \ref{H3}. The condition \ref{H4} is a consequence of the desired symmetry of the $n$-gon configuration in conjunction with the prescription of the identity boundary data in \ref{H2}. Indeed, by requiring austenite boundary data, we infer that $\det(\nabla u)=1$ on each triangle $T_i$, which can only be the case if $R_{\ast}$ is of the described form.
It corresponds to a ``flipping" of the coordinates of $I_n$, see Figures \ref{fig:flip} and \ref{fig:2}.

In the sequel, it will turn out that the symmetry group $\mathcal{P}_n = \mathcal{R}_1^n \cup \mathcal{R}_2^n$ associated with our problem is
a conjugated version of the symmetry group of a regular $n$-gon, called the
dihedral group. More precisely, the standard dihedral group of a regular $n$-gon is given by
\begin{align}
\label{eq:Pn}
  \mathcal{\hat{P}}_n:= \hat{\mathcal{R}}_1^{n} \cup \hat{\mathcal{R}}_2^{n}.
\end{align}
Here $\hat{\mathcal{R}}_1^{n}$ is the collection of all rotations leaving the $n$-gon invariant, i.e.
\begin{align*}
\hat{\mathcal{R}}_1^{n} := \left\{ \begin{pmatrix} \cos(\varphi_j) & \sin(\varphi_j) \\ - \sin(\varphi_j) & \cos(\varphi_j) \end{pmatrix}: \varphi_j = \frac{2\pi j}{n}, \ j \in \{0,\dots,n-1\} \right\},
\end{align*}
and $\hat{\mathcal{R}}_2^{n}$ is the collection of the corresponding reflections $\hat{\mathcal{R}}_2^n:= \begin{pmatrix} 1 & 0 \\ 0 & -1 \end{pmatrix} \mathcal{R}^n_1$. In our problem, we will encounter a conjugated version of this, where
\begin{align*}
\mathcal{R}_2^{n} &:= 
\begin{pmatrix}
e_{11} & e_{11}^{\perp} 
\end{pmatrix}
\hat{\mathcal{R}}^n_2
\begin{pmatrix}
e_{11} & e_{11}^{\perp} 
\end{pmatrix}
=
{
                    \begin{pmatrix}
                      e_{11} & e_{11}^\perp
                    \end{pmatrix}
                               \begin{pmatrix}
                                 1 & 0 \\ 0 & -1
                               \end{pmatrix}
                                              \hat{\mathcal{R}}_1^n                                              
                    \begin{pmatrix}
                      e_{11} & e_{11}^\perp
                    \end{pmatrix}^T}\\
& =  \bigl(e_{11}\otimes e_{11} - e_{11}^\perp\otimes e_{11}^\perp \bigr)\hat{\mathcal{R}}_1^{n}                    .
\end{align*}
We further note that $\hat{\mathcal{R}}_1^n$ is invariant under the change of basis to
$(e_{11} \ e_{11}^\perp)$ since $SO(2)$ is commutative. Hence, the symmetry group in our problem
\begin{align}
\label{eq:symmetrygroup}
  \mathcal{P}_n := \begin{pmatrix}
                      e_{11} & e_{11}^\perp
                    \end{pmatrix}
                               \left(
\hat{\mathcal{R}}_1^n \cup  \begin{pmatrix}
                                 1 & 0 \\ 0 & -1
                               \end{pmatrix} \hat{\mathcal{R}}_1^n
                               \right)
                               \begin{pmatrix}
                      e_{11} & e_{11}^\perp
                    \end{pmatrix}^T =: \mathcal{R}_1^n\cup \mathcal{R}_2^n,
\end{align}
is given by the dihedral group (that is the symmetry group of the standard regular $n$-gon)
conjugated with a change of basis $(e_{11} \ e_{11}^\perp)$.

\begin{figure}
\includegraphics[width=0.4 \textwidth,page=2]{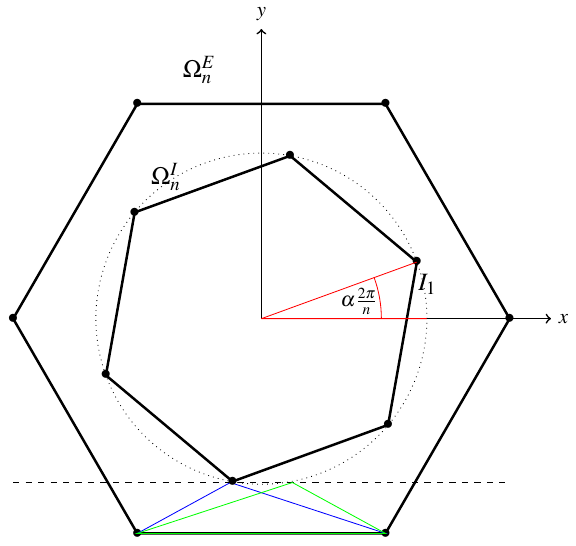}
\caption{The ``flipping'' condition formalised in \ref{H4}. In order to ensure volume preservation an outer blue triangle in the reference configuration (see the online version for the colours) is mapped to the green outer triangle in the deformed domain. Assuming the deformation to be a rotation in the inner $n$-gon $\Omega_n^I$, this is the only possible deformation that preserves the volume of the outer triangles. We also refer to Figure \ref{fig:2} for another illustration of the ``flipping'' condition.}
\label{fig:flip}
\end{figure}

\begin{rmk}
\label{Rem1}
In the sequel, we will often rely on the following commutation relations:
Given $U\in\R^{2\times2}$, and any $Q\in SO(2)$, then 
$$
\bigcup_{P\in\mathcal{P}_n} SO(2) P QU P^T =  \bigcup_{P\in\mathcal{P}_n}SO(2) P U P^T.
$$
Indeed, if $P\in\mathcal{R}_1^n$, then $QP = PQ$. If instead $P\in\mathcal{R}_2^n$, then $PQ = Q^TP$. 
\end{rmk}

\begin{figure}[t]
\begin{subfigure}{.5\textwidth}
  \centering
  \includegraphics[width=.99\linewidth,page=3]{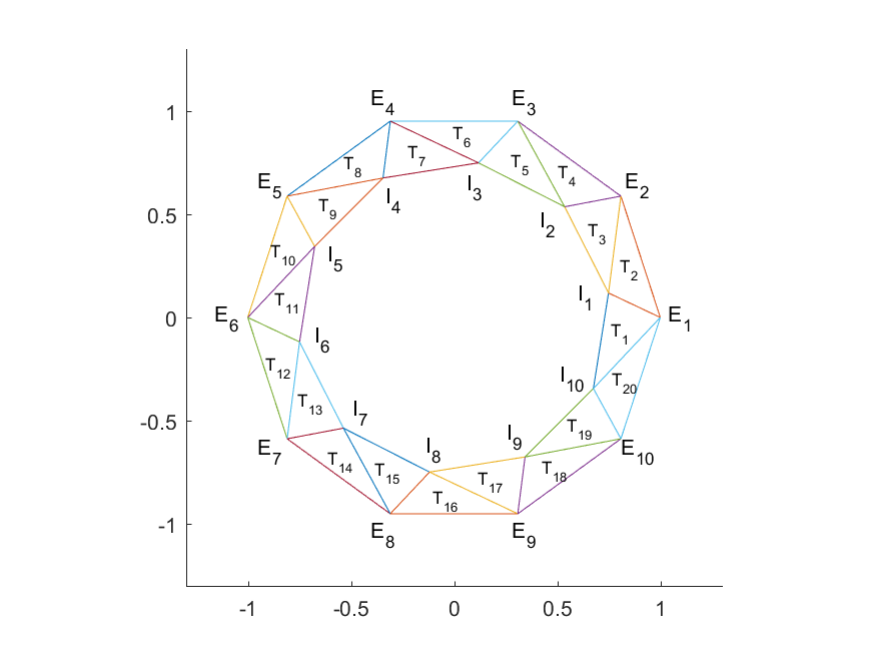}
  \caption{}
  \label{fig00001}
\end{subfigure}%
\begin{subfigure}{.5\textwidth}
  \centering
  \includegraphics[width=.99\linewidth,page=4]{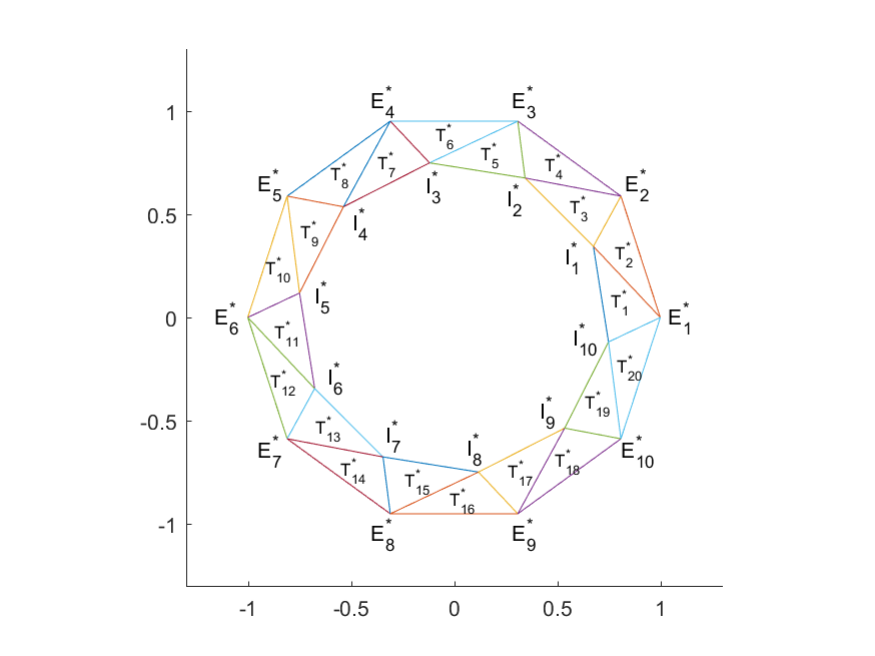}
  \caption{}
  \label{fig00002}
\end{subfigure}
\caption{Example of $n-$gon with $n=10$. On the left before the action of the map $u$, on the right after its action.  Here, we denoted by $E^*_i,I^*_i,T^*_i$ the quantities $u(E_i),u(I_i),u(T_i).$ In order to ensure volume preservation, a necessary condition is the ``flipping'' of the triangles on which $\nabla u$ is constant which is formalised in condition \ref{H4} in our problem formulation.}
\label{fig:2}
\end{figure}

In the next sections, we discuss the necessary and sufficient conditions for a solution of the described problem. Moreover, we discuss the iterability of the associated constructions and the limit $n\rightarrow \infty$.

\subsection{Necessary condition}
\label{sec:necessary}

Regarding the necessary conditions for the existence of a phase transformation associated with a Conti-construction in a regular $n$-gon, we obtain the following analogue of \cite{CKZ17}:

\begin{prop}
\label{lem:nec}
A necessary condition for the satisfaction of \ref{H1}--\ref{H4} in $\Omega_n$ is the condition that
\begin{align}
\label{eq:nec}
\phi:= \arccos \bigl(e_{11}\cdot e_{n1}\bigr) = \frac{\phi_n}{2},
\end{align}
where $\phi_n = \frac{n-2}{n} \pi$ is the interior angle at each corner of the regular $n$-gon. In particular, this entails the necessary condition
\begin{align}
\label{eq:U1}
U  = U (a)= \Bigl(a e_{11}\otimes e_{11}+\frac 1a e_{11}^\perp \otimes e_{11}^\perp + \frac{a^{-1}-a}{\tan \phi} e_{11}\otimes e_{11}^\perp\Bigr),
\end{align}
for some $a>0$, and where $e_{11}^\perp \in \mathbb{S}^1$ is such that $e_{11}^\perp\cdot e_{11}=0$, $e_{11}\times e_{11}^\perp>0$.
The associated point group $\mathcal{P}_n$ is necessarily given by the group in \eqref{eq:symmetrygroup}.
Finally, 
\beq
\label{RadRel}
\frac{r_I}{r_E} = \frac{1}{\cos\Bigl(\frac\pi n \Bigr)}\biggl(\cos\Bigl(\frac{\pi}{n}(1-2\alpha)\Bigr) - \sqrt{\sin\Bigl(\frac{\pi}{n}2\alpha\Bigr) \sin\Bigl( \frac{2\pi}{n}(1-\alpha)\Bigr)}	\biggr),
\eeq
and 
\beq
\label{EqA}
a = \sqrt{\frac{\sin\Bigl(\frac{2\pi}{n}(1-\alpha)\Bigr)}{\sin\Bigl(\frac{2\pi}{n}\alpha\Bigr)}},
\eeq
where $\alpha$ is as in Definition \ref{DefOmega}.
\end{prop}

\begin{rmk}
We notice that for each fixed $n \geq 1$ \eqref{EqA} gives a one-to-one relation between $a>0$ and $\alpha\in(0,1)$. Indeed, $a(\alpha)$ is strictly monotone and $\lim_{\alpha\to 0}a(\alpha) = +\infty$ and $a(1)=0$. Moreover, we note that as expected from the conditions \ref{H1}-\ref{H4}, we have $a(\frac{1}{2})=1$. 
\end{rmk}

\begin{figure}
\includegraphics[scale=0.7,page=5]{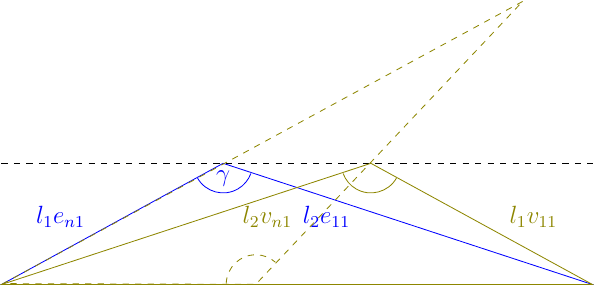}
\caption{An illustration of the eigenvalue condition identified in the proof of
  Proposition \ref{lem:nec}. We observe that $e_{n1}\cdot
  e_{11}=v_{n1}\cdot v_{11}=\cos(\gamma)$ and that there exists $\hat{R} \in SO(2)$ such that
  $\hat{R} e_{n1}=v_{n1}$ and $\hat{R} e_{11}=v_{11}$. After a rotation the image triangle (in green) can be rotated onto the reference triangle (in blue). This yields the triangle whose sides are depicted with the green dashed lines. Formalising this leads to the proof of \eqref{Eq0Nec}.}
\label{fig:eigenv}
\end{figure}

\begin{proof}

The argument to prove \eqref{eq:nec}--\eqref{eq:U1} follows along the lines of \cite{CKZ17}, which we present for self-containedness.

Let us start by noticing that, since we assume that for each $i=1,\dots,n$ the deformation $u$ is affine in $T_i$, then $\nabla u|_{T_i} = F_i$, for some $F_i\in\R^{2\times 2}$.

As in \cite{CKZ17}, we now first identify suitable eigenvectors and eigenvalues in the construction:
Let $l_1 = |\overline{I_{i-1}E_i}|$ and $l_2 = |\overline{I_{i}E_i}|$ for $i=1,\dots,n$ (where we remark that by symmetry these lengths are independent of $i\in\{1,\dots,n\}$, c.f. Definition \ref{DefOmega}). By \ref{H4}, i.e. by the ``flipping'' of the internal points of the outer triangles, there exist $v_{n1},v_{11}\in\mathbb{S}^1$ such that
$$
F_1 \overline{I_nE_1} = F_1 l_{1} e_{n1} = l_2 v_{n1},\qquad
F_1 \overline{I_1E_1} = F_1 l_{2} e_{11} = l_1 v_{11},
$$
and a rotation $\hat{R} \in SO(2)$ such that
  \begin{align*}
    \hat{R} e_{n1} = v_{n1}, \qquad   
     \hat{R} e_{11} = v_{11},
  \end{align*}
 see also Figures \ref{fig:flip}-\ref{fig:eigenv}.
Therefore,
$$
e_{n1}\cdot e_{11} = \hat{R} e_{n1}\cdot \hat{R} e_{11} = v_{n1}\cdot v_{11},
$$
and, setting $R := \hat{R}^T$,
\beq
\label{Eq0Nec}
RF_1e_{n1} = \frac 1a e_{n1},\qquad RF_1e_{11} = a e_{11},
\eeq
where $a:=\frac{l_1}{l_2}.$ Since $u$ is continuous, it must hold that
\beq
\label{Eq1Nec}
RF_2 e_{11} = a e_{11}.
\eeq
Furthermore, repeating the above arguments based on the  condition \ref{H4} (which simply follows by symmetry as $T_3$ is a rotation of $T_1$ by $\frac{2\pi}{n}$),
we have that
$$
RF_3 e_{12} = \frac 1a e_{12},\qquad RF_3e_{22} = a e_{22}. 
$$
The continuity of $u$ then again implies that
\beq
\label{Eq2Nec}
RF_2e_{12}= \frac 1a e_{12}.
\eeq
Let us suppose now that there exists $P,Q\in O(2)$ with $\det P\det Q = 1$ such
that $R F_1 = P R F_2 Q$. 
Then,
\begin{align*}
&a = a|P^Te_{11}|= |P^T R F_1 e_{11}|= | R F_2 Q e_{11}|,\\
&\frac 1a = \frac 1a |P^Te_{n1}|= |P^T R F_1 e_{n1}|= | R F_2 Q e_{n1}|.
\end{align*}
But, by \eqref{Eq1Nec}--\eqref{Eq2Nec}, $a$ and $\frac1a$ are simple eigenvalues of $RF_2$ and thus $Qe_{11} = \pm e_{11}$ and $Qe_{n1} = \pm e_{12}.$
Hence, 
$$
\cos \phi := e_{11}\cdot e_{n1} = Q e_{11}\cdot Q e_{n1} = \pm e_{11}\cdot e_{12} = \pm \cos\Bigl(	\frac{2\pi}{n}+\phi\Bigr).
$$
The only solution $\phi$ to this equation in the interval $\Bigl(0, \frac{\pi(n-2)}{n}\Bigr)$ (where the construction is respected) is $\phi = \frac{\pi(n-2)}{2n}$. Furthermore, defining $U:=RF_1$, by \eqref{Eq0Nec}, it must be of the form
\begin{align}
\label{eq:U11}
\begin{split}
U& = RF_1 
= a e_{11}\otimes e_{11}+\frac 1a e_{11}^\perp \otimes e_{11}^\perp + \frac{a^{-1}-a}{\tan \phi} e_{11}\otimes e_{11}^\perp,
\end{split}
\end{align}
where $e_{11}^\perp$ is such that $e_{11}\times e_{11}^\perp > 0$ and where we exploited the fact that $e_{n1} = \cos(\phi) e_{11} + \sin(\phi) e_{11}^\perp.$ 
This concludes the argument for \eqref{eq:nec} and \eqref{eq:U1}. 

The statement on the symmetry group then follows from the symmetry of the domains.
\\

We next discuss the derivation of the identities \eqref{RadRel} and \eqref{EqA}.
In order to prove \eqref{RadRel}, we first notice that on the one hand,
\begin{align}
\label{l1}
l_1 &= |E_1 - I_n | = \sqrt{r_E^2+r_I^2 -2r_Er_I \cos\Bigl(\frac{2\pi}{n} (1-\alpha) \Bigr)},\\ 
\label{l2}
l_2 &= |E_1 - I_1 |= \sqrt{r_E^2+r_I^2 -2r_Er_I \cos\Bigl(\frac{2\pi}{n}\alpha \Bigr)} .
\end{align}
On the other hand,
\beq
\label{LongEq}
\begin{split}
&l_1l_2 \cos \phi = \overline{E_1I_n}\cdot \overline{E_1I_1} = 
\bigl(I_n - E_1	\bigr)\cdot \bigl(I_1 - E_1	\bigr)\\
& = 
 \begin{pmatrix}
      r_I\cos(\frac{2\pi}{n}(\alpha-1))-r_E \\ r_I \sin(\frac{2\pi}{n}(\alpha-1))
    \end{pmatrix}
    \cdot
     \begin{pmatrix}
      r_I\cos(\frac{2\pi}{n}\alpha)-r_E \\ r_I \sin(\frac{2\pi}{n}\alpha)
    \end{pmatrix}\\
    &= r_E^2 + r_I^2\left(\cos(\frac{2\pi}{n}\alpha)\cos(\frac{2\pi}{n}(\alpha-1)) +\sin(\frac{2\pi}{n}\alpha)\sin(\frac{2\pi}{n}(\alpha-1))  \right) \\& \quad - r_Ir_E \left(\cos(\frac{2\pi}{n}(\alpha-1))+\cos(\frac{2\pi}{n}\alpha)\right) \\
    &= r_E^2+r_I^2 \cos(\frac{2\pi}{n}) - r_Ir_E \left(\cos(\frac{2\pi}{n}(\alpha-1))+\cos(\frac{2\pi}{n}\alpha)\right)\\
& = r_E^2 + r_I^2\cos\Bigl(\frac{2\pi}{n}\Bigr)-2 r_I r_E \cos \Bigl(\frac{\pi}{n}(2\alpha-1)\Bigr)\cos \Bigl(\frac{\pi}{n}\Bigr).
\end{split}
\eeq
Here, in the last step, we have used the trigonometric identity
\begin{align*}
\cos(\psi_1) + \cos(\psi_2) = 2\cos\left(\frac{1}{2}(\psi_1+ \psi_2) \right) \cos\left(\frac{1}{2}(\psi_1 - \psi_2)\right).
\end{align*}
Taking the square of \eqref{LongEq} and exploiting \eqref{l1}--\eqref{l2} gives a fourth order equation in $x=\frac{r_I}{r_E}$. Out of the four solutions of this equation, the only satisfying \eqref{LongEq} and such that $x\in(0,1)$ provided $\alpha\in(0,1)$ is given by \eqref{RadRel}. We refer the reader to Appendix \ref{LongEqSolve} for the details. Furthermore, using that $a = \frac{l_1}{l_2}$ and exploiting \eqref{RadRel} in \eqref{l1}--\eqref{l2} we deduce \eqref{EqA}.  
\end{proof}

\subsection{Sufficient conditions}
\label{sec:single_layer}

We discuss the sufficiency of the necessary condition by explicitly constructing a ``single layer'' Conti construction, i.e. by constructing a deformation as illustrated in Figure \ref{fig:2}.

\begin{prop}
\label{prop:suff}
Let $a>0,$ $\alpha\in(0,1)$, $r_E,r_I>0$ satisfy \eqref{RadRel}--\eqref{EqA}. Let also $U:=U(a)$ be as in \eqref{eq:U1}. Then there exists a deformation $u$ such that \ref{H1}--\ref{H4} are satisfied.
\end{prop}

\begin{proof}
We argue in three steps. Here we first construct a tensor field $F$ in $\Omega^E_n \setminus \Omega^I_n$, and then in $\R^2 \setminus \Omega_n^E$ and $\Omega_n^I$. Finally, we discuss the overall compatibility, showing that $F= \nabla u$ for some piecewise constant deformation $u:\R^2 \rightarrow \R^2$.\\

\emph{Step 1: Deformation in the region $\Omega^E_n \setminus \Omega_n^I$}. We first construct a piecewise constant tensor field $F: \Omega_n^E \setminus \Omega_n^I \rightarrow  \R^{2\times 2}$.
Let us start by setting $F = U$ in $T_1$, and $F = P_0 U P_0$ in $T_2$, where $P_0 = e_{11}\otimes e_{11} - e_{11}^\perp\otimes e_{11}^\perp \in O(2).$ We have that $U$ and $P_0 U P_0$ are compatible across the line parallel to $e_{11}.$ Indeed, 
\beq
\label{U1U2R1}
U-P_0 U P_0 = 2\frac{a^{-1}-a}{\tan \phi} e_{11}\otimes e_{11}^\perp,\qquad	\bigl(U-P_0 U P_0\bigr) e_{11} = 0.
\eeq
Then, we define $F$ as follows:
\beq
\label{defDu}
F = 
\begin{cases}
Q_{\frac{j-1}{2}} U Q_{\frac{j-1}{2}}^T,\qquad&\text{in $T_j$ if $j$ odd},\\
Q_{\frac{j-2}{2}} P_0 U P_0 Q_{\frac{j-2}{2}}^T,\qquad&\text{in $T_j$ if $j$ even}.
\end{cases}
\eeq
Here $Q_{\varphi} := Q(\frac{2\pi}{n} \varphi)$ with $Q(\varphi):= \begin{pmatrix} \cos(\varphi) & -\sin(\varphi)\\ \sin(\varphi) & \cos(\varphi)\end{pmatrix}\in SO(2)$.
Furthermore, we have
\begin{align}
\label{eq:lin_comb}
\begin{split}
e_{n1} &= \cos (\phi) e_{11} + \sin (\phi) e_{11}^\perp = \sin\Bigl(\frac{\pi}{n}\Bigr) e_{11} + \cos\Bigl(\frac{\pi}{n}\Bigr) e_{11}^\perp,\\
e_{12} &= \cos \Bigl(\phi+\frac{2\pi}{n}\Bigr) e_{11} + \sin\Bigl(\phi+\frac{2\pi}{n}\Bigr) e_{11}^\perp = -\sin\Bigl(\frac{\pi}{n}\Bigr) e_{11} + \cos\Bigl(\frac{\pi}{n}\Bigr) e_{11}^\perp,
\end{split}
\end{align}
so that $P_0e_{12}=-e_{n1}$. This yields,
\begin{align*}
F|_{T_2} e_{12} 
& \stackrel{\eqref{defDu}}{=} P_0 U P_0 e_{12} 
= -P_0 U e_{n1} 
\stackrel{\eqref{eq:U1}}{=} - \frac{1}{a} P_0 e_{n1}
 = \frac{1}{a} e_{12} \\
&= Q_1 U e_{n1} = Q_1 U Q_1^{T} e_{12} 
\stackrel{\eqref{defDu}}{=} F|_{T_3} e_{12},
\end{align*}
and hence
\beq
\label{U1U2R1bis}
F|_{T_2} - F|_{T_3} = c\otimes e_{12},
\eeq
for some $c\in\R^2$. Now, using that $e_{i,i+1} = Q_1 e_{i-1,i}$ and that $e_{i+1,i+1} = Q_1 e_{i,i}$, by \eqref{U1U2R1}--\eqref{U1U2R1bis}, we obtain
\begin{align}
\label{RkOdd}
\bigl(F|_{T_i} - F|_{T_{i+1}}\bigl)e_{ii} &= 0,\qquad\text{if $i$ odd},\\
\label{RkEven}
\bigl(F|_{T_i} - F|_{T_{i+1}}\bigl)e_{i,i+1} &= 0,\qquad\text{if $i$ even},
\end{align}
again using the convention that $n+1 = 1$ and $0 = n$.\\

\emph{Step 2: Construction of the deformation in $\R^2 \setminus \Omega_n^E$ and $\Omega_n^I$.} We next extend $F$ to be defined also in $\Omega_n^I$ and in $\R^2 \setminus \Omega_n^E$.
By construction (and in particular by the condition \ref{H4} which just corresponded to the ``flipping''/ ``rotation'' of the inner points), we have that $U\overline{I_nI_1} = R_I\overline{I_nI_1}$ for some $R_I\in SO(2)$. 
 Therefore, $U$ and $R_I$ are compatible across the line parallel to $\overline{I_nI_1},$ that is
$$
U - R_I = b\otimes \overline{I_nI_1}^\perp,
$$
for some $b\in\R^2$. As a consequence,
\beq
\label{RkI}
F|_{T_{2i-1}} = Q_{i-1} U Q_{i-1}^T = R_I + Q_{i-1}b\otimes \overline{I_{i-1}I_i}^\perp,\qquad i=1,\dots,n.
\eeq
We set $F|_{\Omega_n^I}:= R_I$.

We claim that similarly it is possible to deduce the existence of $R_E\in SO(2)$ and $v\in\R^2$ such that 
\beq
\label{RkE}
F|_{T_{2i}}  = R_E + Q_{i-1} v\otimes \overline{E_{i}E_{i+1}}^\perp,\qquad i=1,\dots,n.
\eeq
To infer this, we observe the following:
On the one hand, using the projection of $e_{12}$ onto the basis $\{e_{11},e_{11}^\perp\}$ (see \eqref{eq:lin_comb}) we obtain
\begin{equation}
\label{LengthE1E2}
|\overline{E_iE_{i+1}}|^2 = |l_2e_{11}-l_1e_{12}|^2 = l_2^2|e_{11}-a e_{12}|^2 = l_2^2 \Bigl(1 + a^2 + 2a\sin\left(\frac\pi n\right) \Bigr).
\end{equation}
On the other hand, using \eqref{eq:lin_comb} again, we have
\begin{equation}
\label{LengthE1E2bis}
\begin{split}
|P_0 U P_0 (l_2e_{11}-l_1e_{12})|^2 
&= l_2^2\Bigl|\Bigl(a + \sin\left(\frac\pi n \right)\Bigr)e_{11}- \cos\left( \frac\pi n \right) e_{11}^\perp\Bigr|^2 \\
&= l_2^2 \Bigl(1 + a^2 + 2a\sin \left(\frac\pi n \right) \Bigr).
\end{split}
\end{equation}
Combining both observations, we deduce the claim in \eqref{RkE} and define $F|_{\R^2 \setminus \Omega_n^E}:= R_E$.\\

\emph{Step 3: Overall compatibility and conclusion.}
Since the constructed tensor field $F:\R^2 \rightarrow \R^{2\times 2}$ is piecewise constant and \eqref{RkOdd}--\eqref{RkE} hold, we have that $\nabla\times F = 0$. Therefore, the fact that $\R^2$ is simply connected and \cite[Thm. 2.9]{GiraultRaviart}, imply the existence of a deformation $u\in W_{loc}^{1,\infty}(\R^2;\R^2)$ such that $F= \nabla u$ and such that $R_E^T u$ satisfies the conditions \ref{H1}--\ref{H4}. 
\end{proof}

\subsection{Iteration of layers}
\label{sec:iteration_layer}

In the sequel, motivated by experimentally observed tripole star structures (see Section \ref{sec:stars}) and by the passage $n\rightarrow \infty$ (see Section \ref{sec:infty}), we seek to iterate the construction from Section \ref{sec:single_layer} (as illustrated in Figure \ref{fig:iterate}) leading to several nested ``onion ring layers'' of the described deformations.

\begin{figure}
\begin{subfigure}{.5\textwidth}
  \centering
  \includegraphics[width=.99\linewidth,page=6]{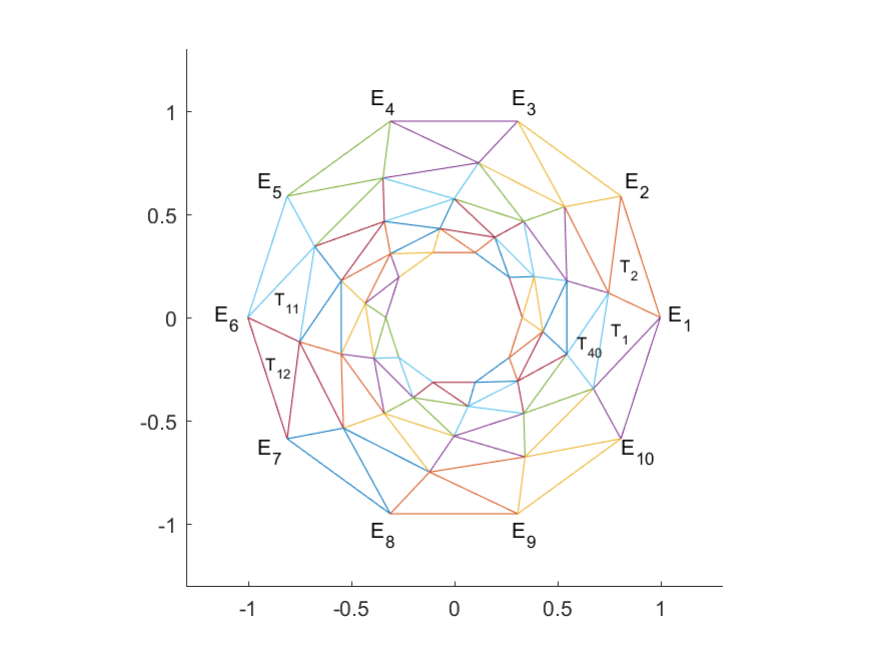}
  \caption{}
  \label{fig00001}
\end{subfigure}%
\begin{subfigure}{.5\textwidth}
  \centering
  \includegraphics[width=.99\linewidth,page=7]{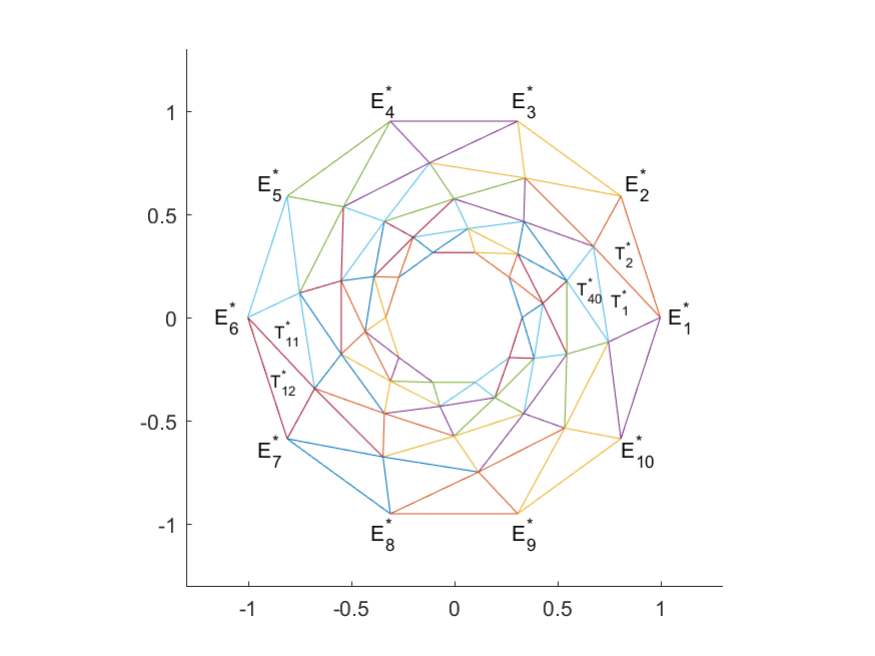}
  \caption{}
  \label{fig00002}
\end{subfigure}
\caption{Nested $n-$gons with $n=10$. On the left and on the right, respectively before and after the action of $u$. Here, we denoted by $E^*_i,T^*_i$ the quantities $u(E_i),u(T_i).$ Corollary \ref{cor:iterate} below states that $\nabla u$ is the same in $T_1,T_{11}$ and in $T_{40}$, where we denoted by $T_{40}$ the triangle $T_{40}=r_IQ_\alpha T_{20} = r_IQ_\alpha Q_{-1} T_{2}$.}  
\label{fig:iterate}
\end{figure}

To this end, we now 
\begin{itemize}
\item fix $\alpha>0$, 
\item set for a matter of simplicity $r_E=1$,
\item and take $r_I$ satisfying \eqref{RadRel}. 
\end{itemize}
Let also $u$ satisfying \ref{H1}--\ref{H4} be given by Proposition \ref{prop:suff}. 
Without loss of generality, below we consider $v_n = R_E u_n$, where as in the proof of Proposition \ref{prop:suff}, $R_E$ is such that $\nabla v_n|_{T_1} = U$, and $U$ is as in \eqref{eq:U1} (cf. proof of Proposition \ref{prop:suff}). 
Thus, let $v_n$ defined by
\beq
\label{eq:vn}
v_n(x) = \sum_{k=0}^{N_n} r_I^k R_E R_*^kQ_{\alpha}^k u(r_I^{-k}Q_{-\alpha}^kx)\chi_{\Sigma_k}(x) + R_E x\chi_{\R^2\setminus\Sigma_0^{co}}(x) + R_ER_*^{N_n+1}x\chi_{\overline\Sigma_{N_n+1}^{co}}(x),
\eeq
where $\chi_B$ is the indicator function on the set $B$, $R_*$ is as in \ref{H4}, $Q_{\alpha}$ is a rotation of angle $\frac{2\pi}{n}\alpha$ and the sets $\Sigma_i$ are defined by
$$
\Sigma_i:=\{x\in\R^2: |x|\in r_I^iQ_{\alpha}^i\Omega_n\}.
$$
We added a subscript $n$ to $v$ in order to highlight that $r_I$ as much as $R_*$ and $u$ depend on $n$. For simplicity, the positive integer $N_n$ is chosen such that $N_n := \inf\{N\in\mathbb{N}: r_I^N\leq \frac12\}$. 

In this section, we now seek to understand the properties of these iterated deformations. In particular, a priori, it is not obvious that the deformation $v_n$ satisfies the same differential inclusion 
\begin{align}
\label{eq:inclu_phase}
\nabla v_n(x) \in \bigcup\limits_{P \in \mathcal{P}_n} SO(2) P U P^T,
\end{align}
as the the deformation gradient $\nabla u$ from the individual layers (as constructed in Proposition \ref{prop:suff}) and with $\mathcal{P}_n$ as in \eqref{eq:symmetrygroup}. If the inclusion \eqref{eq:inclu_phase} were to hold, it would imply that $v_n$ corresponds to an exactly stress-free deformation associated with a phase transformation with associated symmetry group $\mathcal{P}_n$. However, it will turn out that while \eqref{eq:inclu_phase} is true on each individual ``onion ring layer'' for some suitable $U$, it is no longer true for the overall  concatenated construction.

In order to observe this, we first note that the map $u:\Omega_n \rightarrow \R^2$ constructed in Proposition \ref{prop:suff} is highly symmetric.

\begin{cor}
\label{cor:iterate}
Let $a>0,$ $\alpha\in(0,1)$, $r_E,r_I>0$ satisfy \eqref{RadRel}--\eqref{EqA}. Then the map $u$ constructed in Proposition \ref{prop:suff} satisfies
\begin{align}
\label{eq:i}
\begin{split}
\nabla u|_{T_i} &= \nabla u|_{Q_{\frac{n-1}{2}}T_{i+1}}^T,\qquad\text{if $n$ odd},\\
\nabla u|_{T_i} &= \nabla u|_{Q_{\frac{n}{2}}T_i},\qquad\text{if $n$ even},
\end{split}
\end{align}
for any $i=1,\dots,n$, and where $Q_j$ is the rotation of angle $\frac{2\pi j}{n}.$ Furthermore,
\beq
\label{IterPos}
R_*Q_{\alpha}\nabla u|_{T_{i+1}}Q_{\alpha}^T = Q_{1} \nabla u|_{T_i} Q_{1}^T, \qquad\text{for any odd $i=1,\dots,n$ }.
\eeq
\end{cor}

\begin{rmk}
\label{rmk:iterate}
Let us comment on the observations in Corollary \ref{cor:iterate}.
\begin{itemize}
\item[(i)] We first consider the identities in \eqref{eq:i}. These describe a symmetry of the constructed deformation gradients in each individual ``onion ring''. Depending on whether $n$ is even or odd, the deformation gradients in triangles which are ``opposite'' to each other (i.e. on $T_i$ and $Q_{\frac{n-1}{2}} T_{i+1}$ if $n$ is odd, or in $T_i$ and $Q_{\frac{n}{2}} T_i$ if $n$ is even) are related by either transposition or are directly equal (see Figure \ref{fig:iterate}).
\item[(ii)] Next, the condition in \eqref{IterPos} compares two adjacent deformations in two different but consecutive layers. The right hand side corresponds to a deformation in a triangle $T_i$ of the outer onion ring, while the left hand side corresponds to the deformation in the inner onion ring (see the definition \eqref{eq:vn} for $v_n$). The expression in \eqref{IterPos} thus states that these two adjacent deformation gradients have the same value (see Figure \ref{fig:iterate}).
\end{itemize}
\end{rmk}

\begin{proof}
In order to prove the first statement we notice that, if $n$ is even, $Q_{\frac{n}{2}}$ is a rotation by $\pi$, and therefore by \eqref{defDu} the claim follows. Let us hence assume that $n$ is odd. By symmetry we can prove the claim by assuming $i=1$, that is we need to prove that
\beq
\label{claimOdd}
Q_{\frac{n-1}{2}}P_0 U P_0 Q_{\frac{n-1}{2}}^T = U^T.
\eeq
But, using that $Q_{\frac{n-1}{2}} = - Q_{\frac{1}{2}}^T$,
$$
Q_{\frac{n-1}{2}}P_0 U P_0 Q_{\frac{n-1}{2}}^T = Q_{\frac12}^T\bigl(a e_{11}\otimes e_{11} + a^{-1}e_{11}^\perp\otimes e_{11}^\perp - \tan\frac{\pi}{n}(a^{-1}-a)e_{11}\otimes e_{11}^\perp\bigr) Q_{\frac12},
$$
and exploiting the fact that
$$
Q_{\frac12}^Te_{11} = \cos\left(\frac{\pi}{n}\right) e_{11} - \sin\left(\frac{\pi}{n}\right) e_{11}^\perp,\qquad Q_{\frac12}^Te_{11}^\perp = \cos\left(\frac{\pi}{n}\right) e_{11}^\perp + \sin\left(\frac{\pi}{n}\right) e_{11},
$$
we deduce \eqref{claimOdd}.

In order to prove \eqref{IterPos}, we can again assume without loss of generality that $i=1$. Then, proving the statement reduces to showing that
$$
R_*Q_{\alpha}P_0 U P_0 Q_{\alpha}^T = Q_{1} U Q_{1}^T,
$$
or, equivalently, that
\begin{align}
 \label{eq:corollary26} 
P_0 U P_0 = Q_{\alpha} U Q_{1-\alpha}^T.
\end{align}
A proof of this equality is given in Appendix \ref{AppCode}; we also refer to the result and argument of the next proposition.
\end{proof}

While Corollary \ref{cor:iterate} implies that the inclusion \eqref{eq:inclu_phase} holds for the outer triangles of the inner onion ring, we next prove that this fails for the inner triangles of the onion ring.

Recalling our definition of $v_n$, in Cartesian coordinates the validity of the iterability of our construction boils down to the question whether
  \begin{align}
  \label{eq:quest}
    R_* Q_{\alpha} \nabla u  Q_{-\alpha} \in \bigcup_{P \in \mathcal{P}_n} SO(2) P U P^T.
  \end{align}
  In the following we show that this condition can not be exactly satisfied with
  our choice of symmetry group $\mathcal{P}_n$ unless $\alpha=\frac{1}{2}$, in
  which case the construction is trivial.
  More precisely, we show that for $\alpha \neq \frac{1}{2}$, the inclusion \eqref{eq:quest} can only hold for either the outer or the inner triangles of the iterated ring. Additionally, we give a second proof of \eqref{IterPos}.
  
  \begin{prop}
  \label{prop:non_iterable}
    Let $\alpha \in (0,1), \alpha \neq \frac{1}{2}$, then there exists a level set
    of the gradient of $u$ in the first iterated ring such that
    \begin{align}
      \label{eq:nonincluded}
      R_* Q_{\alpha} \nabla u  Q_{-\alpha} \not \in \bigcup_{P \in \mathcal{P}_n} SO(2) P U P^T.
    \end{align}
    Moreover, the inclusion \eqref{eq:quest} holds for the outer triangles of the inner onion ring.
  \end{prop}

  \begin{proof}
    We note that the inclusion problem \eqref{eq:nonincluded} can be
    equivalently phrased in terms of the Cauchy-Green tensors.
    A self-contained proof of this reduction is provided in Lemma \ref{lem:reduction}.    
  Using the explicit structure of $\nabla u$ given in equation \eqref{defDu} and
  that $R_j \in \mathcal{R}_1^n\subset \mathcal{P}_n$, it thus suffices to consider
  two triangles $T_0,T_1$ and the inclusion problems
  \begin{align}
    \label{eq:reducedinclusion}
    \begin{split}
    Q_{\alpha} U^T U Q_{-\alpha} &\in \bigcup_{P \in \mathcal{P}_n} P^T U^T U P, \\
    Q_{\alpha} P_0U^T U P_0 Q_{-\alpha} &\in \bigcup_{P \in \mathcal{P}_n} P^T U^T U P,
    \end{split}
  \end{align}
  where
     \begin{align*}
     P_0 = P_0^T =
    \begin{pmatrix}
      e_{11} & e_{11}^\perp
    \end{pmatrix}
               \begin{pmatrix}
                 1 & 0 \\ 0 &-1
               \end{pmatrix}
                           \begin{pmatrix}
      e_{11} & e_{11}^\perp
    \end{pmatrix}^T \in O(2),      
  \end{align*} 
  and
  \begin{align*}
    U=  \begin{pmatrix}
      e_{11} & e_{11}^\perp
    \end{pmatrix}
               \begin{pmatrix}
                 a & \frac{a^{-1}-a}{\tan(\phi)} \\ 0 &\frac{1}{a}
               \end{pmatrix}
                \begin{pmatrix}
      e_{11} & e_{11}^\perp
    \end{pmatrix}^T =: \begin{pmatrix}
      e_{11} & e_{11}^\perp
    \end{pmatrix}
               U_1
                \begin{pmatrix}
      e_{11} & e_{11}^\perp
    \end{pmatrix}^T  .
  \end{align*}
  Furthermore, we may change our basis from the canonical unit basis to the basis $(e_{11}, e_{11}^\perp)$ and equivalently
  express \eqref{eq:reducedinclusion} as
  \begin{align}
    \label{eq:incl1}
 Q_{\alpha} U_1^T U_1 Q_{-\alpha} &\in \bigcup_{P \in \hat{\mathcal{P}}_n} P^T U_1^T U_1 P, \\
 \label{eq:incl2} Q_{\alpha} \diag(1,-1) U_1^T U_1 \diag(1,-1) Q_{-\alpha} &\in \bigcup_{P \in \hat{\mathcal{P}}_n} P^T U_1^T U_1 P,
  \end{align}
  where
  \begin{align}
  \label{eq:hatP}
    \hat{\mathcal{P}}_n= \mathcal{R}_1^n \cup  \begin{pmatrix}
      1 & 0 \\ 0 & -1
    \end{pmatrix} \mathcal{R}_1^n=: \hat{\mathcal{R}}_1^n \cup \hat{\mathcal{R}}_2^n
  \end{align}
  is the standard dihedral group.
  We note that
  \begin{align}
      U_1^T U_1 &=
    \begin{pmatrix}
      a^2 & \frac{1-a^2}{\tan(\phi)} \\
       \frac{1-a^2}{\tan(\phi)} & \frac{1}{a^2} + \left( \frac{a^{-1}-a}{\tan(\phi)} \right)^2
     \end{pmatrix},
  \end{align}
  is a symmetric matrix with determinant one and eigenvalues $\lambda,
  \lambda^{-1}$, which are distinct if and only if $\alpha \neq \frac{1}{2}$.
  Thus, there exists a rotation $R_{\varphi}$
  such that
  \begin{align}
    U_1^T U_1= Q_{-\varphi}\diag(\lambda,\lambda^{-1}) Q_\varphi.
  \end{align}
  Expressing \eqref{eq:incl1} and \eqref{eq:incl2} with respect to this diagonal matrix, we thus obtain the requirement that
  $\diag(\lambda,\lambda^{-1})=Q^T \diag(\lambda,\lambda^{-1})Q$ for a suitable
  $Q=Q(P,\alpha,\varphi) \in O(2)$ of the structure given below.
  Since we assume that $\lambda\neq \lambda^{-1}$ it follows that $Q$ has to map
  the eigenvectors $v_1,v_2$ of $U_1^T U_1$ to $\pm v_1,\pm v_2$ and thus 
  \eqref{eq:incl1} and \eqref{eq:incl2} are satisfied if and only if there exist
  $P \in \hat{\mathcal{P}}_n$ such that:
  \begin{align}
    \label{eq:incl1b}
    Q_{\varphi} P Q_{\alpha}Q_{-\varphi} \in \{Id, -Id, \diag(1,-1),\diag(-1,1)\}, \\
    \label{eq:incl2b}
    Q_{\varphi} P Q_{\alpha} \diag(1,-1)Q_{-\varphi} \in \{Id, -Id, \diag(1,-1),\diag(-1,1)\},
  \end{align}
  respectively.
  We first consider \eqref{eq:incl1b} and note that if $P=Q_j \in \hat{\mathcal{R}}_{1}^n$ with $j\in\{1,\dots,n\}$, the
  left-hand-side reduces to $Q_{j+\alpha}\in \{Id,-Id\}$, which is never
  satified since $\alpha \in
  (0,1)$.
  If instead $P=\diag(1,-1) Q_j$ for $j\in\{1,\dots,n\}$, then
  \begin{align}
  \begin{split}
    Q_\varphi P Q_{\alpha}Q_ {-\varphi}
    &= Q_{\varphi}\diag(1,-1) Q_{j+\alpha} Q_{-\varphi}\\ 
    &= Q_{\varphi-\frac{j+\alpha}{2}} \diag(1,-1)Q_{-\varphi+\frac{j+\alpha}{2}} \in \{\diag(1,-1),\diag(-1,1)\},
    \end{split}
  \end{align}
  if and only if
  \begin{align}
    \begin{split}
      \frac{2\pi}{n} \left(-\varphi+\frac{j+\alpha}{2}\right) \in \pi \Z \\
      \Leftrightarrow j+\alpha -2\varphi \in n\Z 
    \label{eq:incl1c} \Leftrightarrow \alpha-2\varphi \in \Z.
    \end{split}
  \end{align}
  We will later compute $\varphi$ to show that this condition is satisfied iff $\alpha =\frac{1}{2}$.
  Before proceeding to this, let us however also consider the second inclusion \eqref{eq:incl2b}.
  If $P=\diag(1,-1)Q_j \in \hat{\mathcal{R}}_n^1$ for some $j\in\{1,\dots,n\}$, the left-hand-side of \eqref{eq:incl1b} reduces to
  \begin{align*}
    Q_\varphi Q_{-j-\alpha}\diag(1,-1) \diag(1,-1) Q_{-\alpha} = Q_{-j-\alpha} \not \in \{Id,-Id\}.
  \end{align*}
  If instead $P=Q_j$ for some $j\in\{1,\dots,n\}$, we obtain
  \begin{align}
  \label{eq:sec_cond}
  \begin{split}
&    Q_\varphi Q_{j+\alpha}\diag(1,-1) Q_{-\varphi} \\
    &= Q_{\varphi}Q_{(j+\alpha)/2} \diag(1,-1)Q_{-(j+\alpha)/2} Q_{-\varphi} \in \{\diag(1,-1),\diag(-1,1)\},
    \end{split}
  \end{align}
  if and only if
  \begin{align}
   \begin{split}
      \frac{2\pi}{n} \left(-\varphi-\frac{j+\alpha}{2}\right) \in \pi \Z \\
      \Leftrightarrow -j-\alpha -2\varphi \in n\Z 
    \label{eq:incl2c} \Leftrightarrow -\alpha-2\varphi \in \Z.
    \end{split} 
  \end{align}
  In particular, considering the difference of \eqref{eq:incl1c} and
  \eqref{eq:incl2c}, we observe that for both inclusions \eqref{eq:incl1c} and \eqref{eq:incl2c} to be satisfied it is necessary that
  $2\alpha \in \Z$ and thus $\alpha =\frac{1}{2}$. This concludes the proof of
  the first statement of the proposition.

  We additionally show that \eqref{eq:incl2c} is always satisfied for all $\alpha \in (0,1)$ by
  computing $\varphi=\varphi(\alpha)$.
  Indeed, we claim that
  \begin{align}
    v= \left(\cos\left(\frac{2\pi}{n} \frac{\alpha-1}{2}\right),\sin\left(\frac{2\pi}{n} \frac{\alpha-1}{2}\right)\right)
  \end{align}
  is an eigenvector of $U_1^T U_1$.
  Since $\varphi$ was defined by $U_1^T U_1=
  Q_{-\varphi}\diag(\lambda,\lambda^{-1})Q_{-\varphi}^T$, this implies that
  $\varphi=-\frac{\alpha-1}{2}$ and hence \eqref{eq:incl2c} is satisfied.
  It remains to show that $v$ is indeed an eigenvector.
  As we consider two-dimensional matrices, it suffices to show that $U_1^T U_1
  v$ is colinear to $v$ and thus equivalently
  \begin{align}
  \label{eq:eigen}
  \begin{split}
   0 &=  v^T
    \begin{pmatrix}
      0 & -1 \\ 1 & 0
    \end{pmatrix} U_1^T U_1 v \\
     &= \left(\frac{1}{a^2}-a^2+\left( \frac{a^{-1}-a}{\tan(\phi)} \right)^2 \right) \sin\left(\frac{2\pi}{n} \frac{\alpha-1}{2}\right)\cos\left(\frac{2\pi}{n} \frac{\alpha-1}{2}\right) \\
     & \quad + \frac{1-a^2}{\tan(\phi)} \left(\cos^2\left(\frac{2\pi}{n} \frac{\alpha-1}{2}\right)-\sin^2\left(\frac{2\pi}{n} \frac{\alpha-1}{2}\right) \right) \\
     &= \left(\frac{1}{a^2}-a^2+\left( \frac{a^{-1}-a}{\tan(\phi)} \right)^2 \right) \frac{1}{2} \sin\left(\frac{2\pi}{n} (\alpha-1)\right) \\
    & \quad + \frac{1-a^2}{\tan(\phi)}\cos\left(\frac{2\pi}{n} (\alpha-1)\right).
\end{split}  
  \end{align}
  We recall that by \eqref{EqA}
  \begin{align*}
 &   a^2=\frac{\sin(\frac{2\pi}{n}(1-\alpha))}{\sin(\frac{2\pi}{n}\alpha)}, \
    \tan(\phi)= \tan\left(\frac{n-2}{2n}\pi\right)= \cot\left(\frac{\pi}{n}\right), \\
  &  (a^{-1}-a)^2= a^{-2}+a^2-2.
  \end{align*}
  We now note that the equality \eqref{eq:eigen} is satisfied if $a^2=1$ and thus
  $\alpha=\frac{1}{2}$, otherwise we may divide by $(1-a^2)$
  to further reduce to proving
  \begin{align*}
    \quad \left(a^{-2}+1 + (a^{-2}-1)\tan^2\left(\frac{\pi}{n}\right) \right) \frac{1}{2} \sin\left(\frac{2\pi}{n}(\alpha-1)\right) + \cos\left(\frac{2\pi}{n}(\alpha-1)\right) \tan\left(\frac{\pi}{n}\right) =0.
  \end{align*}
  For easier notation, we introduce $\gamma=\frac{2\pi}{n}\alpha$,
  $\beta=\frac{2\pi}{n}(\alpha-1)=\gamma-\frac{2\pi}{n}$, and thus $a^{-2}=-\frac{\sin(\gamma)}{\sin(\beta)}$.
  Then the above simplifies to
  \begin{align*}
    \frac{1}{2}\left(-\sin(\gamma)+\sin(\beta) + \left(-\sin(\gamma)-\sin(\beta)\right)\tan^2\left(\frac{\pi}{n}\right)\right) + \cos(\beta)\tan\left(\frac{\pi}{n}\right)=0.
  \end{align*}
  We then insert $\sin(\gamma)=\cos(\beta)\sin\left(\frac{2\pi}{n}\right)+
  \sin(\beta)\cos\left(\frac{2\pi}{n}\right)$ and collect terms involving $\cos(\beta)$
  and $\sin(\beta)$: 
  \begin{align*}
    \cos(\beta) \left(-\frac{1}{2}\left(\sin\left(\frac{2\pi}{n}\right)+ \sin\left(\frac{2\pi}{n}\right)\tan^2\left(\frac{\pi}{n}\right)\right)+\tan\left(\frac{\pi}{n}\right)\right) \\
 + \sin(\beta)\frac{1}{2}\left(-\cos\left(\frac{2\pi}{n}\right)+1-\left(\cos\left(\frac{2\pi}{n}\right)+1\right)\tan^2\left(\frac{\pi}{n}\right)\right)=0.
  \end{align*}
  In order to observe that this is indeed correct, we
 use the double-angle identities
$\sin(\frac{2\pi}{n})=2\sin(\frac{\pi}{n})\cos(\frac{\pi}{n})$
and $\cos(\frac{2\pi}{n})=\cos^2(\frac{\pi}{n})-\sin^2(\frac{\pi}{n})$ to
obtain that
\begin{align*}
  & \quad -\frac{1}{2}\left(\sin\left(\frac{2\pi}{n}\right)+ \sin\left(\frac{2\pi}{n}\right)\tan^2\left(\frac{\pi}{n}\right)\right) + \tan\left(\frac{\pi}{n}\right) \\
  &= - \left(\sin\left(\frac{\pi}{n}\right)\cos\left(\frac{\pi}{n}\right)+ \sin\left(\frac{\pi}{n}\right)\cos\left(\frac{\pi}{n}\right)\left(\frac{1}{\cos^2\left(\frac{\pi}{n}\right)}-1\right)\right) + \frac{\sin(\frac{\pi}{n})}{\cos(\frac{\pi}{n})}=0, 
\end{align*}
as well as
\begin{align*}
  & \quad -\cos\left(\frac{2\pi}{n}\right)+1 - \left(\cos\left(\frac{2\pi}{n}\right)+1\right)\tan^2\left(\frac{\pi}{n}\right) \\
  &= -\cos^2\left(\frac{\pi}{n}\right)+\sin^2\left(\frac{\pi}{n}\right)+1 - \left(\cos^2\left(\frac{\pi}{n}\right)-\sin^2\left(\frac{\pi}{n}\right)+1\right)\tan^2\left(\frac{\pi}{n}\right) \\ 
  &= 2 \sin^2\left(\frac{\pi}{n}\right) - 2 \cos^2\left(\frac{\pi}{n}\right)\tan^2\left(\frac{\pi}{n}\right)=0.
\end{align*}
  This concludes the proof. 
\end{proof}

\subsubsection{Remarks on geometrically non-linear tripole star constructions}
\label{sec:stars}

\begin{figure}[t]
\centering
\includegraphics[width=9cm,page=8]{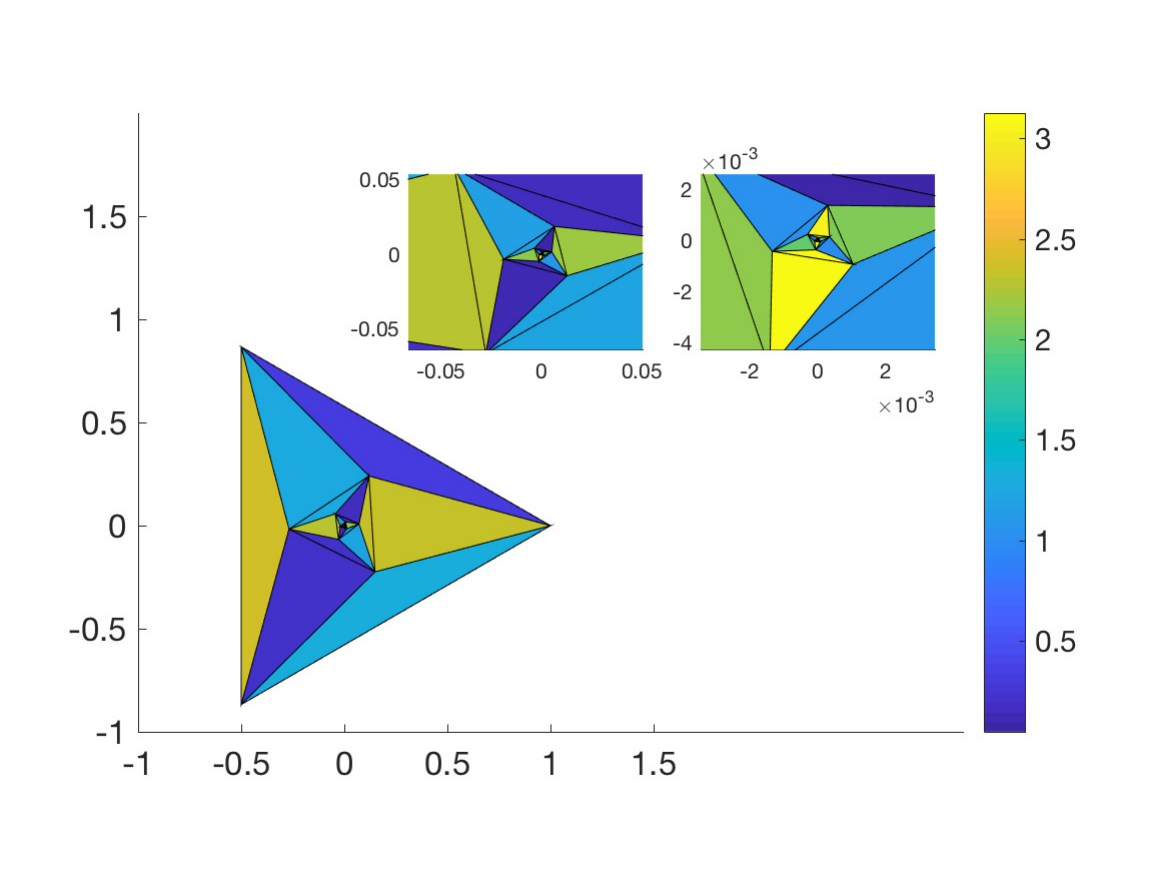}
\caption{
Exact construction of a self-similar tripole star obtained by solving (\ref{IterPos}). Here we set $\alpha=0.47$. The colormap represents 
the eigenvector associated with the largest eigenvalue of the right Cauchy–Green tensor $\nabla u^T\nabla u$ which is parametrised as $(\cos\theta,\sin\theta,0)$. 
The right tensor is rotated by
an angle equal to $\frac{2\pi}{3} \alpha$ when moving across each hierarchy in the onion construction. Each layer consists of a sharp $\frac{\pi}{3}$ rotation
dictated by the symmetry of the problem and an additional small rotation of amplitude equal to
$\frac{\pi}{3}|1-2\alpha|$ which 
is required by compatibility
 and which
causes additional rotational stretch
(see in-plot magnifications).
}
\label{1902201704}
\end{figure}

Exact solutions obtained for $\alpha\approx 1/2$ as in Figure \ref{1902201704} display self-similar “nested” structures. These are reminiscent of ``tripole star structures'' -- a distinctive
type of patterns which are observed in a class of metal alloys undergoing the (three-dimensional) hexagonal-to-orthorhombic transition in the plane \cite{MA80,MA80a,KKK88,KK91}, a transformation characterised by three martensitic variants with special rotational symmetries.
Investigation of these types of microstructures (typically in two-dimensional models of the hexagonal-to-orthorhombic transformation) has been object of extensive numerical studies based on the minimisation of stored energies defined in both fully non-linear and linearised elasticity for the hexagonal-to-orthorhombic transformation (see, for instance \cite{CKOOKI07, WWC99,JCD04,CJ01,PL13} and the references quoted therein). Ultimately, in many of these works minimisation boils down to solving the associated differential inclusion problem (of the form \eqref{eq:ex_stress_free}), for a piecewise affine vector $u:\Omega\to\R^2$ to be taken over a domain $\Omega$ and with boundary conditions that are suitable to reproduce the tripole stars. 

In the experimental literature on the hexagonal-to-orthorhombic phase transformation, it is noted that the observed star patterns are of low but not of vanishing energy, in the sense that they are not exactly stress-free within the geometrically non-linear theory of elasticity. The experimental literature describes these structures as disclinations. This is in accordance with our results from the previous sections stating that
\begin{itemize}
\item[(i)] a single, exactly stress-free layer of a tripole star deformation can not be achieved with three variants of martensite, but requires six variants,
\item[(ii)] an iteration of the individual layers is not possible with only three (or six) variants of martensite. Already in the second layer, this will lead to misfits (which give rise to the experimentally observed stresses). In \cite{DV76}, for instance, the authors report a deviation of the outer-most and the second inner iteration by roughly four degrees.
\end{itemize}
As also observed in the literature \cite{KK91} this is a \emph{geometrically non-linear} effect. Indeed, by introducing the geometrically linearised elasticity version of the (two-dimensional) hexagonal-to-orthorhombic phase transformation, an exact construction of a self-similar tripole star pattern has been obtained
in \cite{CPL14} by imposing kinematic compatibility across each interface and by defining a displacement field that reproduces the three martensitic variants associated with the hexagonal-to-orthorhombic transformation. The symmetry and rigidity of the problem is inherited in the shape of the microstructure in that the tripole stars are obtained by rotating, rescaling and translating a copy of a single kite-shaped polygon which is perfectly symmetric with respect to its axes.

The results of Section \ref{sec:nonlinear_n_well} generalise the linearised construction of \cite{CPL14} in the following way. By replacing the non-linear differential inclusion associated with the hexagonal-to-orthorhombic transformation with (\ref{IterPos}) which involves extra rotations (and reflections) of the bain strain matrices and therefore more flexibility, it is possible to construct exact tripole stars by matching rotated and dilated copies of slightly non-symmetrical tetrahedra and to quantify the deviation from the perfectly symmetric construction of the linearised case. Thanks to \eqref{eq:i} we can estimate from above the nonlinear elastic mismatch in one single layer of our construction caused by having just three martensitic variants (hexagonal-to-orthorhombic transformation) rather than six (as in \cite{CKZ17} or Proposition \ref{prop:suff}). Indeed, this can be bounded from above by (cf. Section \ref{sec:number_wells})
$$
|U-U^T| = \tan\Bigl(\frac\pi n\Bigr)|a-a^{-1}|, 
$$
and so is small whenever $a\approx1$. This small mismatch is not captured by the linear elasticity model.

In order to achieve the matching across every annulus, the deformation field necessarily has to incorporate, at each hierarchy, an additional rotation $Q_{\alpha}$ of an angle equal to $\frac{2\pi}{n}\alpha$ (see also the comment in the caption of Figure \ref{1902201704}). This leads to the presence of elastic energies.

Indeed, it is interesting to view the constructions from an energetic point of view. Setting $K_n(a):= \bigcup\limits_{P\in \mathcal{P}_n} SO(2)P U(a) P^T$ with $U(a)$ as in \eqref{eq:U1}, we consider the energy
\begin{align*}
\mathcal{E}_n(\nabla u) = \int\limits_{\Omega_n} \dist^2(\nabla u,K_n(a) \cup SO(2)) dx + \epsilon |\nabla^2 u|(\Omega_n).
\end{align*}
Here the additional well $SO(2)$ corresponds to the austenite phase. For this energy, the \emph{single layer} deformations from Proposition \ref{prop:suff} are extremely inexpensive: the elastic energy vanishes, while the surface energy is finite. Hence the energy behaves like $C\epsilon$ for some constant $C>0$.
However, \emph{iterated} constructions as in Proposition \ref{prop:n_infty} already cost more: here, by the geometric refinement of the structures, the surface energy still behaves as $C \epsilon$ for some constant $C>0$, while the elastic energy can be estimated by 
\begin{align*}
E_{elast} \leq C \sum\limits_{j=1}^{N_n}r_I^{2j}\dist(Q_{\alpha j}^T \nabla u^T \nabla u Q_{\alpha j}, \bigcup_{j=1}^{2n} U_j^T U_j),
\end{align*}
where $U_j = Q_{\frac{j-1}{2}}U(a)^T U(a) Q_{\frac{j-1}{2}}^T$ if $j$ is odd and $U_j = Q_{\frac{j-2}{2}}U(a)^T U(a) Q_{\frac{j-2}{2}}^T$ if $j$ is even and $U(a)$ is as in \eqref{eq:U1}.
By arguments as in the proof of Proposition \ref{prop:non_iterable} for $\alpha$ close to $\frac{1}{2}$ and $N_n\in \N$ not too large, the total energy thus is controlled by
\begin{align*}
\mathcal{E}_n(\nabla v_n) \leq  C \sum\limits_{j=1}^{N_n}r_I^{2j}\dist(2j\alpha, \Z) + C \epsilon. 
\end{align*}
Hence, we obtain a three parameter minimisation problem, with the parameters $\alpha, \epsilon, N_n$ (where the $N_n$ dependence is mild as the series in $j$ is summable as a geometric series).
In particular, in spite of the presence of stresses, for $\alpha \in (0,1)$ sufficiently close to $\frac{1}{2}$ (depending on $N_n$ and $\epsilon$) there is a regime, in which also in the geometrically non-linear setting, it is feasible that the tripole star structures are observed
and are rather stable.

\subsection{The limit $n\to\infty$} 
\label{sec:infty}

Equipped with the finite $n$ construction from the previous section, in this section, we discuss the passage to the limit as $n\to\infty$. Physically, this corresponds to the nematic liquid crystal elastomer limit.

We begin by discussing the limit of the construction from Proposition \ref{prop:suff}. 
First, due to \eqref{RadRel}, $\frac{r_I}{r_E} = 1-\frac{2\pi\sqrt{\alpha(1-\alpha)}}{n} + O(n^{-2})$ as $n\to\infty$. Therefore the internal radius converges to the external one. Hence, in order to observe a non-trivial limiting configuration as $n\rightarrow \infty$, in the sequel, we iterate more and more layers of our construction for finite $n$ (as discussed in Section \ref{sec:finite_n}). 

Let us explain this in more detail. Without loss of generality, below we consider $v_n := R_E u_n$, where as in the proof of Proposition \ref{prop:suff}, $R_E$ is such that $\nabla v_n|_{T_1} = U$ and where $U$ is as in \eqref{eq:U1} (cf. proof of Proposition \ref{prop:suff}). 
As in Section \ref{sec:iteration_layer}, we now fix $\alpha>0$, set for a matter of simplicity $r_E=1$, take $r_I$ satisfying \eqref{RadRel} and consider
\begin{align*}
v_n(x) = \sum_{k=0}^{N_n} r_I^k R_E R_*^kQ_{\alpha}^k u(r_I^{-k}Q_{-\alpha}^kx)\chi_{\Sigma_k} + R_E x\chi_{\R^2\setminus\Sigma_0^{co}} + R_ER_*^{N_n+1}x\chi_{\overline\Sigma_{N_n+1}^{co}},
\end{align*}
where $\chi_B$ is the indicator function on the set $B$, $R_*$ is as in \ref{H4}, $Q_{\alpha}$ is a rotation of angle $\frac{2\pi}{n}\alpha$ and the sets $\Sigma_i$ are defined by
$$
\Sigma_i:=\{x\in\R^2: |x|\in r_I^iQ_{\alpha}^i\Omega_n\}.
$$
Again, we choose the positive integer $N_n$ such that $N_n := \inf\{N\in\mathbb{N}: r_I^N\leq \frac12\}$.

Below we denote by $B_r$ the open ball centred at zero and of radius $r$. With this notation in hand, we pass to the limit $n\rightarrow \infty$, thus, physically passing to the liquid crystal elastomer regime (see Section \ref{sec:solids_liquids}).

\begin{prop}
\label{prop:n_infty}
Let $\alpha \in (0,1)$, then there exists $v\in W_{loc}^{1,\infty}(\R^2;\R^2)$ such that $v_n\to v$ in the
$W_{loc}^{1,p}(\R^2;\R^2)-$norm for each $p\geq1$, and $\nabla v= Q(\beta_0)$ on $\R^2\setminus B_1$, $\nabla v = Q(\beta_0)
Q(\rho_0 \log\frac12)$ on $B_{1/2}$ and 
$$
\nabla v = Q(\rho_0 \log r) Q(\omega)\bigl(a \bar{e}_{11}\otimes\bar{e}_{11} +a^{-1}\bar{e}_{11}^\perp\otimes\bar{e}_{11}^\perp \bigr)Q^T(\omega), \qquad \text{in $B_1\setminus B_{\frac12}$},
$$
where $Q(s)$ denotes the rotation of angle $s\in\R$,
$\rho_0=\frac{2\alpha-1}{\sqrt{\alpha(1-\alpha)}}$, $\beta_0 = \sin^{-1}(1-2\alpha)$, $r =|x|$ and $\omega =
\arctan\Bigl(\frac{x\cdot e_2}{x\cdot e_1} \Bigr).$ Furthermore, $a =
\sqrt{\frac{1-\alpha}{\alpha}}$ and {$\bar{e}_{11} =
  (\sqrt{1-\alpha},-\sqrt{\alpha}).$}
\end{prop}

\begin{rmk}
\label{rmk:notation}
If we seek to emphasise the dependence of the limiting deformation $v$ on $a$, we also use the notation $v_{a}$.
\end{rmk}

\begin{proof}
We have that $v_n(x) = R_Ex$ in $\R^2\setminus B_1$ for every $n$. As $n\to\infty$, the rotation matrix $R_E\to Q(\sin^{-1}(1-2\alpha))$, a rotation of angle $\sin^{-1}(1-2\alpha)$. Indeed, $R_E$ is such that $R_E \overline{E_1E_2} = P_0UP_0\overline{E_1E_2}$, and hence the angle $\beta_0$ of $R_E$ is given by 
$$
\beta_0=\sin^{-1} \frac{\overline{E_1E_2}\times P_0UP_0\overline{E_1E_2}}{|\overline{E_1E_2}| |P_0UP_0\overline{E_1E_2}|}. 
$$
We recall that
\begin{align*}
  \overline{E_1E_2}&= l_2e_{11}- l_1 e_{12}, \\
  P_0 U P_0 \overline{E_1E_2}&= a l_2 e_{11}- \frac{l_1}{a}e_{12}.
\end{align*}
and that by \eqref{LengthE1E2}--\eqref{LengthE1E2bis}
\begin{align*}
  |P_0UP_0\overline{E_1E_2}|^2= |\overline{E_1E_2}|= l_2^2 (1+a^2+2a\sin\frac{\pi}{n}).
\end{align*}
Using \eqref{eq:lin_comb} we deduce that 
$$
\beta_0 = \sin^{-1}\frac{(a^2-1)\cos\frac\pi n}{1+a^2+2a\sin\frac\pi n} \to \sin^{-1}(1-2\alpha)$$ 
as $n\to\infty.$

Therefore, we focus on the deformation in $B_1$ and notice that since $v_n$ is a bounded sequence in $W^{1,\infty}(B_1;\R^2)$, there exists $\tilde v\in W^{1,\infty}(B_1;\R^2)$ with $\tilde v|_{\partial B_1} = R_E x$, and a non relabelled subsequence such that $v_n\to \tilde v$ weakly$-*$ in $W^{1,\infty}(B_1;\R^2)$, uniformly in $C(B_1;\R^2)$. We now claim that $\nabla v_n \to \nabla v$ a.e. in $B_1$, which (together with dominated convergence) in turn implies $\nabla v_n \to \nabla v$ in $L^p(B_1)$ for each $p\geq 1$, and $\tilde v=v$. \\

Let us start by observing that, by our preceding considerations, the deformation $v_n$ is explicitly given by \eqref{eq:vn}, and thus
\begin{align}
  \label{eq:nablaurings}
    \nabla v_n(x)= \sum_{k=0}^{N_n} R_E R_*^kQ_{\alpha}^k(\nabla u) (r_I^{-k}Q_{-\alpha}^kx)Q_{-\alpha}^k \chi_{\Sigma_k} + R_E \chi_{\R^2\setminus\Sigma_0^{co}} + R_ER_*^{N_n+1}\chi_{\overline\Sigma_{N_n+1}^{co}},
  \end{align} 
  and where $\nabla u$ is given by \eqref{defDu}. Next we notice that the set
  $$
  \mathcal{Z} := \bigcup_{n=3}^\infty \bigcup _{k=0}^{N_n} \bigcup _{i=1}^{n} \Bigl(\partial(r_I^k Q_{\alpha}^k Q_{i-1} T_1)\cup \partial(r_I^k Q_{\alpha}^k Q_{i-1} T_2)\Bigr),
  $$
  which is the union over all the boundaries of the $2 n$ triangles in each of the $N_n+1$ layers $\Sigma_k$, has zero two-dimensional Lebesgue measure. Here, as above, $Q_{j}$ is a rotation of angle $\frac{2\pi j}n.$ Let us now fix $x\in B_{\frac12}\setminus\mathcal{Z}$. Then, there exists $n_x \geq 3$ such that $x\in B_{\frac12} \setminus \bigcup_{k=0}^{N_n}\Sigma_k$ for every $n\geq n_x$. For $n\geq n_x$, then $\nabla v_n(x) = R_E R_*^{N_n}$ and $R_E R_*^{N_n} \to Q(1-2\alpha) Q(\rho_0\log\frac12)$ as $n\to\infty$ (see \eqref{RstarConv} below).  
  
Let now $x\in B_1 \setminus\bigl(B_{\frac12}\cup\mathcal{Z}\bigr)$. Then, there exists $n_x \geq 3$ such that $x\in \bigcup_{k=0}^{N_n} \Sigma_k$ for every $n\geq n_x$. Therefore, given any $n\geq n_x$ we have that there exists $0\leq k \leq N_n$ and $1\leq i\leq n$ such that $x\in r_I^kQ_{\alpha}^kQ_{i-1} T_1$ or such that $x\in r_I^kQ_{\alpha}^kQ_{i-1} T_2$. Suppose without loss of  generality the first inclusion holds, as the second case can be treated similarly (see \eqref{Uconv2} below). We then   have that 
\beq
\label{Gradvn}
\nabla v_n(x) = R_*^kQ_{\alpha}^k Q_{i-1} U Q_{i-1}^T Q_{-\alpha}^k.
\eeq
Furthermore, 
\beq
\label{distPoints}
|x - r_I^k Q^k_\alpha Q_{i-1} e_1| \leq r_I^{k}\diam\, T_1 \leq r_I^k\frac{c_0}{n},
\eeq
where we denoted by $\diam\, T_1$ the maximal Euclidean distance between two points within $\overline T_1$ (the closure of $T_1$), which can be bounded by a positive constant $c_0$ (independent of $n,i,k$) divided by $n$. Let now $(r,\omega),(r_0,\omega_0)\in (\frac12,1)\times [0,2\pi)$ be respectively the polar coordinates of $x$ and $r_I^k Q^k_\alpha Q_{i-1} e_1$. We notice that, by \eqref{distPoints}, 
\beq
\label{RotsConv}
|Q^k_\alpha Q_{i-1} - Q(\omega)| = |Q(\omega_0) - Q(\omega)|\leq \frac{c_1}{n},
\eeq
for some $c_1>0$ independent of $i,k,n.$ On the other hand, 
  \begin{align*}
    k =   {\frac{\log(r_0)}{\log(r_I)}} .
    \end{align*}
  Now, as $\log r_I = -\frac{2\pi\sqrt{\alpha(1-\alpha)}}{n} + o(n^{-1})$
  as $n\to \infty$, and using the notation that $Q_{\varphi} = Q(\frac{2\pi}{n}\varphi)$, we obtain that 
  \beq
  \label{RstarConv}
  \begin{split}
    R_*^k&= Q_{(1-2\alpha) k}= Q\left(\frac{2\pi}{n} (1-2\alpha) \frac{\log(r_0)}{-\frac{2\pi\sqrt{\alpha(1-\alpha)}}{n} +o(n^{-1})}  \right)\\
    &\rightarrow Q\left(\frac{(2\alpha-1)}{\sqrt{\alpha(1-\alpha)}} \log(r_0)\right)=: Q(\rho_0\log(r_0)).
  \end{split}
  \eeq
Finally, we recall that $e_{11}$ is a normalised version of
    \begin{align*}
      1-\frac{r_I}{r_E}e^{i\frac{2\pi}{n}\alpha} &= 1- \left(1-\frac{2\pi}{n}\sqrt{\alpha(1-\alpha)}\right) \left(1+i\frac{2\pi}{n}\alpha \right) + O(n^{-2}) \\
      &= \frac{2\pi}{n} (\sqrt{\alpha(1-\alpha)} - i \alpha) + \mathcal{O}(n^{-2}),
    \end{align*}
    where we have identified $\R^2$ with $\C$.
    Normalising and taking the limit, we hence obtain that
    \begin{align*}
      e_{11}\rightarrow
      \begin{pmatrix}
        \sqrt{1-\alpha} \\ -\sqrt{\alpha}
      \end{pmatrix} =: \bar{e}_{11}.
    \end{align*}
     As a consequence, 
     \beq
     \label{Uconv}
U \to a \bar{e}_{11}\otimes\bar{e}_{11} +a^{-1}\bar{e}_{11}^\perp\otimes\bar{e}_{11}^\perp =: U_\infty,
     \eeq
     where we used that $\frac{1}{\tan{\phi}}=
     \frac{1}{\tan(\frac{n-2}{2n}\pi)}\rightarrow 0$.
     We remark that, in the case $x\in r_I^kQ_{\alpha}^kQ_{i-1}T_2$ for some $n\geq 3$, $0\leq i\leq n$ and some $0\leq k\leq N_n$ the proof differs just for \eqref{Uconv} which  should read 
     \beq
     \label{Uconv2}
      \begin{split}
       P_0 U P_0 &\rightarrow (\bar{e}_{11}\otimes \bar{e}_{11} -  e_{11}^\bot \otimes e_{11}^\bot) (a \bar{e}_{11}\otimes\bar{e}_{11} +a^{-1}\bar{e}_{11}^\perp\otimes\bar{e}_{11}^\perp ) (\bar{e}_{11}\otimes \bar{e}_{11} -  e_{11}^\bot \otimes e_{11}^\bot)  \\
       &= a \bar{e}_{11}\otimes\bar{e}_{11} +a^{-1}\bar{e}_{11}^\perp\otimes\bar{e}_{11}^\perp = U_\infty.
     \end{split}
     \eeq
Therefore, collecting \eqref{Gradvn}--\eqref{Uconv}, by the triangle inequality we get
\[
\begin{split}
|\nabla v_n(x) &-\nabla v(x)|\\
\leq c_2 &\bigl( \max\{|U-U_\infty|,|P_0UP_0-U_\infty|\} + |Q_{\alpha}^kQ_{i-1} - Q(\omega)| \\
&+ |R^k_* - Q(\rho_0\log r_0)| 
+ |Q(\rho_0\log r) - Q(\rho_0\log r_0)| \bigr) \to 0,
\end{split}
\]
for some $c_2>0$. This concludes the proof of the claim.
\end{proof}

With the results of Proposition \ref{prop:n_infty} in hand, we can also compute the associated deformation:

\begin{cor}
\label{cor:def}
Let $\alpha \in (0,1)$. Then, we have
  \begin{align*}
    v_n(x)\rightarrow v_a(x):=
    r Q(\omega+\rho_0\log(r))
    \begin{pmatrix}
      2 \sqrt{\alpha (1-\alpha)} \\ 1-2\alpha 
    \end{pmatrix},
  \end{align*}   
  uniformly in $B_1\setminus B_{\frac{1}{2}}$.
\end{cor}

\begin{figure}[t]
  \centering
  \includegraphics[width=0.4\linewidth,page=9]{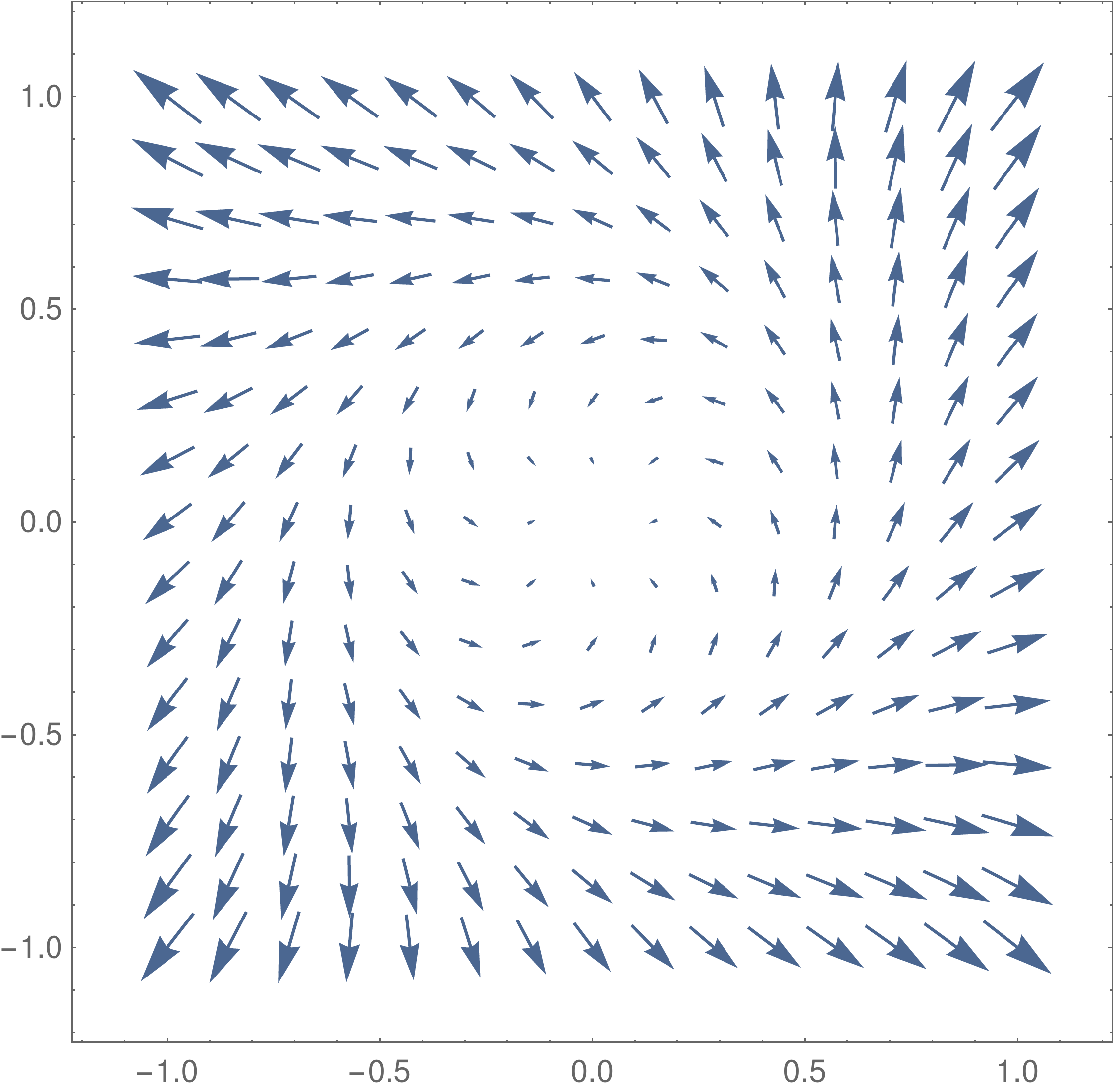}
  \includegraphics[width=0.4\linewidth,page=10]{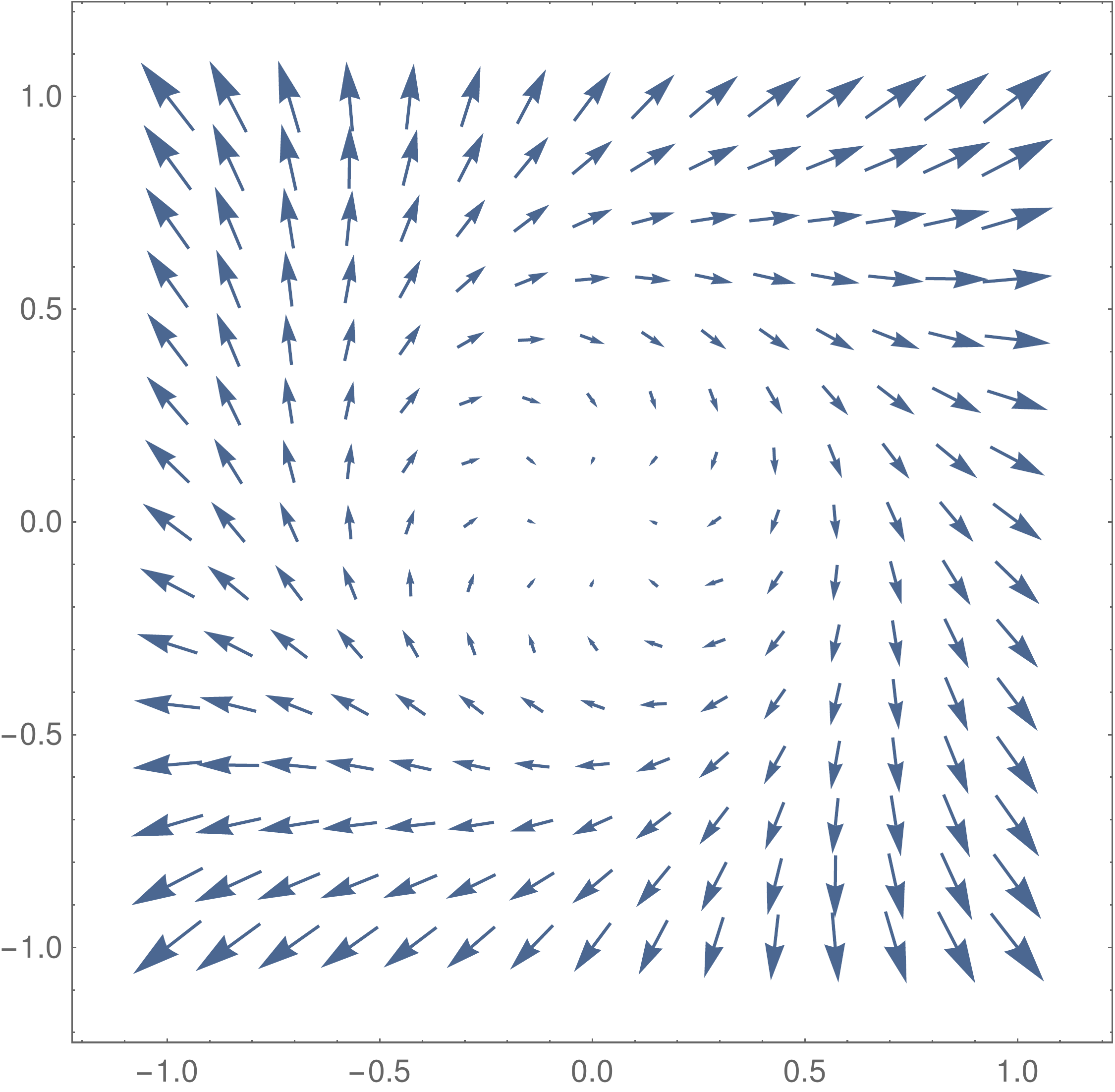} 
  \caption{Vector plot of $v(x)$ for $\alpha=0.2$ (left) and $\alpha=0.8$ (right).}
  \label{fig:vectorfields}
\end{figure}

\begin{proof}
In order to compute the underlying vector field, we note that in polar coordinates
\begin{align*}
x = \begin{pmatrix} r \cos(\omega) \\ r \sin(\omega) \end{pmatrix},
\end{align*}
we have by virtue of the chain rule
\begin{align*}
\begin{pmatrix}
\frac{\p v_j}{\p x_1}\\
\frac{\p v_{j}}{\p x_2}
\end{pmatrix}
= 
\begin{pmatrix}
\cos(\omega) & - \frac{1}{r}\sin(\omega)\\
\sin(\omega) & \frac{1}{r} \cos(\omega)
\end{pmatrix}
\begin{pmatrix}
\frac{\p \hat{v}_j}{\p r}\\
\frac{\p \hat{v}_j }{\p \omega}
\end{pmatrix},
\end{align*}
where $\hat{v}_j(r,\omega) = v_j(r \cos(\omega), r \sin(\omega))$ and $j\in \{1,2\}$. As a consequence,
\begin{align*}
\begin{pmatrix}
\frac{\p \hat{v}_j}{\p r}\\
\frac{\p \hat{v}_j}{\p \omega}
\end{pmatrix}
= r \begin{pmatrix}
\frac{1}{r}\cos(\omega) &  \frac{1}{r}\sin(\omega)\\
-\sin(\omega) &  \cos(\omega)
\end{pmatrix}
\begin{pmatrix}
\frac{\p v_j}{\p x_1}\\
\frac{\p v_{j}}{\p x_2}
\end{pmatrix}
= F(r,\omega) M(\omega) Q(-\rho_0 \log(r)) e_j,
\end{align*}
where 
\begin{align*}
F(r,\omega) &= 
 r \begin{pmatrix}
\frac{1}{r}\cos(\omega) &  \frac{1}{r}\sin(\omega)\\
-\sin(\omega) &  \cos(\omega)
\end{pmatrix}, \\
\ M(\omega) &= Q(\omega)(a \overline{e}_{11}\otimes \overline{e}_{11} + a^{-1} \overline{e}_{11}^{\perp}\otimes \overline{e}_{11}^{\perp})^T Q(-\omega).
\end{align*}
Simplifying the corresponding expressions, we obtain
\begin{align*}
\begin{pmatrix}
\frac{\p \hat{v}_1}{\p r} \\
\frac{\p \hat{v}_1}{\p \omega}
\end{pmatrix}
= \begin{pmatrix}
\frac{1+2\alpha(\alpha-1)}{\sqrt{\alpha(1-\alpha)}}\cos(\omega + \rho_0 \log(r)) + (1-2\alpha) \sin(\omega + \rho_0 \log(r))\\
r\left( (-1+2\alpha) \cos(\omega + \rho_0 \log(r)) - 2\sqrt{\alpha(1-\alpha)} \sin(\omega + \rho_0 \log(r)) \right)
\end{pmatrix},\\
\begin{pmatrix}
\frac{\p \hat{v}_2}{\p r} \\
\frac{\p \hat{v}_2}{\p \omega}
\end{pmatrix}
= \begin{pmatrix}
(-1+2\alpha)\cos(\omega + \rho_0 \log(r)) + \frac{1+2\alpha(\alpha-1)}{\sqrt{1-\alpha}} \sin(\omega + \rho_0 \log(r))\\
r\left( 2\sqrt{(1-\alpha)\alpha}\cos(\omega+ \rho_0 \log(r)) + (-1+2\alpha) \sin(\omega+ \rho_0 \log(r)) \right)
\end{pmatrix}.
\end{align*}
Integrating these expressions (in particular the $\omega$ integration becomes quite straight forward) then yields the desired result. 
\end{proof}

\subsection{From elastic crystals to nematic elastomers}
\label{sec:solids_liquids}

As explained in the introduction, the specific solutions to the differential inclusion which we consider in this article allow us to treat Conti-type constructions for elastic crystals and nematic liquid crystal elastomers within a unified framework. This is particularly transparent in the limit $n\rightarrow \infty$. Here as a direct consequence of the considerations in the last section, we infer the following observation:

\begin{cor}
\label{lem:K_inf}
Let $K_n(a):= \bigcup\limits_{P \in \mathcal{P}_n } SO(2) P U(a) P^T$, where $U(a)$ is as in \eqref{eq:U1}. Then, as $n \rightarrow \infty$, it converges in a pointwise sense to the set
\begin{align}
\label{eq:limit_set}
\begin{split}
K_{\infty}(a)
&:= \bigcup\limits_{P\in O(2)} SO(2) P U_{\infty}(a) P^T \\
&= \{F \in \R^{2\times 2}: \ \det(F)=1, \ \lambda(F) = a, \ \mu(F) = a^{-1}\},
\end{split}
\end{align}
{where $U_{\infty}(a):= a \bar e_{11} \otimes \bar e_{11} + \frac{1}{a} \bar e_{11}^{\perp}\otimes \bar e_{11}^{\perp} $, $\bar e_{11}$ is as in Proposition \ref{prop:n_infty}} and where $\lambda(F), \mu(F)$ denote the singular values of the matrix $F$. 
In particular, the deformation $v$ from Proposition \ref{prop:n_infty} is a solution to the differential inclusion
\begin{align*}
\nabla v \in K_{\infty}(a) \mbox{ in } B_1.
\end{align*}
\end{cor}

We note that the set in \eqref{eq:limit_set} essentially corresponds to the planar nematic liquid crystal elastomer energy wells (modulo possible rescaling, see the discussion below).

\begin{proof}
The convergence of $K_n(a)$ to $K_{\infty}(a)$ follows from the pointwise convergence of $U(a)$ to the matrix $U_{\infty}(a)$ as $n\rightarrow \infty$ (see \eqref{Uconv}) and the fact that $\mathcal{P}_n \rightarrow O(2)$ as $n\rightarrow \infty$.

In order to observe the claimed identity, we
note that by the properties of the determinant and of $U_{\infty}(a)$, it holds
\begin{align*}
\bigcup\limits_{P\in O(2)} SO(2) P U_{\infty}(a) P^T  &\subset 
\{F \in \R^{2\times 2}: \ \det(F)=1, \ \lambda(F) = a, \ \mu(F) = a^{-1}\}\\
& =: K_{\infty}'(a).
\end{align*}
It hence remains to prove the reverse inclusion. Let $F \in K_{\infty}'(a)$. Then, by the polar decomposition $F = R_1 V$ for some $R_1\in SO(2)$ and some $V\in\R^{2\times 2}$ symmetric, positive definite with eigenvalues $a,a^{-1}$. Now, by the spectral theorem and the fact that $U_\infty$ is diagonal, there exists $R_2\in SO(2)$ such that $R_2U_\infty R_2^T = V$. As a consequence, $F = R_1R_2U_\infty R_2^T$, which concludes the proof.

The identity for $\nabla v$ follows directly from Proposition \ref{prop:n_infty}.
\end{proof}

In the sequel, we explain a precise sense in which \eqref{eq:limit_set} can be understood as the energy wells for a planar nematic liquid crystal elastomer differential inclusion. This allows us to view the deformation $v$ from Proposition \ref{prop:n_infty} and Corollary \ref{cor:iterate} as a microstructure arising in the modelling of certain planar deformations in nematic liquid crystal elastomers.

To this end, we begin by investigating planar solutions to the geometrically non-linear, nematic elastomer  differential inclusion \eqref{1901311846}. More precisely, we consider $u:B_1(0)\times [0,1] \rightarrow \R^3$ which is of the form
\begin{align*}
u(x_1,x_2,x_3) = (\tilde{u}(x_1,x_2),0) + \left[ \begin{pmatrix} 0 & 0 \\ 0 & r^{-\frac{1}{6}} \end{pmatrix}x \right]^T.
\end{align*}
Here $r^{-\frac{1}{3}}$ is one of the constants from \eqref{1901311846} and we assume that $\tilde{u}(x_1,x_2)=M'(x_1,x_2)$ on $\partial \Omega$ for some $M' \in \R^{2\times 2}$. Seeking an exactly stress-free deformation within the framework of the BWT model \eqref{1901311846}, the two eigenvalues of $\nabla' \tilde{u}$ are therefore determined by the differential inclusion $\nabla u \in \tilde{K}_{\infty}$ with $\tilde{K}_{\infty}$ as in \eqref{1902010046}. Here the notation $\nabla' \tilde{u}$ refers to the gradient of $\tilde{u}$ in the $x_1, x_2$ directions. Without loss of generality assuming that $r>1$ and with slight abuse of notation, the singular values are thus given by $\lambda_1= r^{-\frac{1}{6}}$ and $\lambda_2 = r^{\frac{1}{3}}$. 
In other words, in order to solve the differential inclusion  $\nabla u \in \tilde{K}_{\infty}(r)$, it is necessary and sufficient that
\begin{align}
\label{eq:K_2D}
\nabla' \tilde{u}  \in K_{2D}(r),
\end{align} 
where 
\begin{align}
\label{eq:K_2D_well}
 K_{2D}(r) :=\{ F\in \R^{2\times 2}: \ \lambda_1(F)= r^{-\frac{1}{6}}, \ \lambda_2(F) = r^{\frac{1}{3}}, \ \det(F) = r^{\frac{1}{6}}\}.
\end{align}
We note that the set in \eqref{eq:K_2D_well}   coincides with the set from Corollary \ref{lem:K_inf} up to a rescaling which modifies the determinant, i.e., $K_{2D}(r)= r^{\frac{1}{12}} K_{\infty}(r^{\frac{1}{4}})$.

By the theory of relaxation (see for instance \cite{DaM12,D}), interesting microstructures arise if 
\begin{align*}
M'\in \inte K^{qc}_{2D}(r):=\{F\in \R^{2\times 2}: \ r^{-\frac{1}{6}} < \lambda_1(\nabla \tilde{u}) \leq \lambda_2(\nabla \tilde{u})< r^{\frac{1}{3}}, \ \det(F) = r^{\frac{1}{6}}\}.
\end{align*}
In particular, we obtain that for $a=r^{\frac{1}{4}}$  the deformation $r^{\frac{1}{12}}v$ from Proposition \ref{prop:n_infty} and Corollary \ref{cor:iterate} is a solution to the differential inclusion \eqref{eq:K_2D} with a non-trivial microstructure.\\

Concluding our discussion on the geometrically non-linear theory, we present an example of a director field minimising the energy density of nematic elastomers in Figures \ref{1901112220}-\ref{fig:finite_n}.  
Here the planar deformation gradient $\nabla u(x)$ is obtained as an exact solution in the sense that we have $\nabla u\in K_{2D}(r)$, where
it is imagined to be the $2\times 2$ planar deformation associated with a full $3\times 3$ volume-preserving deformation. Consequently, the nematic elastomer is in planar expansion in all the deformed configurations for $a>1$.
The planar director field is taken in the form $\hat{n}(x)=(\cos\theta,\sin\theta,0)$ and it corresponds to the eigenvector associated with the largest eigenvalue of $\nabla u \nabla u^T$, in agreement with (\ref{1901312036}).
More exact constructions are displayed in Figure \ref{fig:finite_n} for large nematic anisotropies at finite $n$. These correspond to solutions
\begin{align*}
\nabla u \in K_n(a),
\end{align*}
which however can always also be interpreted as a nematic elastomer inclusion problem as
\begin{align*}
K_{n}(a) \subset r_n^{\frac{1}{12}} K_{2D}(r_n),
\end{align*}
where $r_n>0$ is a function of $n,a$.
Although an anisotropy parameter of the order $r_n=O(10^2)$ is non-physical, we report these solutions as they represent nice examples of the theory developed in this article showing large deformations and director rotation.

Observe that the solutions obtained for finite $n$ for the \textit{discrete} NLCEs model still survive as exact solutions of nematic elastomer configurations since $\mathcal{P}_n\subset O(2)$. A possible application of the discrete model of NLCEs thus obtained for finite $n$ is for benchmarking of large numerical simulations. Here the advantage is that the discrete modelling approach involves only a finite subsets of energy wells and has the potential to provide a faster and more stable energy minimisation with respect to the full isotropic NLCEs model.

\begin{figure}[!htbp]
\centering
\includegraphics[width=6.2cm,page=11]{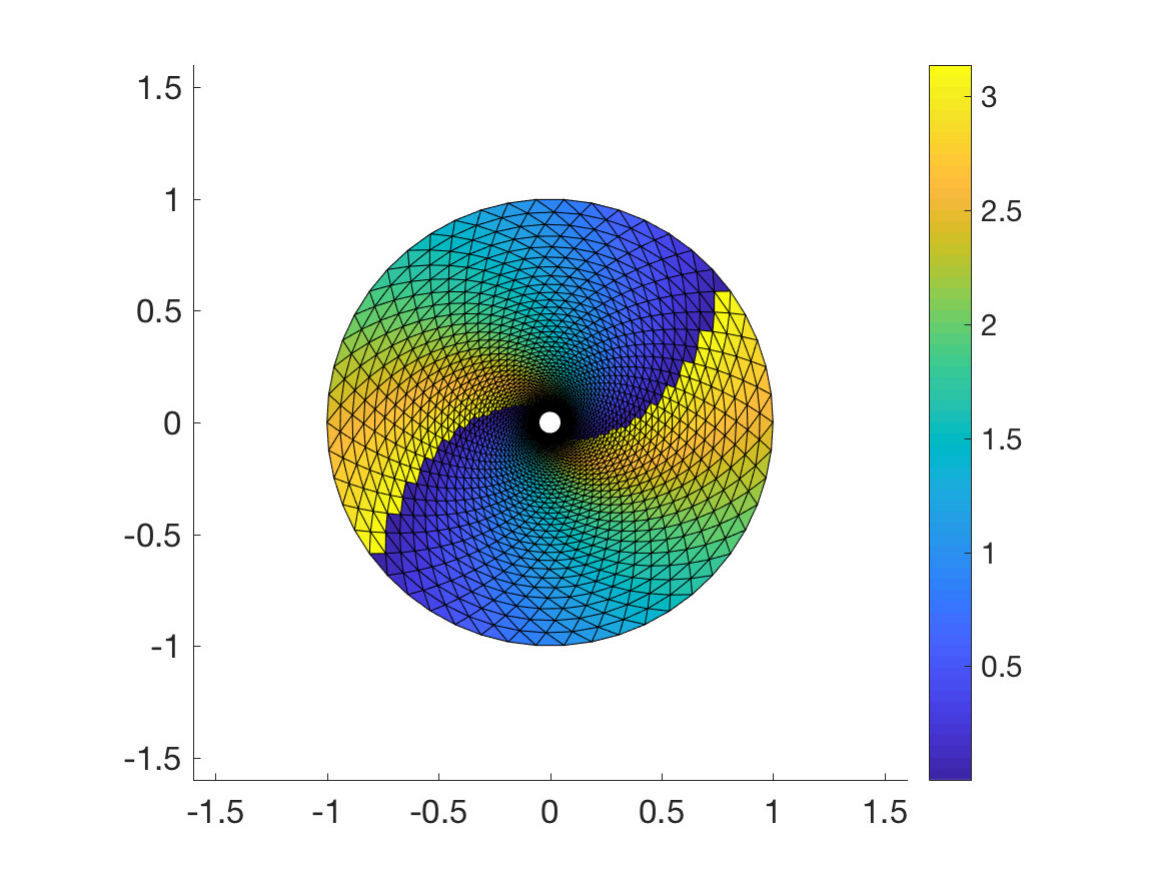}
\includegraphics[width=6.2cm,page=12]{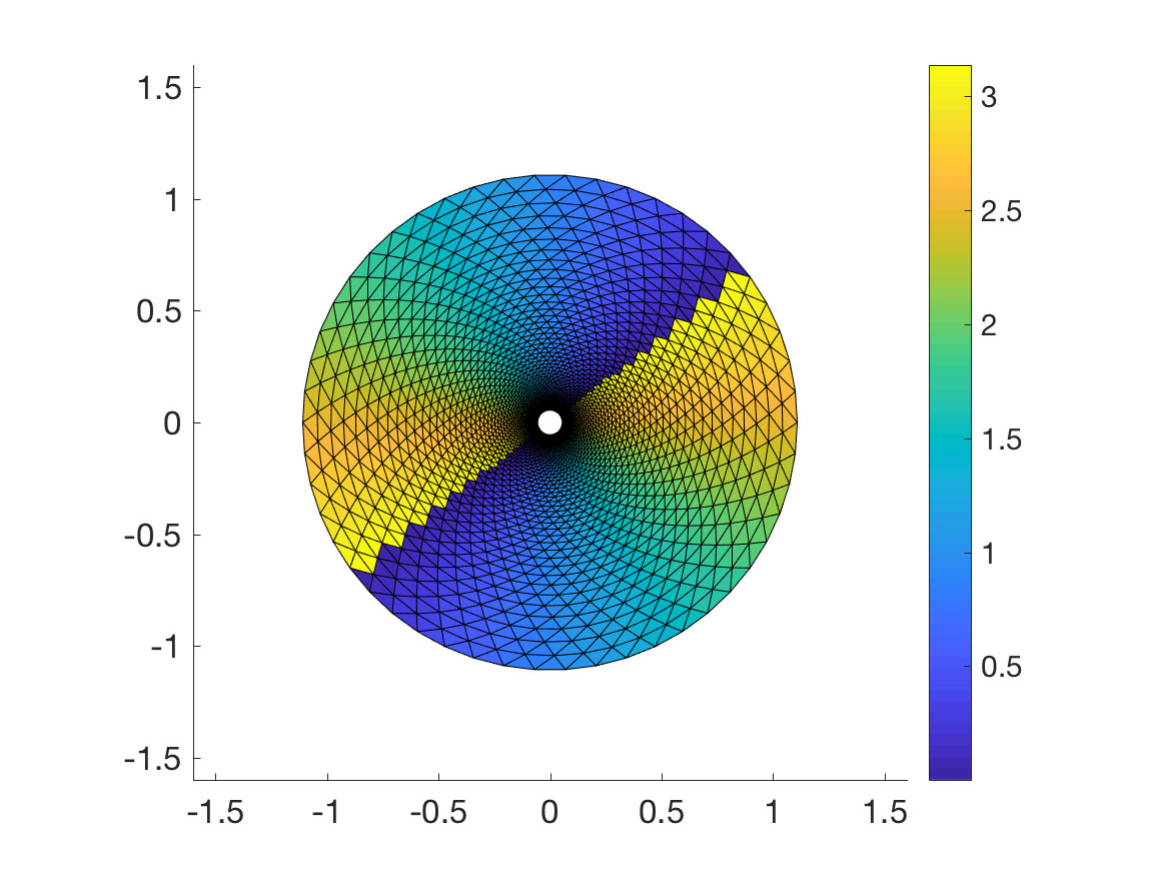}
\caption{
Example of planar director fields that minimise pointwise the energy density for nematic elastomers. Directors are parametrised as $\hat{n}=(\cos\theta,\sin\theta,0)$ and the value of $\theta$ is represented by means of a colormap. Recall due to the head-tail symmetry the orientation of the molecules is fully described by $\hat{n}\otimes \hat{n}$ and therefore there is no discontinuity in passing from $\theta=0$ to $\theta=\pi$.
Here $\alpha=0.35$, $n=50$ for which we have $r_n\approx 3.5$. The director field is displayed in the reference configuration (left) and in the deformed configuration (right). Observe the planar expansion as $r_n>1$.}\label{1901112220}
\end{figure}

\begin{figure}[!htbp]
\centering
\includegraphics[width=6.2cm,page=13]{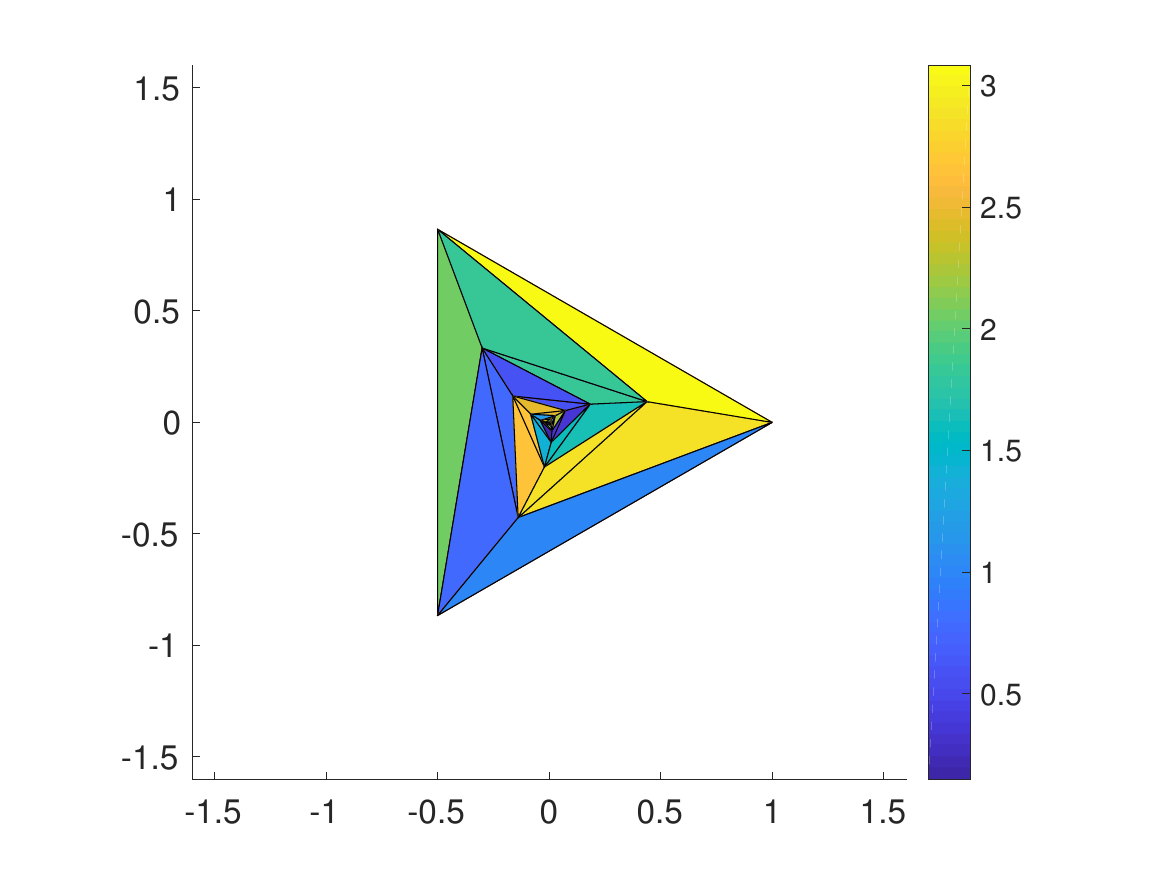}
\includegraphics[width=6.2cm,page=14]{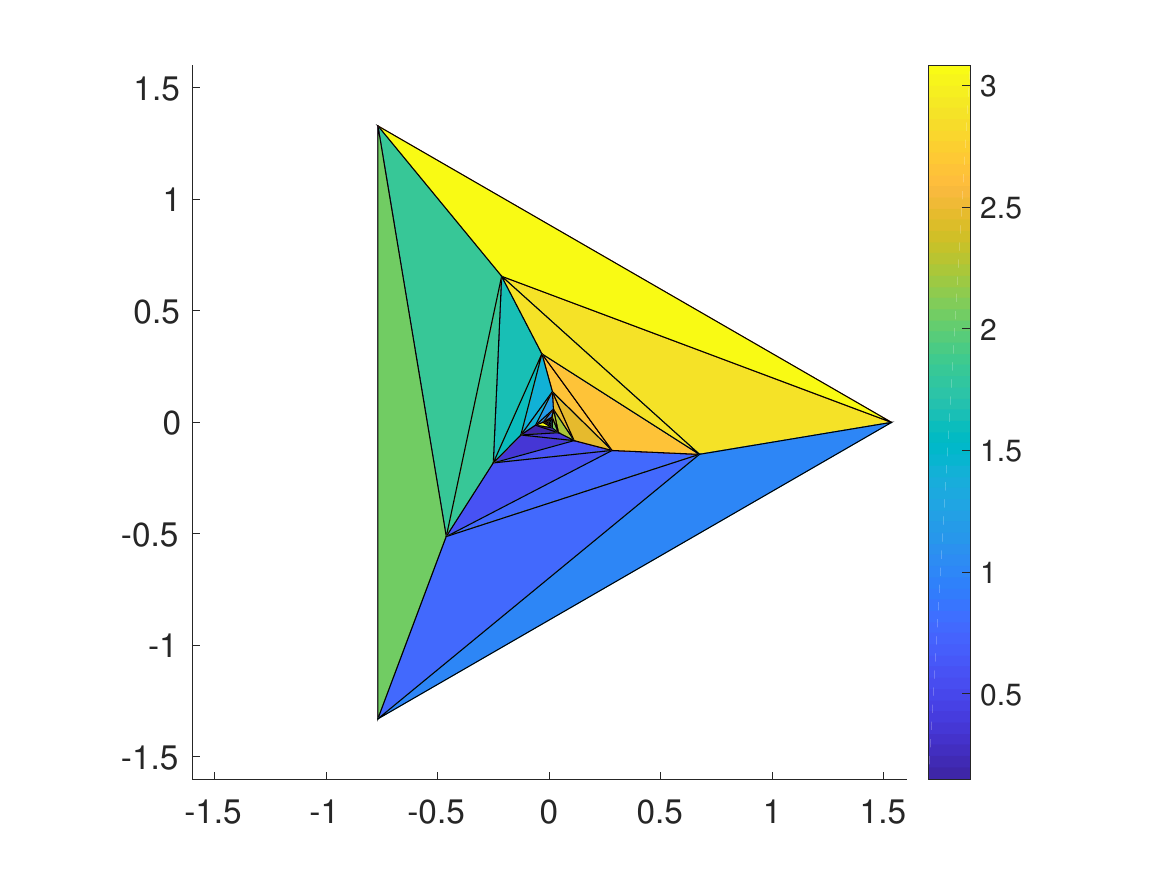}
\includegraphics[width=6.2cm,page=15]{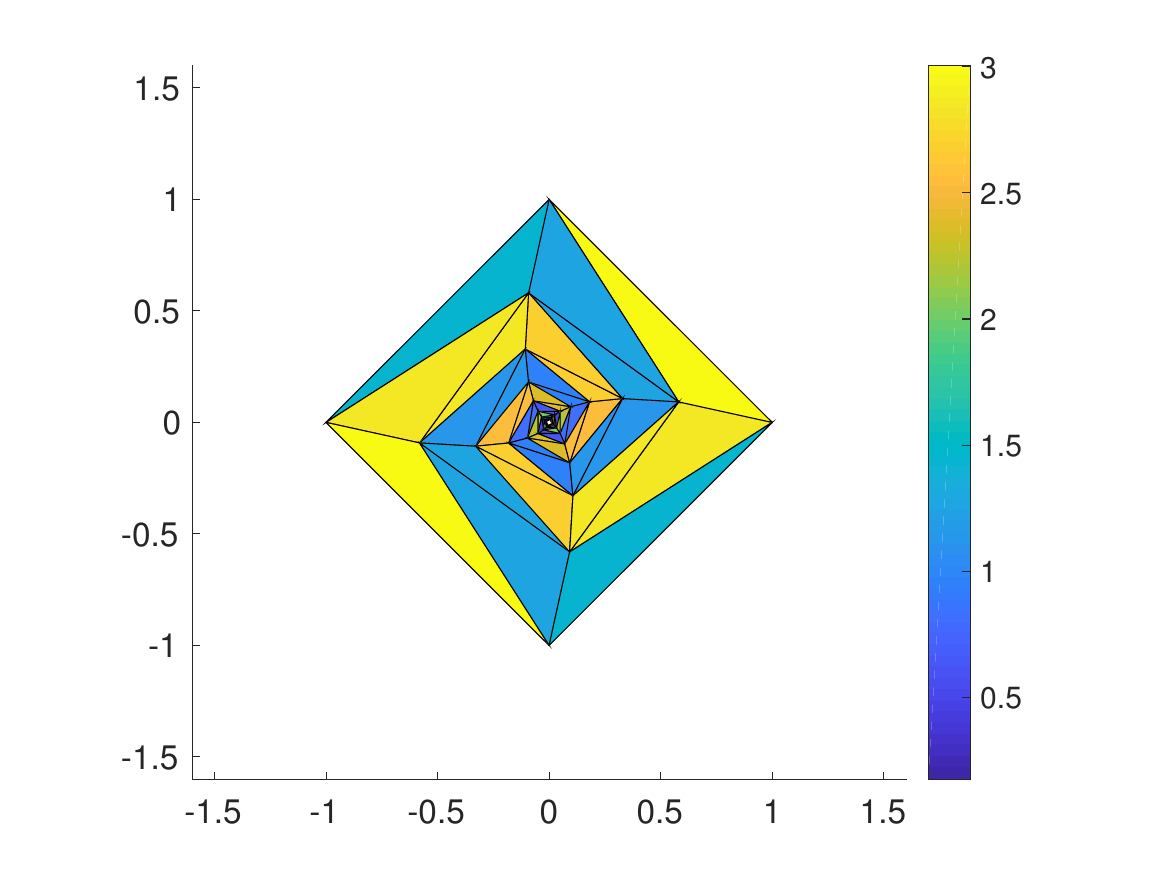}
\includegraphics[width=6.2cm,page=16]{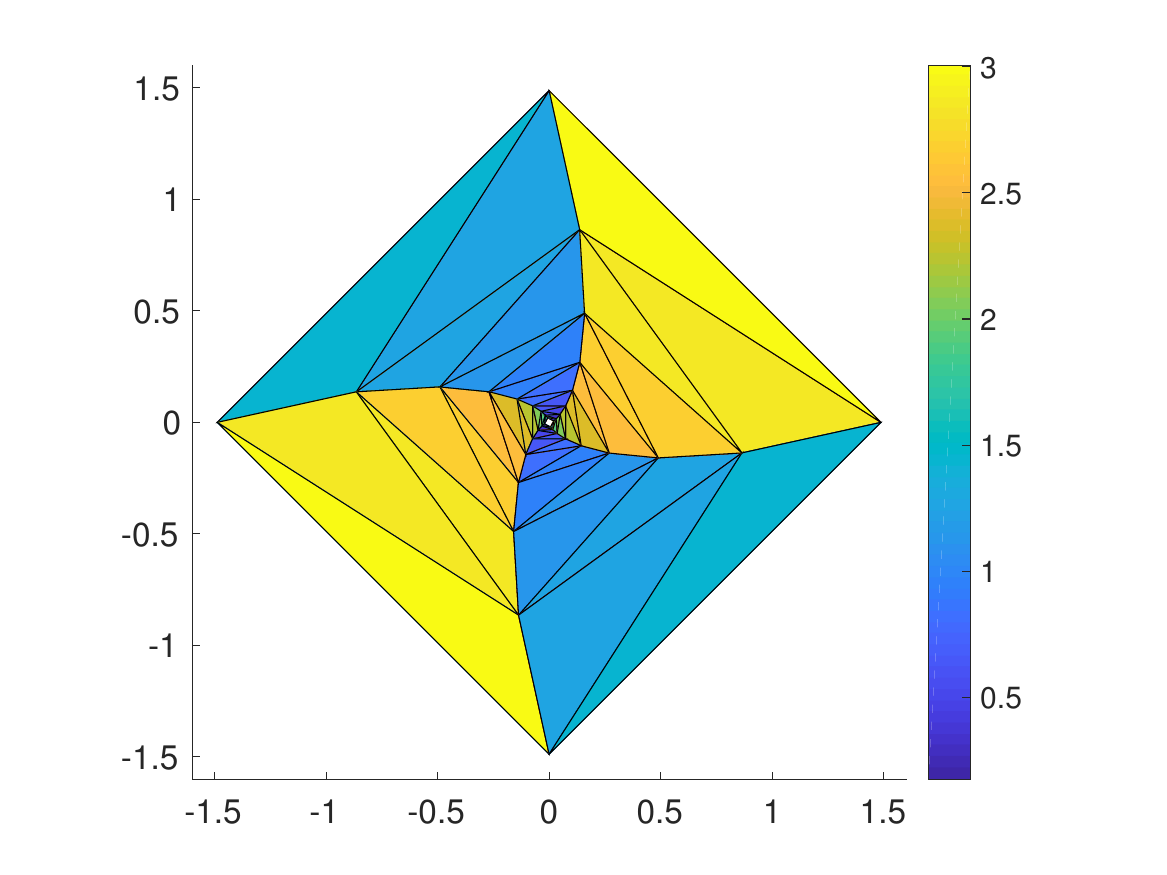}
\includegraphics[width=6.2cm,page=17]{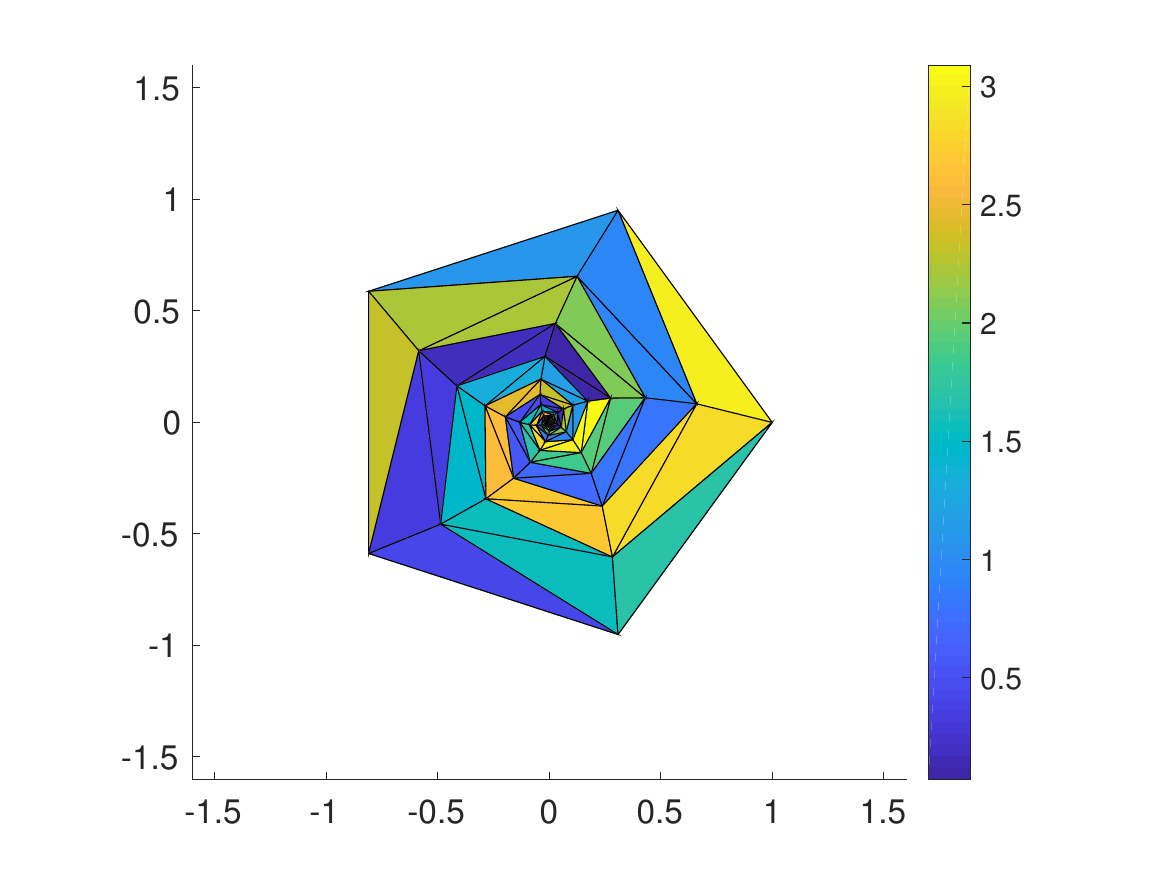}
\includegraphics[width=6.2cm,page=18]{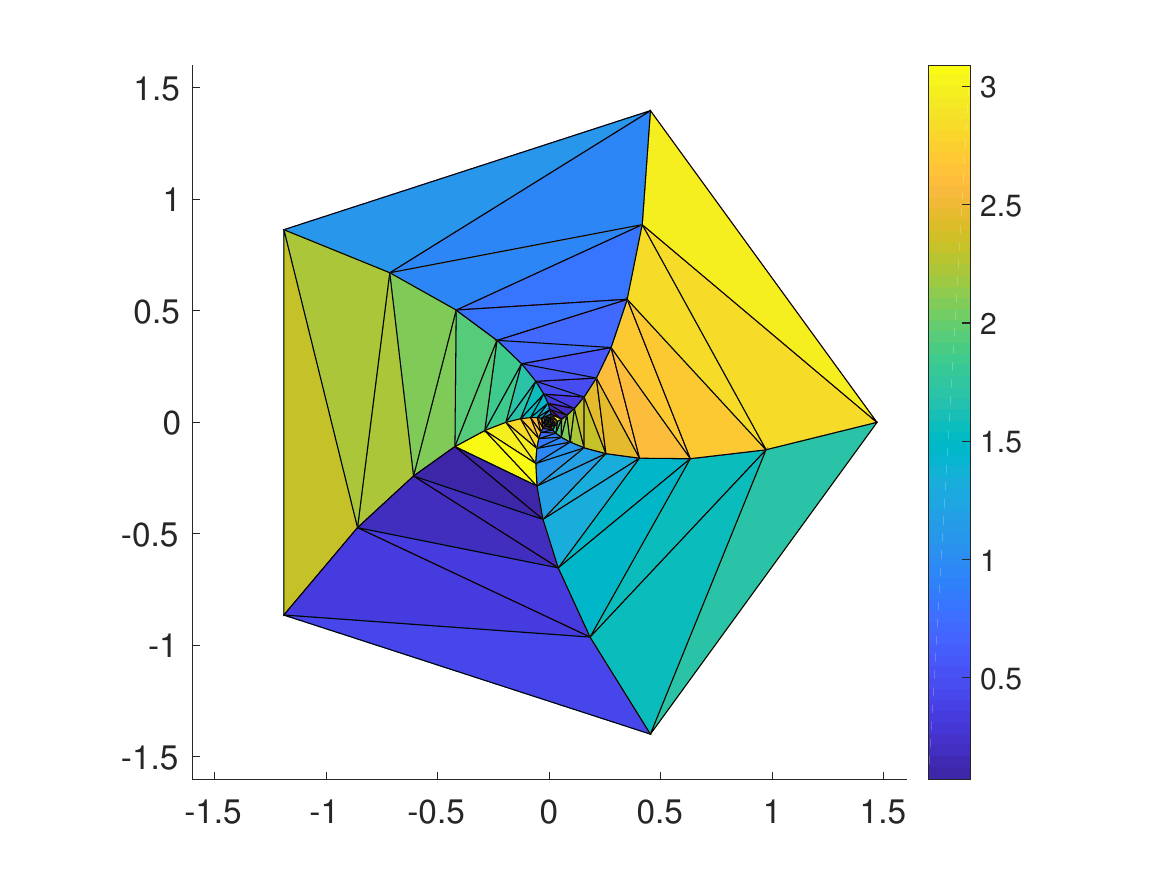}
\includegraphics[width=6.2cm,page=19]{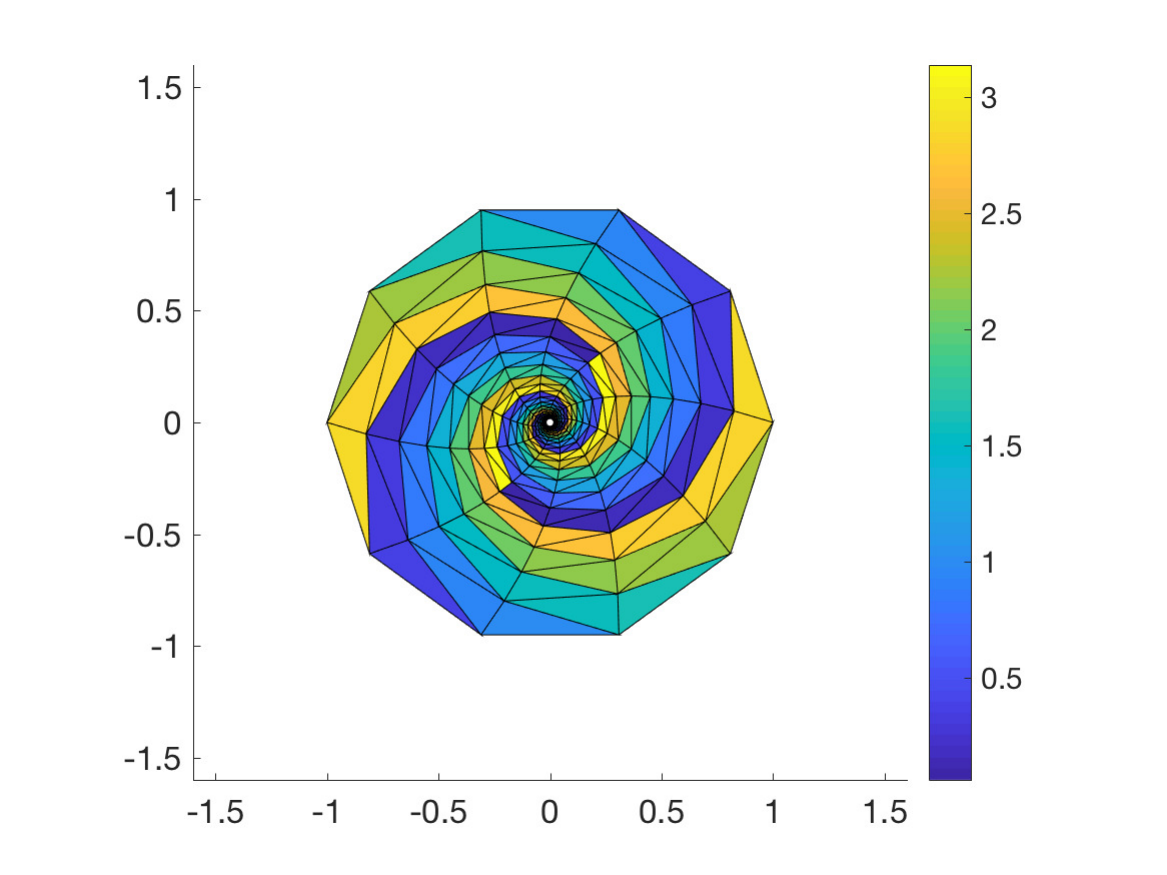}
\includegraphics[width=6.2cm,page=20]{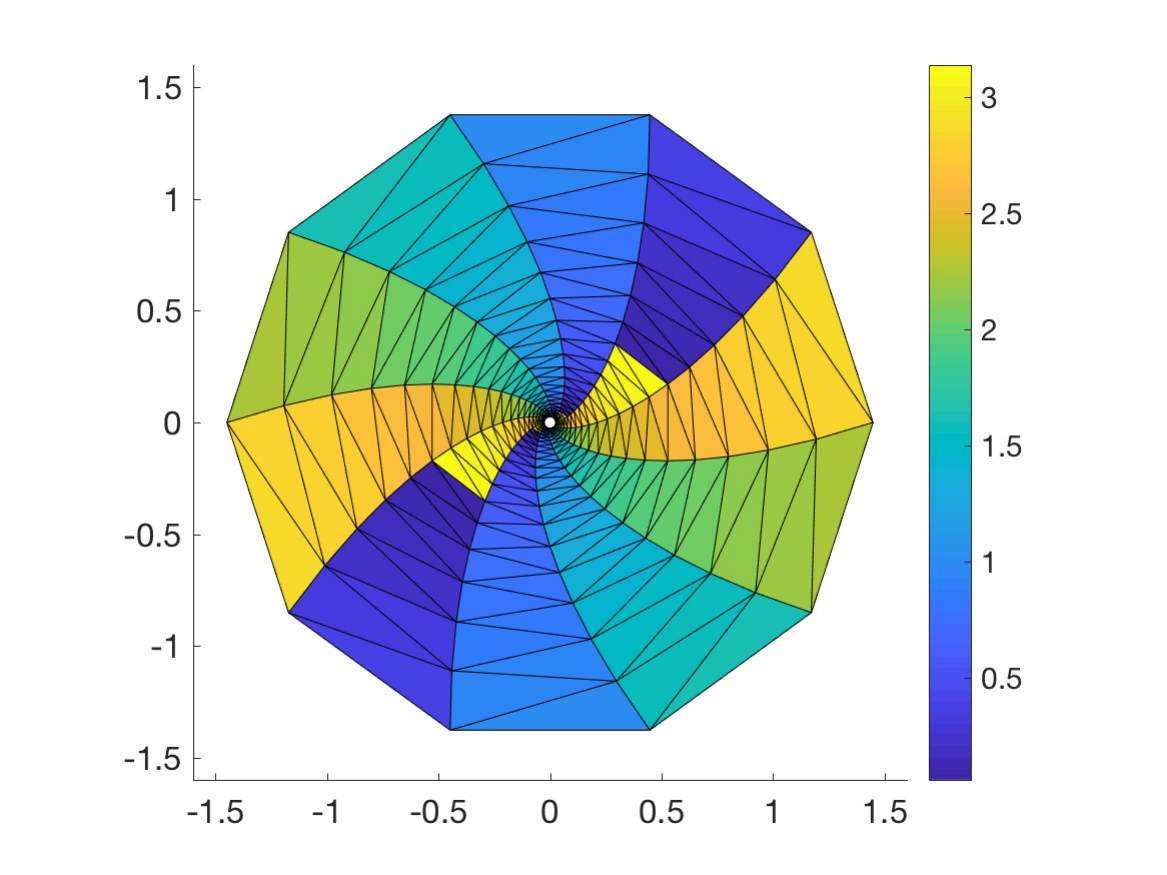}
\caption{
More examples of planar director fields obtained for large values of the nematic anisotropy parameter $r_n$. Here
$\alpha=0.1$. From top to bottom, $n=3$ ($r_n\approx 171$), $n=4$ ($r\approx 118$), $n=5$ ($r_n\approx 102$) and $n=10$ ($r_n\approx 85$) respectively. The director field is displayed in the reference configuration (left) and in the deformed configuration (right). Observe a large planar expansion as $r_n \gg 1$.}
\label{fig:finite_n}
\end{figure}

\section{Linearisation of the non-linear constructions}

\label{sec:lin_main}

In this section, we discuss the geometrically linearised (but physically non-linear) counterpart of the setting discussed in the previous sections. Here our main observations are the following: 
\begin{itemize}
\item First, we note that it is always possible to infer a \emph{linear analogue} of the geometrically non-linear Conti construction by linearisation. Motivated by the physically most relevant situation that $u(x) = x + \epsilon v(x)$ for some small $\epsilon$ and a function $v$ with $\nabla v$ controlled, in Section \ref{sec:lin}, we study the linearisation of our geometrically non-linear constructions around $\alpha = \frac{1}{2}$.
\item Moreover, we study the \emph{number of wells} involved in the geometrically linearised Conti type constructions. While in the geometrically non-linear setting already in the case of a single onion ring layer, it is necessary to work with a phase transformation with $2n$ wells if $n$ is odd (but only $n$ wells is $n$ is even), in the geometrically linearised case (linearised at $\alpha = \frac{1}{2}$) only $n$ wells are needed, independently of whether $n$ is odd or even (see Section \ref{sec:number_wells}).
\item We also consider the \emph{iterability} of the single onion ring constructions: In contrast to the geometrically non-linear setting, the geometrically linearised solutions (at $\alpha = \frac{1}{2}$) \emph{can} be iterated into solutions with multiple onion layers without having to include a larger number of wells. As already noted in the materials science literature the presence of ``disclinations'' \cite{KK91} hence is a purely \emph{geometrically non-linear} effect (see Section \ref{sec:finite_n}).
\item Finally, in Section \ref{sec:lin_infty}, similarly as in the geometrically non-linear set-up, we also discuss the \emph{limit} $n\rightarrow \infty$ and relate this to the analogous differential inclusions arising in the modelling of nematic liquid crystal elastomers. In particular, we show that our solutions exactly reproduce the model solution which had been derived in \cite{ADMDS15}.
\end{itemize}

\subsection{Linearisation}
\label{sec:lin}

We begin by deriving a geometrically linear Conti construction from the geometrically non-linear one by linearisation at $\alpha = \frac{1}{2}$. In order to simplify our presentation, we study the linearisation in the coordinates given by $e_{11}$ and $e_{11}^{\perp}$ (see Lemma \ref{lem:coord} below for a justification).

The linearisation of the wells is given by
\begin{align*}
E_j= \frac{d}{d\alpha}\left[ e(U_j(\alpha)) \right]|_{\alpha = \frac{1}{2}},
\end{align*}
where $e(M):= \frac{1}{2}(M+M^T)$ denotes the symmetrised part of a matrix $M\in
\R^{2\times 2}$ and $U_j(\alpha):=\nabla u|_{T_j}$ is the restriction of the piecewise constant function $\nabla u$ from Propositions \ref{prop:suff} or \ref{prop:n_infty} (which in particular depends on $\alpha$).
In particular, 
\begin{align}
\label{eq:E1}
E_1 =  \begin{pmatrix}1 & -  \frac{\cos(\phi_n/2)}{\sin(\phi_n/2)} \\ -  \frac{\cos(\phi_n/2)}{\sin(\phi_n/2)} & -1\end{pmatrix} ,
\end{align}
and
\begin{align*}
E_2 =  \begin{pmatrix} 1 &   \frac{\cos(\phi_n/2)}{\sin(\phi_n/2)} \\ \frac{\cos(\phi_n/2)}{\sin(\phi_n/2)} & -1\end{pmatrix} .
\end{align*}

In order to justify our linearisation (in the $\alpha$ dependent choice of coordinates), we note the following:

\begin{lem}
\label{lem:coord}
For each $\alpha\in (0,1)$ let $e_{11}, e_{11}^{\perp}$ denote the ($\alpha$ dependent) coordinates from Section \ref{sec:problem}. Then, we have
\begin{align*}
&\frac{d}{d\alpha} \left[ (e_{11} \ e_{11}^{\perp}) e(U_j(\alpha)) (e_{11} \ e_{11}^{\perp})^{T} \right]\big|_{\alpha = \frac{1}{2}}\\
&=  (e_{11} \ e_{11}^{\perp}) \big|_{\alpha = \frac{1}{2}} \left( \frac{d}{d\alpha}  e(U_j(\alpha)) \big|_{\alpha=\frac{1}{2}} \right) (e_{11} \ e_{11}^{\perp})^{T} \big|_{\alpha = \frac{1}{2}}.
\end{align*}
\end{lem}

As a consequence and as expected, it does not matter in which coordinates we consider the geometric linearisation of the problem at hand. Hence, in the sequel, without further comment, we will always consider the linearisation in the coordinates $(e_{11} \ e_{11}^{\perp})\big|_{\alpha = \frac{1}{2}}$ .

\begin{proof}
We show that for a general rotation $Q$ which depends differentiably on the parameter $\alpha$, we have
\begin{align*}
\frac{d}{d\alpha} \left[ Q e(U_j(\alpha)) Q^T \right]\big|_{\alpha = \frac{1}{2}}
=  Q \big|_{\alpha = \frac{1}{2}} \left( \frac{d}{d\alpha}  e(U_j(\alpha)) \big|_{\alpha=\frac{1}{2}} \right) Q^T \big|_{\alpha = \frac{1}{2}}.
\end{align*}
Indeed, this is a direct consequence of the product rule. Denoting derivatives with respect to $\alpha$ by a dash, we obtain
\begin{align}
\label{eq:deriv_alpha}
\begin{split}
\frac{d}{d\alpha} \left[ Q e(U_j(\alpha)) Q^T \right]\big|_{\alpha = \frac{1}{2}}
&= Q'\big|_{\alpha = \frac{1}{2}} e(U_j(\alpha))\big|_{\alpha = \frac{1}{2}} Q^T\big|_{\alpha = \frac{1}{2}}\\
& \quad + Q\big|_{\alpha = \frac{1}{2}} [e(U_j(\alpha))]'\big|_{\alpha = \frac{1}{2}} Q^T\big|_{\alpha = \frac{1}{2}} \\
& \quad + Q\big|_{\alpha = \frac{1}{2}} e(U_j(\alpha))\big|_{\alpha = \frac{1}{2}} (Q^T)'\big|_{\alpha = \frac{1}{2}}.
 \end{split}
\end{align}
Now, using that
\begin{align*}
Q'\big|_{\alpha = \frac{1}{2}} =c \begin{pmatrix} 0 & 1 \\ -1 & 0 \end{pmatrix} Q\big|_{\alpha = \frac{1}{2}},
\end{align*}
and the fact that $[e(U_j(\alpha))]'\big|_{\alpha = \frac{1}{2}} \in \frac{1}{\cos(\frac{\pi}{n})} (O(2)\setminus SO(2))\cap \R^{2\times 2}_{sym}$, implies by the commutation relations for rotations and reflections that the first and second contributions in \eqref{eq:deriv_alpha} cancel. Thus, we obtain the desired result.
\end{proof}

As a direct consequence of the non-linear constructions from the previous section, we obtain the following geometrically linearised Conti constructions:

\begin{prop}
\label{cor:rank_one}
Let $u_{\alpha}: \R^2 \rightarrow \R^2$ be a non-linear deformation associated with a non-linear Conti construction with $\alpha \in (0,1)$. Then the function $v_0:=\frac{d}{d\alpha} u_{\alpha} \big |_{\alpha=\frac{1}{2}}: \R^2 \rightarrow \R^2$ is a displacement vector field for a geometrically linear Conti construction, i.e. it is a piecewise affine, continuous vector field, which has constant gradient on the triangles $T_1,\dots, T_{2n}$. The symmetrised gradients involved in the linearised construction are given by the matrices $E_1,\dots,E_{2n}$ corresponding to the linearisations and symmetrisations of $U_1(a(\alpha)),\dots, U_{2n}(a(\alpha))$. In the exterior of the polygon $\Omega_n^E$ and in the polygon $\Omega_n^I$, the displacement gradient is a skew matrix.
\end{prop}

\begin{proof}
We first note that, by the explicit expressions from Section \ref{sec:nonlinear_n_well} for any $\alpha>0$ the deformation $u_{\alpha}$ depends differentiably on the parameter $\alpha$. Thus, in order to prove that $v_0$ is a displacement for the geometrically linear Conti construction, it suffices to show that $v_0$ is continuous along the sides of the triangles $T_1,\dots,T_{2n}$. Let $\ell_{\alpha}:\R^2 \rightarrow \R^2$ be a line segment with normal $\nu_{\alpha} \in \R^2$ describing one of the edges of the triangles $T_1,\dots,T_{2n}$. Let $x \in \ell_{\alpha}$ and denote by $\ell_{\alpha}^+(x)$ denote the limit of points $y\in \R^2$ with $y\cdot\nu_{\alpha} \geq 0 $ and $y \rightarrow x$. Define $\ell_{\alpha}^{-}(x)$ similarly. Then, by continuity of $u_{\alpha}$ for all $\alpha>0$ we in particular have that for all $x\in \ell_{\alpha}$
\begin{align*}
u_{\alpha}(\ell_{\alpha}^+(x)) - u_{\alpha}(\ell_{\alpha}^-(x))=0.
\end{align*}
As a consequence, 
\begin{align*}
\left.\frac{d}{d\alpha}\left[ u_{\alpha}(\ell_{\alpha}^+(x)) - u_{\alpha}(\ell_{\alpha}^-(x))  \right]  \right|_{\alpha = \frac{1}{2}}  = 0.
\end{align*}
By the product rule this however turns into
\begin{align}
\label{eq:cont}
0=\left[ v_0(\ell_1^+(x)) -  v_0 (\ell_1^-(x)) \right] + \left[ u_1(x)(\ell_{1,+}'(x))- u_1(x)(\ell_{1,-}'(x))\right],
\end{align}
where $\ell_{1,\pm}'(x) := \left. \left[ \frac{d}{\alpha} \ell_{\alpha}^{\pm}(x) \right] \right|_{\alpha = \frac{1}{2}}$. By the $C^1$ continuity of $\ell_{\alpha}(x)$ we however have $\ell_{1,+}'(x) = \ell_{1,-}'(x)$. Hence, the continuity of $u_1$ implies that \eqref{eq:cont} turns into
\begin{align*}
0=\left[ v_0(\ell_1^+(x)) -  v_0 (\ell_1^-(x)) \right].
\end{align*}
This is the claimed continuity of $v_0$ along the edges of the triangles.
\end{proof}

\begin{rmk}
As a direct consequence of the derivation of the linear displacement $u_{\alpha}$ from the non-linear constructions from Section \ref{sec:nonlinear_n_well}, we also obtain the symmetrised rank-one directions from the rank-one directions of the non-linear problem:
Let $U_1 - U_2 = \frac{\cos(\phi_n/2)}{\sin(\phi_n/2)} e_1 \otimes e_2$. Then the matrices $E_1,E_2$ obtained as above, satisfy
\begin{align*}
E_1-E_2 = \frac{\cos(\phi_n/2)}{\sin(\phi_n/2)}e( e_1 \otimes e_2).
\end{align*}
\end{rmk}

\subsection{Remarks on the number of wells}
\label{sec:number_wells}

We now seek to investigate the geometrically linearised Conti type constructions from Proposition \ref{cor:rank_one} in more detail. In particular, it will turn out that in contrast to the geometrically non-linear setting, in the geometrically linearised setting only $n$ wells are needed for a single onion layer construction, \emph{independently} of whether $n$ is odd or even (we recall that in the geometrically non-linear setting $2n$ wells were needed if $n$ was odd). 
This follows from Corollary \ref{cor:rank_one}, the values of the strains which are used there and the interaction of the linearisation with the symmetry group $\mathcal{P}_n$. Although this also directly follows by combining the results from Section \ref{sec:finite_n} with the linearisation procedure, we give an independent proof which highlights the structure of the linear wells.
In the next section, we will then study the iterability of the single onion ring layer constructions in the geometrically linearised setting.

\begin{lem}
\label{lem:lin_vs_nonlin}
Let $n\in \N$, $n\geq 3$ be odd and let 
\begin{align}
\label{eq:dif_in_lin}
K= \{ R E_1 R^T: \ R \in \hat{\mathcal{P}}_n\},
\end{align}
where $\hat{\mathcal{P}}_n:= \hat{\mathcal{R}}_1^n \cup \hat{\mathcal{R}}_2^n$ is defined as in \eqref{eq:hatP}.
Then the single layer Conti construction obtained in Proposition \ref{cor:rank_one} is such that exactly $n$ different strains are used. More generally, the set of linearised energy wells $K$ consists of exactly $n$ different wells, i.e. $\# K = n$. 
\end{lem}

\begin{rmk}
\label{rmk:hatcoord}
Here and in the sequel, we work with the symmetry group $\hat{\mathcal{P}}_n$ instead of the group $\mathcal{P}_n$ since we are considering the problem in the $e_{11}, e_{11}^{\perp}$ coordinates.
\end{rmk}

\begin{figure}[t]
\centering
\includegraphics[scale=0.7,page=21]{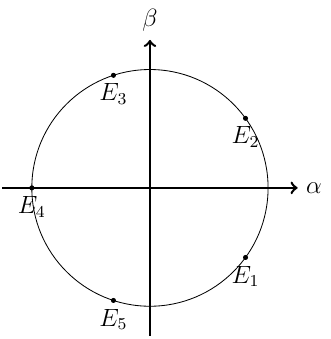}
\caption[The geometrically linearized $n$-well problem for $n=5$.]{The geometrically linearised $n$-well problem for $n=5$. As exploited in the proof of Lemma \ref{lem:lin_vs_nonlin}, the set $K$ can be parametrised through a vector $(\alpha,\beta)$, i.e. each element of $K$ is of the form 
$\begin{pmatrix} \alpha & \beta \\ \beta & -\alpha \end{pmatrix}$ with $\alpha^2 + \beta^2 = const$. 
Hence, by the identities from the properties (i), (ii) in the proof of Lemma \ref{lem:lin_vs_nonlin}, it is possible to visualise the set of wells as a regular $n$-gon as illustrated in this figure.}
\label{fig:wells_lin}
\end{figure}

\begin{proof}
We first prove that $\# K \leq n$, the fact that $\#K = n$ will be a consequence of the argument for this.\\

The symmetry group $\hat{\mathcal{P}}_n$ acts on $K$ by conjugation. In particular, $K$ is obtained as the orbit of $E_1$ under conjugation with elements of $\hat{\mathcal{P}}_n$. As $\hat{\mathcal{P}}_n \subset O(2)$, we more generally consider the conjugation class of the matrix $\begin{pmatrix}1 & d \\ d & -1 \end{pmatrix}$ for $d\in \R$ (which is of the same structure as $E_1$) under $O(2)$. 

Since
\begin{align*}
O(2)& = SO(2)\cup \begin{pmatrix} -1 & 0 \\ 0 & 1\end{pmatrix} SO(2)\\
&=\left\{
\begin{pmatrix}
a & b\\ -b & a
\end{pmatrix} \cup \begin{pmatrix}
a & -b\\ -b &- a
\end{pmatrix}: \ a,b\in \R \mbox{ such that } a^2 + b^2 = 1
 \right\},
\end{align*}
on the one hand, we compute
\begin{align}
\label{eq:mat1}
\begin{pmatrix}
a & b\\ -b & a
\end{pmatrix}
\begin{pmatrix}
1 & d\\
d & -1
\end{pmatrix}
\begin{pmatrix}
a & -b \\ b & a
\end{pmatrix}
= \begin{pmatrix}
a^2 + 2abd -b^2 & -2ab +a^2d - b^2 d \\
-2ab +a^2d - b^2 d & b^2-2abd-a^2
\end{pmatrix}.
\end{align}
On the other hand, we also have
\begin{align}
\label{eq:mat2}
\begin{pmatrix}
a & b\\ b & -a
\end{pmatrix}
\begin{pmatrix}
1 & d\\
d & -1
\end{pmatrix}
\begin{pmatrix}
a & b \\ b & -a
\end{pmatrix}
= \begin{pmatrix}
a^2 + 2abd -b^2 & 2ab -a^2d + b^2 d \\
2ab -a^2d + b^2 d & b^2-2abd-a^2
\end{pmatrix}.
\end{align}
Comparing the matrix in \eqref{eq:mat1} with the one in \eqref{eq:mat2}, we note that the diagonal entries agree, while the off-diagonal ones deviate by a sign. 
Letting $\hat{\mathcal{R}}_1^n:= \hat{\mathcal{P}}_n\cap SO(2)$ and $\hat{\mathcal{R}}_2^n = \hat{\mathcal{P}}_n \setminus \hat{\mathcal{R}}^n_1$, we study the orbit of $E_1$ under the action of $\hat{\mathcal{R}}_1^n$. It has the following properties:
\begin{itemize}
\item[(i)] The orbit of $E_1$ under $\hat{\mathcal{R}}_1^n$ forms a regular $n$-gon in trace-free strain space parametrised as matrices of the form
\begin{align*}
\begin{pmatrix}
\alpha &  \beta \\ \beta & -\alpha
\end{pmatrix}, \ \alpha, \beta \in \R,
\end{align*}
see Figure \ref{fig:wells_lin}.
\item[(ii)] For $c=\frac{1}{\cos(\frac{\pi}{n})}$ the matrix
$c\begin{pmatrix} -1 & 0 \\ 0 & 1 \end{pmatrix}$ 
is an element of this $n$-gon. 
\end{itemize}
Both properties (i) and (ii) follow from trigonometric identities: For (i) we note that as $\begin{pmatrix} \alpha & \beta \\ \beta & - \alpha \end{pmatrix} \in (\alpha^2+\beta^2)(O(2)\setminus SO(2))$ by the commutation relations for rotations and reflections, we have
\begin{align}
\label{eq:strain_space_action}
\begin{split}
&\begin{pmatrix} \cos(\varphi) & - \sin(\varphi) \\ \sin(\varphi) & \cos(\varphi) \end{pmatrix}
\begin{pmatrix} \alpha & \beta \\ \beta & - \alpha \end{pmatrix}
\begin{pmatrix} 
\cos(\varphi) & \sin(\varphi)\\
-\sin(\varphi) & \cos(\varphi)
\end{pmatrix}\\
&= \begin{pmatrix} 
\cos(2\varphi) & -\sin(2\varphi)\\
\sin(2\varphi) & \cos(2\varphi)
\end{pmatrix}
\begin{pmatrix} \alpha & \beta \\ \beta & - \alpha \end{pmatrix}.
\end{split}
\end{align} 
Hence, conjugating a matrix $\begin{pmatrix} \alpha & \beta \\ \beta & - \alpha \end{pmatrix}$ by a rotation with angle $\varphi$ just rotates the matrix $\begin{pmatrix} \alpha & \beta \\ \beta & - \alpha \end{pmatrix}$ by the angle $2\varphi$. As a consequence, we note that as $n$ is odd, the orbit of $E_1$ under $\mathcal{R}_1^n$ is exactly given by a regular $n$-gon (as starting from $E_1$ we first reach all elements of the orbit which are at the even lattice sites of the $n$-gon with respect to the starting point $E_1$ and then after continuing to rotate, we also obtain the odd ones).

In order to deduce the second property (ii), we first study under which conditions the off-diagonal entry in \eqref{eq:mat1} vanishes. In order to simplify notation, we set $a=\cos(\varphi)$, $b=\sin(\varphi)$ with $\varphi= \frac{2\pi j}{n}$, $j\in \Z$, and $d= \cot(\phi_n/2)$, and note that then
\begin{align}
\label{eq:off_diag}
2 a b - a^2 d +b^2 d = \frac{1}{\cos\left(\frac{\pi}{n}\right)} \cos\left( \frac{\pi-4j\pi}{n} \right).
\end{align}
In order to prove the claim in (ii), we search for values of $j$ such that this expression vanishes. Hence, we seek an integer $j$ such that
$1-4j \in n \Z$. This is solved by
\begin{align*}
j= \left\{
\begin{array}{ll}
\frac{n+1}{4}, &\mbox{ if } n \equiv 3 \ (\mbox{mod } 4),\\
 \frac{3n+1}{4}, &\mbox{ if } n \equiv 1 \ (\mbox{mod } 4).
\end{array}
 \right.
\end{align*}
The claim (i) then follows as the expression $a^2-2abd-b^2$ in \eqref{eq:mat1} turns into $\frac{\cos(\frac{\pi-4j\pi}{n})}{\cos(\frac{\pi}{n})}= \frac{-1}{\cos(\frac{\pi}{n})}$.

With the properties (i), (ii) at hand, by the symmetry of the $n$-gon, we infer that the orbit of $E_1$ under the action of the group $\hat{\mathcal{R}}_1^n$ contains the matrix $\begin{pmatrix} g & -f \\ -f & -g
\end{pmatrix}$, iff it contains the matrix 
$\begin{pmatrix} g & f \\ f & -g
\end{pmatrix}$. 
In particular, this implies that if a matrix of the form \eqref{eq:mat1} is contained in the orbit of $E_1$ under $\hat{\mathcal{R}}_1^n$, also the corresponding matrix in \eqref{eq:mat2} is already contained in the orbit of $E_1$ under $\hat{\mathcal{R}}_1^n$. As a consequence, the orbit of $E_1$ under $\hat{\mathcal{R}}_2^n$ does not contain new information and $\# K \leq n$. 

The observation that $\# K \geq n$ is a direct consequence of property (i) from above, which thus yields $\# K = n$ and which concludes the argument.
\end{proof}

\subsection{Constructions for finite $n$ in the geometrically linear framework}
\label{sec:finite_n}

In this section, we discuss the concatenated structures that are obtained from linearising the geometrically non-linear $n$-well constructions from Section \ref{sec:iteration_layer} for a finite value of $n$ at $\alpha = \frac{1}{2}$. As a direct consequence of the properties of the geometrically non-linear deformations from Section \ref{sec:nonlinear_n_well}, we obtain the following facts:

\begin{itemize}
\item[(i)] The solutions in each ``onion layer ring" can be iterated in such a way that the overall construction only involves $n$ symmetrised deformation gradients. It comes from a single phase transformation.
\item[(ii)] The resulting iterated structures are highly symmetric and recover and generalise the experimentally observed tripole star type deformations (see the discussion in Section \ref{sec:stars}). In particular, the incompatibility of these patterns is a purely non-linear effect, which is not captured by the linearised theory.
\end{itemize}

\begin{cor}
\label{prop:star_def}
Let $v_{n,\alpha}: \Omega_n \rightarrow \R^2$ denote the deformations constructed in Section \ref{sec:iteration_layer} (for $\alpha \in (0,1)$). Then, the deformations $\tilde{v}_n(x):= \frac{d}{d\alpha} v_{n,\alpha}(x)|_{\alpha =\frac{1}{2}}$ are exactly stress-free deformations which attain only $n$ values for their symmetrised deformation gradients which are all related by the action of the symmetry group, i.e. for almost every $x\in \Omega_n$
\begin{align*}
e(\nabla \tilde{v}_{n})(x) \in \left\{E_1,\dots,E_n\right\}:=\left\{Q_j E_1 Q_j^T:  \ j \in \{0,\dots,n-1\}\right\},
\end{align*}
where as in Section \ref{sec:nonlinear_n_well} we use the notation $Q_j:=Q\left( \frac{2\pi}{n} j \right)$ and $Q(\gamma)= \begin{pmatrix} \cos(\gamma)& -\sin(\gamma) \\ \sin(\gamma) & \cos(\gamma)\end{pmatrix}$.
Moreover, the following symmetry assertions hold true:
\begin{itemize}
\item[(i)] If $n=2k+1, \ k \in \N,$ we have for all $j \in\{1,\dots,n\}$
\begin{align*}
Q_{\frac{n-1}{2}} P_0 E_j P_0 Q_{\frac{n-1}{2}} ^T = E_j ,
\end{align*}
and if $n=2k$, $k \in \N$, we have for all $j \in\{1,\dots,n\}$
\begin{align*}
Q_{\frac{n}{2}} E_j Q_{\frac{n}{2}}^T = E_j.
\end{align*}
\item[(ii)] For all $j\in\{1,\dots,n\}$ we have
\begin{align*}
R_{\frac{1}{2}} E_{j-1} R_{-\frac{1}{2}} = Q_1 E_j Q_1^T .
\end{align*}
\end{itemize}
\end{cor}

Although this result is a direct consequence of the corresponding properties of the geometrically non-linear problem (see Section \ref{sec:finite_n}), we reprove these here, as the geometrically linear setting allows for significant computational simplifications compared to the geometrically non-linear situation.

\begin{proof}
The fact that 
\begin{align*}
e(\nabla \tilde{v}_n)\in \{E_1,\dots,E_n\}:=\left\{Q_j E_1 Q_j^T: \ j \in \{0,\dots,n-1\}\right\} 
\end{align*}
follows from the observation that the symmetrised gradients are obtained by linearisation of the iterated non-linear construction (here $Q_{\alpha}$ denotes the rotation from Section \ref{sec:finite_n}). Indeed, by the same considerations as in Lemma \ref{lem:coord} we infer that
\begin{align*}
&\frac{d}{d\alpha} \left( Q_{\alpha} e(\nabla v_{\alpha,n}) Q_{\alpha}^T \right)|_{\alpha = \frac{1}{2}}\\
&= (Q_{\alpha}' e(\nabla v_{\alpha,n}) Q_{\alpha}^T)|_{\alpha = \frac{1}{2}}
+ (Q_{\alpha} e(\nabla v_{\alpha,n})' Q_{\alpha}^T)|_{\alpha = \frac{1}{2}}
+ (Q_{\alpha} e(\nabla v_{\alpha,n}) (Q_{\alpha}')^T)|_{\alpha = \frac{1}{2}}\\
&= (Q_{\alpha}')|_{\alpha = \frac{1}{2}} e(\nabla v_{\frac{1}{2},n}) R_{\frac{1}{2}}^T
+ R_{\frac{1}{2}} e(\nabla \tilde{v}_{\frac{1}{2},n}) R_{\frac{1}{2}}^T
+ R_{\frac{1}{2}} e(\nabla v_{\frac{1}{2},n}) (Q_{\alpha}')^T|_{\alpha = \frac{1}{2}}\\
&= R_{\frac{1}{2}} e(\nabla \tilde{v}_{\frac{1}{2},n}) R_{\frac{1}{2}}^T.
\end{align*}
Here the dash denotes differentiation with respect to $\alpha$; moreover, we used that 
\begin{align*}
e(\nabla v_{\frac{1}{2},n})= Id, \ Q_{\alpha}' = \begin{pmatrix} 0 & 1\\ -1 & 0 \end{pmatrix}Q_{\alpha}. 
\end{align*}
As $R_{\frac{1}{2}} \in \hat{\mathcal{P}}_n$, this proves the claim on the inclusion.

In order to prove (i), by symmetry it suffices to prove the claim for $j=1$ and $j=2$. Since $Q_{\frac{n}{2}}= Q(\pi)$, the result is straightforward for $j=2$. We thus focus on the case $j=1$ for which we need to prove that
\begin{align*}
Q_{\frac{n-1}{2}} P_0 E_1 P_0 Q_{\frac{n-1}{2}}^T= E_1.
\end{align*}
This however follows from the following observations:
\begin{itemize} 
\item By the explicit form of $E_1$, we have $E_1 \in \frac{1}{\sin(\frac{\pi}{n})}SO(2)$, whence by the  commutation relations for reflections and rotations,
\begin{align*}
P_0 E_1 P_0 = E_2,
\end{align*}
(the action of $P_0$ just flips the sign in the off-diagonal component). 
\item By a similar reasoning (see \ref{eq:strain_space_action}) we then also obtain that
\begin{align*}
Q_{\frac{n-1}{2}} E_2 Q_{\frac{n-1}{2}}^T
= Q_{n-1} E_2 .
\end{align*}
\item
By the structure of the set of $E_j$, we however have $Q_{n-1} E_2 = E_1$ (more generally, we have $E_j = Q_j  E_1$ for all odd $j$).
\end{itemize}
As a consequence, by combining the previous observations
\begin{align*}
Q_{\frac{n-1}{2}} P_0 E_1 P_0  Q_{\frac{n-1}{2}}^T
&= Q_{\frac{n-1}{2}} E_2  Q_{\frac{n-1}{2}}^T
= Q_{n-1} E_2 = E_1,
\end{align*}
which yields the desired result.

Finally, we provide the argument for (ii): Again we consider only the case $j=1$ and $j=2$. Considering first the case $j=1$, we note that 
\begin{align*}
R_{\frac{1}{2}} E_2 R_{-\frac{1}{2}} = R_{\frac{1}{2}}P_0 E_1 P_0 R_{-\frac{1}{2}}
= Q_{\frac{1}{2}}P_0 E_1 P_0 Q_{\frac{1}{2}}^T.
\end{align*}
It hence suffices to prove that
\begin{align*}
Q_1 E_1 Q_1^T
= Q_{\frac{1}{2}} P_0 E_1 P_0 Q_{\frac{1}{2}}^T.
\end{align*}
This however is equivalent to
\begin{align*}
Q_1 E_1 Q_1^T
= P_0 E_1 P_0 = E_2 .
\end{align*}
Since by \eqref{eq:strain_space_action}, we have 
\begin{align*}
Q_{\frac{1}{2}} E_1 Q_{\frac{1}{2}}^T = Q_1 E_1 = E_2,
\end{align*} 
the claim follows for $j=1$. The argument for $j=2$ is analogous. 
\end{proof}

\subsection{Limit $n\rightarrow \infty$}
\label{sec:lin_infty}

Similarly as in Section \ref{sec:infty} in the geometrically non-linear set-up, also in the geometrically linearised setting we now consider the limit $n\rightarrow \infty$. In particular, we are then  naturally lead to the same deformation as the one discussed in \cite{ADMDS15} in the context of nematic liquid crystal elastomers. 

\begin{lem}
\label{lem:diff_incl_n_infty}
For $n\rightarrow \infty$, the set $K$ from \eqref{eq:dif_in_lin} turns into 
\begin{align*}
K_{\infty}
&:= \left\{ R \begin{pmatrix} 1 & 0 \\ 0 & -1 \end{pmatrix} R^T: \ R \in O(2), \ c \in \R \right\}\\ 
&= \left\{c\begin{pmatrix} a & b \\ b & - a \end{pmatrix}: \ a^2+b^2 =1, \ c = \frac{1}{\sin(\frac{\pi}{n})} \right\}.
\end{align*}
\end{lem}

\begin{proof}
The first identity follows from considering $n\rightarrow \infty$ in \eqref{eq:E1}. The second identity is a consequence of the explicit form of $O(2)$.
\end{proof}

As a consequence, the differential inclusion which we study turns into 
\begin{align}
\label{eq:diff_incl}
e(\nabla u) \in K_{\infty}.
\end{align}

Linearising the solution from Proposition \ref{prop:n_infty}, we obtain a two-dimensional solution to the differential inclusion $\eqref{eq:diff_incl}$ with zero boundary conditions:

\begin{prop}
\label{prop:infty_lin}
The function 
\begin{align*}
w(x)= 2(1-\log(r^2))(x_2,-x_1)
\end{align*}
is a solution to the differential inclusion \eqref{eq:diff_incl} in $B_1 \setminus \overline{B_{1/2}}$ and $\nabla u = 0 $ in $\R^n \setminus \overline{B_1}$.
\end{prop}

\begin{figure}[t]
\includegraphics[scale=0.5,page=22]{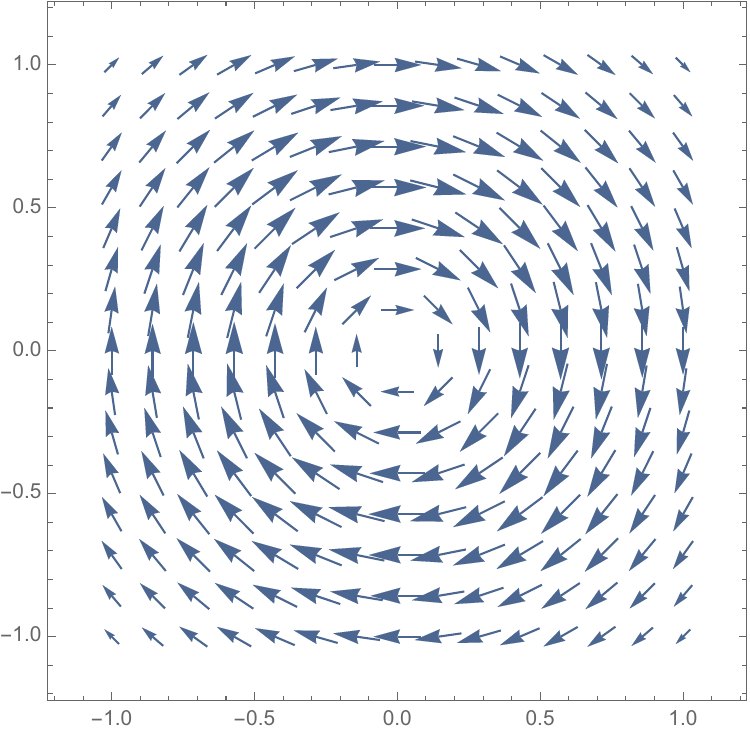}
\caption{A plot of the vector field $w(x)$ from Proposition \ref{prop:infty_lin}.}
\end{figure}

\begin{proof}
The claim follows from the identity
\begin{align*}
w(x_1,x_2):=\frac{d}{d\alpha} v_{\alpha}(x_1,x_2)|_{\alpha=\frac{1}{2}},
\end{align*}
where $v_{\alpha}$ is denotes the family of solutions from Proposition \ref{prop:n_infty} with $\alpha \in (0,1)$, the differential inclusion which is solved by $v_{\alpha}$ and the linearisation arguments from above. 
Indeed, a computation shows that
\begin{align*}
\frac{d}{d\alpha} v_{\alpha}(x_1,x_2)\big|_{\alpha = \frac{1}{2}}
&= r Q'(\omega + \rho_0|_{\alpha = \frac{1}{2}})\rho_0'|_{\alpha = \frac{1}{2}} \log(r) \begin{pmatrix} 1 \\ 0 \end{pmatrix}\\
& \quad + r Q(\omega + \rho_0|_{\alpha = \frac{1}{2}} \log(r)) \begin{pmatrix} 0 \\ -2 \end{pmatrix}\\
& = 4r Q'(\omega) \log(r) \begin{pmatrix}1 \\ 0 \end{pmatrix} + rQ(\omega)\begin{pmatrix} 0 \\ -2 \end{pmatrix}\\
&= 2(1-\log(r^2))\begin{pmatrix} x_2 \\ -x_1 \end{pmatrix}.
\end{align*}
Here $Q'(\omega)$ denotes the derivative of $Q(\omega)$ with respect to $\omega$.
\end{proof}

\begin{rmk}
We note that up to a multiplicative constant and an affine off-set whose gradient is a skew matrix, the function $w(x)$ recovers the special solution from Theorem 2.1 in \cite{ADMDS15}. This was found in \cite{ADMDS15} in the context of convex integration solutions for differential inclusions in  nematic liquid crystal elastomers. In the next section, we establish the connection between the differential inclusion \eqref{eq:diff_incl} and the one associated with two-dimensional liquid crystal elastomers.
\end{rmk}

\subsubsection{Geometrically linear planar solutions for nematic liquid crystal elastomer models}

\label{sec:liquids_lin}

In this section, we recall the modelling of nematic liquid crystal elastomers within the geometrically linearised theory and relate the associated differential inclusion for planar deformation to the differential inclusions, which we have considered in the previous section.

A prominent class of stored energy densities in the modelling of nematic liquid crystal elastomers within the geometrically linearised theory (which can formally be obtained as the linearisation of the non-linear energies) is of the form 
\begin{align}
\label{eq:lin_stored_en}
V(E) = \min\limits_{\hat{n}\in S^2} |E-U_{\hat{n}}|^2, \ U_{\hat{n}} = \frac{1}{2}(3\hat{n} \otimes \hat{n} - I),
\end{align}
where $E \in \R^{3\times 3}_{sym}$ and $\tr(E)=0$, see 
\cite{Ces10}.
Seeking to study energy zero solutions, one is thus lead to the corresponding differential inclusion problem
\begin{align}
\label{eq:diff_incl_lin}
e(\nabla u) \in K_{3D}:=\{E \in \R^{3\times 3}_{sym}; \ \mu_1(E)=\mu_2(E) = - \frac{1}{2}, \ \mu_3(E)=1 \},
\end{align}
where $\mu_j(E)$ denote the ordered eigenvalues of $E$. We note that for affine boundary conditions, the relaxation of this differential inclusion is given by
\begin{align}
\label{eq:relax_lin}
e(\nabla u) \in K_{3D}, \ \nabla u = M \mbox{ for some } M \in K_{3D}^{qc},
\end{align}
where
\begin{align}
\label{eq:Kqc}
K_{3D}^{qc} = \{E \in \R^{3\times 3}_{sym}: \ 
-\frac{1}{2} \leq \mu_1(E) \leq \mu_2(E) \leq \mu_3(E) \leq  1, \ \tr(E)=0 \}.
\end{align}
We refer for this to \cite{ADMDS15}, and also to Chapter 7 in \cite{DaM12}.

An interesting class of deformations is given by planar deformations. These were for instance studied in \cite{CD11}. In searching for energy zero solutions to \eqref{eq:lin_stored_en} with microstructure only in the planar direction, we study the following displacements
\begin{align*}
v(x_1,x_2,x_3)= (\tilde{v}(x_1,x_2),0) + \left[\begin{pmatrix} M' & 0 \\ 0 & m_{33} \end{pmatrix}x \right]^T,
\end{align*}
where $M\in \R^{2\times 2}$ and where $\tilde{v}(x_1,x_2)=0$ on $\partial B_1$, i.e. where the boundary data are encoded in the matrix $M=\begin{pmatrix} M' & 0 \\ 0 & m_{33} \end{pmatrix}$. 
In order to both ensure that $v$ is a solution to the differential inclusion \eqref{eq:relax_lin} and that 
there is interesting microstructure in the problem, we set $m_{33}=-\frac{1}{2}$ and consider boundary data $M$ which are of the form $M=\begin{pmatrix} M' & 0 \\ 0 & -\frac{1}{2} \end{pmatrix}$. 
For the resulting two-dimensional displacement $\tilde{v}$ one is then lead to the following differential inclusion:
\begin{align}
\label{eq:diff_incl_2D}
e(\nabla \tilde{v} + M') \in K_{2D}:=\left\{E \in \R^{2\times 2}_{sym}: \ - \frac{1}{2} = \mu_1(E) < \mu_2(E) = 1; \ \tr(E) = \frac{1}{2}\right\} .
\end{align}
The (relaxed) condition for $M'$ turns into
\begin{align}
\label{eq:bound_con}
e(M') \in \inte(K^{qc}_{2D}):= \left\{E \in \R^{2\times 2}_{sym}: \ - \frac{1}{2} < \mu_1(E) \leq \mu_2(E)<1; \ \tr(E) = \frac{1}{2}\right\}.
\end{align}

We are now searching for a solution $\tilde{v}(x_1,x_2)$ satisfying \eqref{eq:diff_incl_2D}, \eqref{eq:bound_con} such that $\tilde{v}(x_1,x_2)=0$ on $\partial B_1$.
To this end, we note the following necessary and sufficient conditions:

\begin{lem}
\label{lem:lin_nec_suff}
Let $\tilde{v}(x_1,x_2)$ be a solution to 
\begin{align}
\label{eq:incl_lin_2d}
e(\nabla \tilde{v} + M') \in K_{2D} \mbox{ in } B_1, \ \tilde{v}=0 \mbox{ on } \partial B_1.
\end{align}
Then, a necessary and sufficient condition for \eqref{eq:incl_lin_2d} is that 
\begin{align}
\label{eq:Eikonal}
\left( \p_1 \tilde{v}_1 + e_{11}(M') - \frac{1}{4} \right)^2
+ \left( \frac{1}{2}(\p_1 \tilde{v}_2 + \p_2 \tilde{v}_1)+ e_{12}(M') \right)^2 = \frac{9}{16}.
\end{align}
\end{lem}

\begin{rmk}
\label{rmk:Eik}
By using the trace constraints from \eqref{eq:bound_con} and \eqref{eq:diff_incl_2D}, we can rewrite
\begin{align}
\label{eq:trace}
e(M') = \begin{pmatrix} e_{11}(M') & e_{12}(M') \\ e_{12}(M')& -e_{11}(M') + \frac{1}{2} \end{pmatrix}, \ e(\nabla \tilde{v}) = \begin{pmatrix} \p_1 \tilde{v}_1 & \frac{1}{2}(\p_1 \tilde{v}_2 + \p_2 \tilde{v}_1)\\
\frac{1}{2}(\p_1 \tilde{v}_2 + \p_2 \tilde{v}_1) & \p_{2} \tilde{v}_2 \end{pmatrix},
\end{align}
with $\p_2 \tilde{v}_2 = - \p_1 \tilde{v}_1$.
The differential inclusion \eqref{eq:Eikonal} can then be written in a more symmetric form:
\begin{align}
\label{eq:Eik1}
\begin{split}
&\frac{1}{2}\left( \p_2 \tilde{v}_2 + e_{22}(M') - \frac{1}{4} \right)^2 +
\frac{1}{2}\left( \p_1 \tilde{v}_1 + e_{11}(M') - \frac{1}{4} \right)^2\\
& \quad + \left( \frac{1}{2}(\p_1 \tilde{v}_2 + \p_2 \tilde{v}_1)+ e_{12}(M') \right)^2 = \frac{9}{16}.
\end{split}
\end{align}
For $e_{11}(M')=\frac{1}{4}$, $e_{12}(M') = 0$ and $e_{22}(M') = \frac{1}{4}$ equation  \eqref{eq:Eik1} hence resembles a vectorial Eikonal type equation. 
\end{rmk}

\begin{rmk}
As a further observation, which might also be of interest in the context of the (quantitative) investigation of convex integration solutions, we point out that the setting of geometrically linear liquid crystal problems fits into the framework of \cite{RZZ17}. As a consequence, it is possible to deduce the existence of ``wild'' solutions with higher regularity. This is a consequence of the structure of the set $K^{qc}$ from \eqref{eq:Kqc} for which appropriate in-approximations and replacement constructions can be found similarly as in the $O(n)$ case. 
\end{rmk}

\begin{proof}
\emph{Necessity:} By definition of the set $K_{2D}$, for all matrices $\tilde{E} \in K_{2D}$ it holds that $\det(\tilde{E})= \frac{1}{2}$. Hence, a necessary condition for \eqref{eq:incl_lin_2d} is clearly given by the requirement that
\begin{align*}
\det(e(\nabla \tilde{v} + M')) = -\frac{1}{2}.
\end{align*}
With a few computations, it can be observed that this is equivalent to \eqref{eq:Eikonal}.\\

\emph{Sufficiency:}
A sufficient requirement for the validity of \eqref{eq:incl_lin_2d} is that 
\begin{align}
\label{eq:eigenval}
\det(e(\nabla \tilde{v}) + M' - \lambda Id) = 0
\end{align}
for $\lambda = 1$ and $\lambda =\frac{1}{2}$. Equation \eqref{eq:eigenval} can be rewritten as 
\begin{align*}
\lambda^2 - \frac{\lambda}{2} + \det(e(\nabla \tilde{v}) + M') = 0.
\end{align*}
Simplifying this expression for the choice $\lambda = 1$ and $\lambda =\frac{1}{2}$ then indeed also leads to \eqref{eq:Eikonal}.
\end{proof}

With Lemma \ref{lem:lin_nec_suff} in hand, we can relate the differential inclusion from \eqref{eq:diff_incl} to the nematic liquid crystal elastomer differential inclusions \eqref{eq:incl_lin_2d}, \eqref{eq:Eikonal}.
This allows us to ``explain'' the coincidence of the solution from Proposition \ref{prop:infty_lin} and the one found in \cite{ADMDS15}:

\begin{cor}
\label{prop:infty_liquid_cry}
Let $v$ be the solution from Proposition \ref{prop:infty_lin}. Then, $\frac{4}{3}v$ is a solution to \eqref{eq:Eik1} with 
\begin{align*}
e(M') = \begin{pmatrix} \frac{1}{4} & 0 \\ 0 & \frac{1}{4} \end{pmatrix}. 
\end{align*}
\end{cor}

\begin{proof}
The result follows directly by comparing the form of $K_{\infty}$ from Lemma \ref{lem:diff_incl_n_infty} and \eqref{eq:Eikonal}. For the chosen value of $e(M')$ the differential inclusions only differ by a multiplicative constant.
\end{proof}

\section{Remarks on three-dimensional constructions}
\label{sec:3D}

In this section we discuss adaptations of the two-dimensional constructions of
Section \ref{sec:nonlinear_n_well} to the case of two nested regular tetrahedra
$T_1,T_2\subset \R^3$.
Here, it turns out that while it is possible to construct families of
volume-preserving piecewise affine transformations, there are no non-trivial constructions which exhibit an $m$-well structure
\begin{align}
  \label{eq:mwell3D}
  \nabla u \in SO(3)\cup \bigcup_{P \in \mathcal{P}_n} SO(3)P U P^T; \ \det(U)=1,
\end{align}
where $\nabla u \in SO(3)$ corresponds to an austenite configuration and $\mathcal{P}_n \subset SO(3)$ denotes a suitable symmetry group.

After possibly rescaling and rotating $u$, we may assume that $T_1$ is given by
the convex hull of the four points 
\begin{align}
  \label{eq:innerchoice}
  \begin{pmatrix}
    1 \\ 1 \\ 1
  \end{pmatrix},
  \begin{pmatrix}
    1 \\ -1 \\ -1
  \end{pmatrix},
  \begin{pmatrix}
    -1 \\ 1 \\-1
  \end{pmatrix},
  \begin{pmatrix}
    -1 \\ -1 \\ 1
  \end{pmatrix}.
\end{align}
With this choice of coordinates, the barycenter of $T_1$ is in $(0,0,0)$
and two distinguished axes of rotation are given by the $x_3$ axis
\begin{align}
  \label{eq:axis1}
  \R \begin{pmatrix}
    0\\0\\1
  \end{pmatrix}
\end{align}
and 
\begin{align}
  \label{eq:axis2}
\mathbb{R}
  \begin{pmatrix}
    1\\1\\1
  \end{pmatrix}
\end{align}
Furthermore, the dual tetrahedron to $T_1$ is up to
rescaling given by $-T_1$.

In the following we consider two symmetric constructions, where the inner
tetrahedron $T_2$ has the same barycenter and shares an axis of symmetry with
$T_1$.
The deformation $u$ is then obtained by rotating $T_2$ around this axis and linearly interpolating on the polyhedra spanned
by vertices, edges and surfaces of $T_1$ and $T_2$. By our choice of coordinates
we may assume that the distinguished axis is either given by \eqref{eq:axis1}
which is illustrated in Figures \ref{fig:sym} and \ref{fig:sym_inter}, or by
\eqref{eq:axis2} which is illustrated in Figures \ref{fig:inner} and \ref{fig:flip3D}.

In particular, since $u$ is required to be volume-preserving as $K \subset \{M:
\text{det}(M)=1\}$, it follows that $u$ needs to preserve the distance of the
vertices of $T_2$ to the corresponding surfaces of $T_1$.
Computations show that there is no non-trivial choice of $T_2$ and $R T_2$ such that
this distance is the distance for all four corners of $T_2$.
Hence, we relax this constraint to consider the case where $T_2$ is chosen to be
a rescaled dual copy of $T_1$, which is initially rotated around either $\R  \begin{pmatrix}
    1\\1\\1
  \end{pmatrix}$ or the $x_3$-axis. These configurations are depicted in Figure \ref{fig:sym} and Figure
\ref{fig:inner}, respectively.

\begin{figure}[h]
  \centering
  \includegraphics[width=0.5\linewidth,page=23]{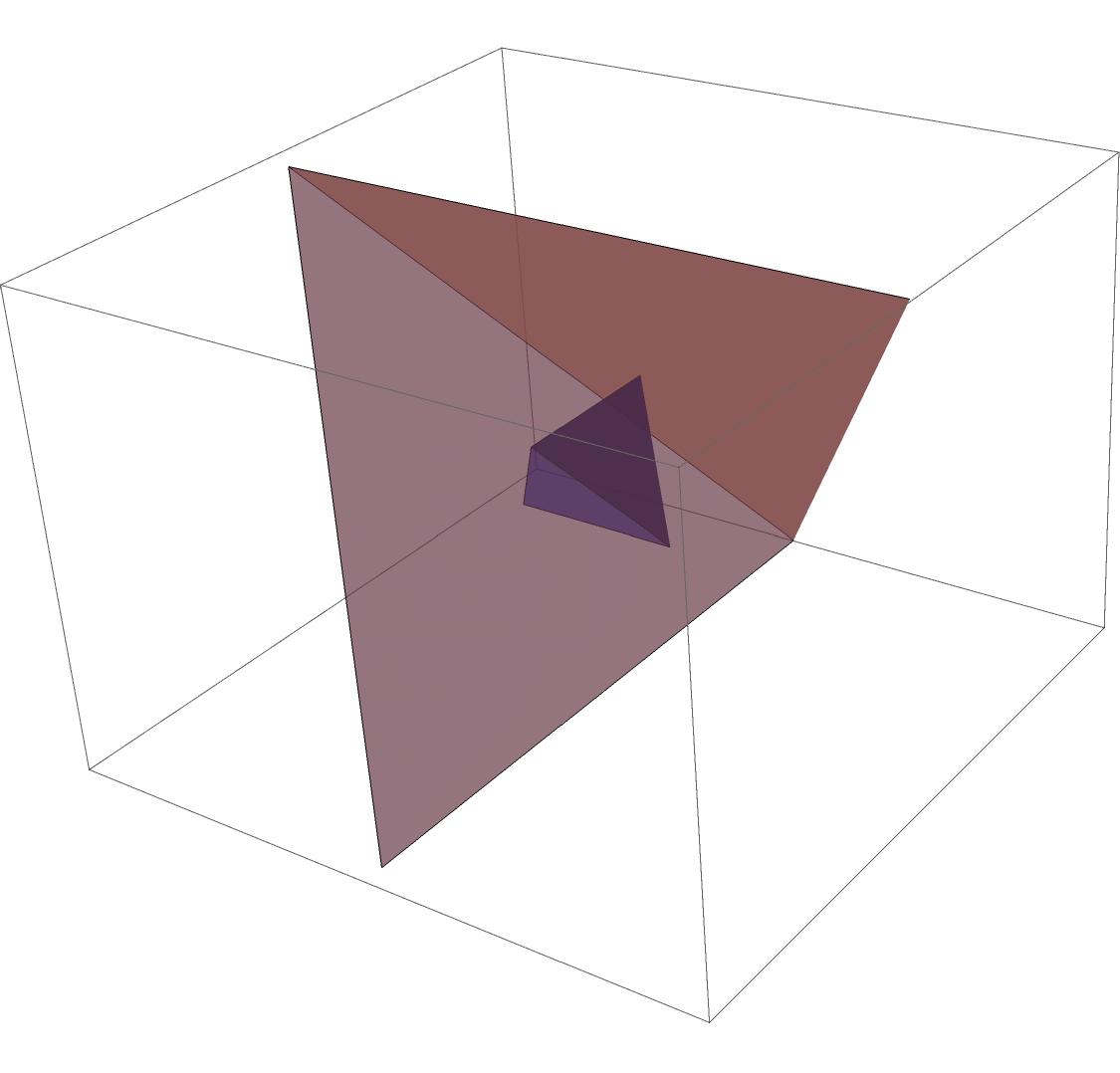}
  \caption{The inner tetrahedron $T_2$ shares the $x_3$-axis as a common symmetry
    axis with $T_1$. Upon rotating $T_2$ around this axis, a piecewise affine
    transformation $u$ is obtained by interpolating on the various polyhedra
    shown in Figure \ref{fig:sym_inter}.}
  \label{fig:sym}
\end{figure}

\begin{figure}[h]
  \centering
  \includegraphics[width=0.3\linewidth,page=24]{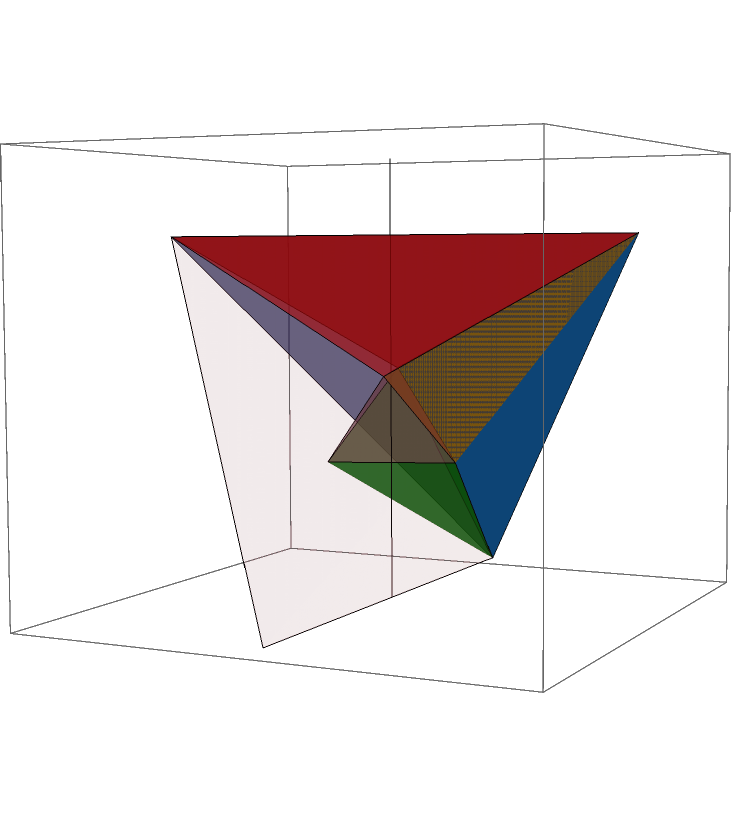}
  \includegraphics[width=0.3\linewidth,page=25]{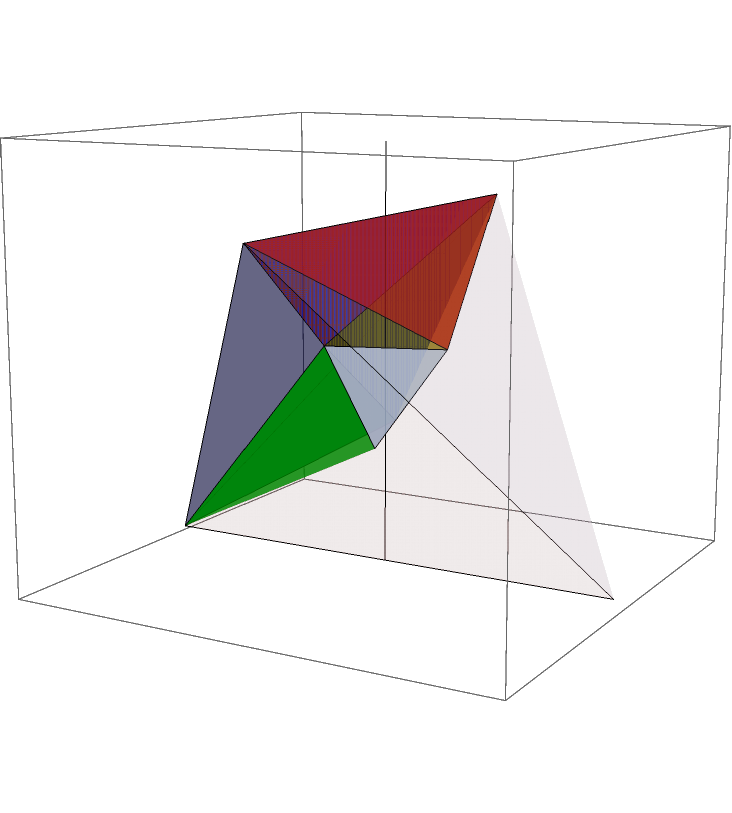}
  \caption{The inner and outer tetrahedron both share the $x_3$ axis as a
    symmetry axis, viewed here from two different rotated points of view.
    We consider a map $u$, which is affine in the
    interior tetrahedron, given by the identity outside the outer tetrahedron
    and given by affine interpolations in the remaining regions composed of
    (irregular) tetrahedra.
    Up to symmetry, there are $5$ distinct interpolation regions, which are
    colored in this picture.}
  \label{fig:sym_inter}
\end{figure}

\begin{figure}[h]
  \centering
  \includegraphics[width=0.5\linewidth,page=26]{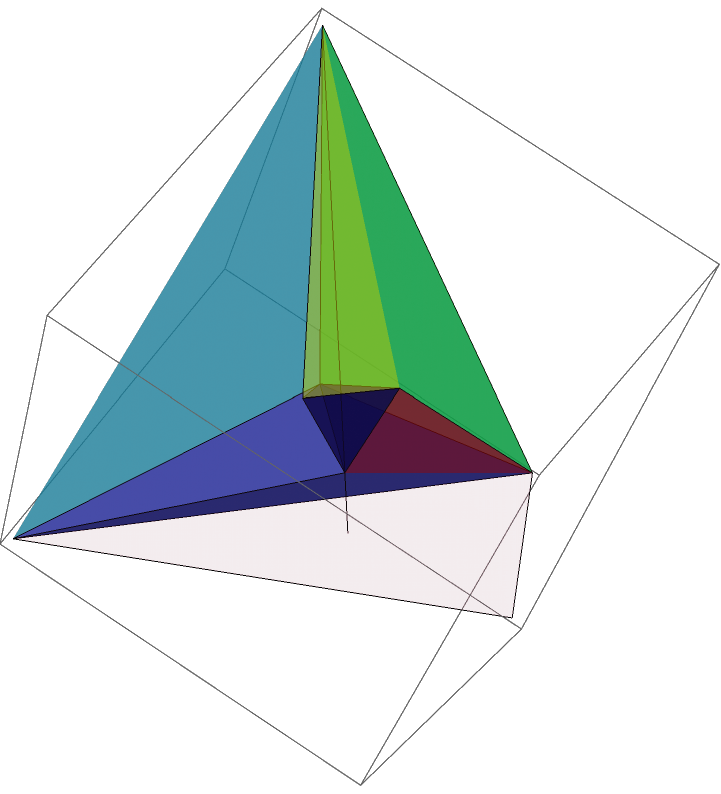}
  \caption{The tetrahedra share a common symmetry axis through their barycenter
    and one of the corners. We consider a map $u$, which is given by a rotation
    in the interior tetrahedron, by the identity outside the outer tetrahedron
    and given by affine interpolations in the remaining regions composed of
    (irregular) tetrahedra.
    Up to symmetry, there are $5$ distinct interpolation regions, which are
    colored in this picture.}
  \label{fig:inner}
\end{figure}

\begin{figure}[h]
  \centering
  \includegraphics[width=0.5\linewidth,page=27]{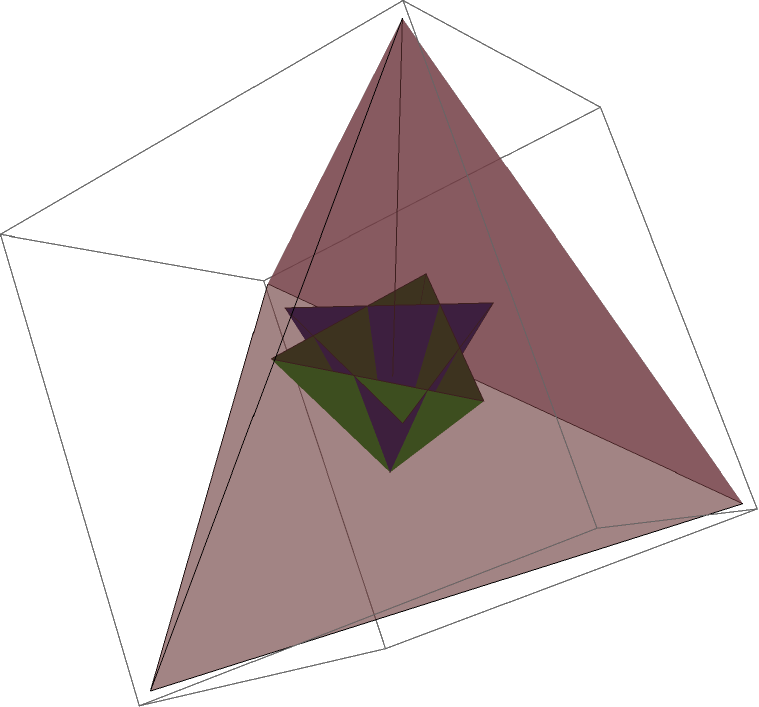}
  \caption{Given an initial configuration depicted in blue, the green ``flipped''
    configuration is the only rotation around the symmetry axis that preserves
    the volume of the polyhedra spanned by an inner corner and an outer surface.}
  \label{fig:flip3D}
\end{figure}

\subsection{Rotations around the $x_3$-axis}
\label{sec:x3}

We first consider the setting depicted in Figures \ref{fig:sym} and
\ref{fig:sym_inter}.
We in particular note that the cyan interpolation region (for colours we refer to the online version of the article) is obtained by
interpolating between a surface $S$ of $T_1$ and a vertex $v$ of $T_2$.
The volume-preservation constraint $\det(\nabla u)=1$ imposed by the $m$-well
condition \eqref{eq:mwell3D} thus implies that the map $u$ needs to preserve the
distance between the surface $S=u(S)$ and $u(v)$.
Similarly to the two-dimensional setting (c.f. Figure
\ref{fig:2} and the preceding remarks) this implies that if initially
\begin{align}
  \label{eq:initialx3}
  T_2=r R_{-\theta} (-T_1),
\end{align}
then necessarily
\begin{align}
 u(T_2)= r R_{+\theta} (-T_1),
\end{align}
where $r \in (1, \frac{1}{3})$ is a scaling factor, $R_{\theta}$ is the rotation around the $x_3$ axis with angle $\theta$
and we recall that, up to scaling, $-T_1$ is the dual tetrahedron to $T_1$.
Thus $u$ acts on $T_2$ by a rotation by $2\theta$ and we say that the
tetrahedron is ``flipped'' from being rotated by an angle $-\theta$ to being
rotated by an angle $\theta$.
With this choice, for any $r>0$ and any $\theta>0$, it follows that $u$ is a
volume preserving affine transformation in each of the regions highlighted in
Figure \ref{fig:sym_inter}.
However, while volume-preservation is a necessary condition for the $m$-well
problem \eqref{eq:mwell3D}, this is not sufficient.
We may explicitly compute that in the red interpolation region $\nabla u$ is
given by the shear
\begin{align}
  U_1(\theta)=
  \begin{pmatrix}
    1 & 0 & 0 \\
    2\tan(\theta) & 0 & 0 \\
    0 & 0 &1
  \end{pmatrix},
\end{align}
and in particular is independent of $r>0$.
Since none of the interpolated transformations are given by rotations, we thus
ask whether there exist suitable choices of $\theta, r$ such that  
\begin{align}
  \label{eq:mwellx3}
  \nabla u \in \bigcup_{P \in \mathcal{P}_n} SO(3)P U_1 P^T
\end{align}
in the remaining regions for a suitable choice of a symmetry group
$\mathcal{P}_n$.
A necessary condition for this requirement is that in all interpolation regions
the singular values of $\nabla u$ agree with the singular values of $U_1$.
An explicit numerical computation yields that the singular values are given by
$(\lambda,1,\frac{1}{\lambda})$, where $\lambda$ depends on the angle
$\theta$ and the scaling factor $r$ chosen in \eqref{eq:initialx3} and the
interpolation region.

Figure \ref{fig:singularvaluesx3} shows plots of $\lambda$ in the various
regions and was obtained by direct numerical calculations.

\begin{figure}[h]
  \centering
   \includegraphics[width=0.5\linewidth,page=28]{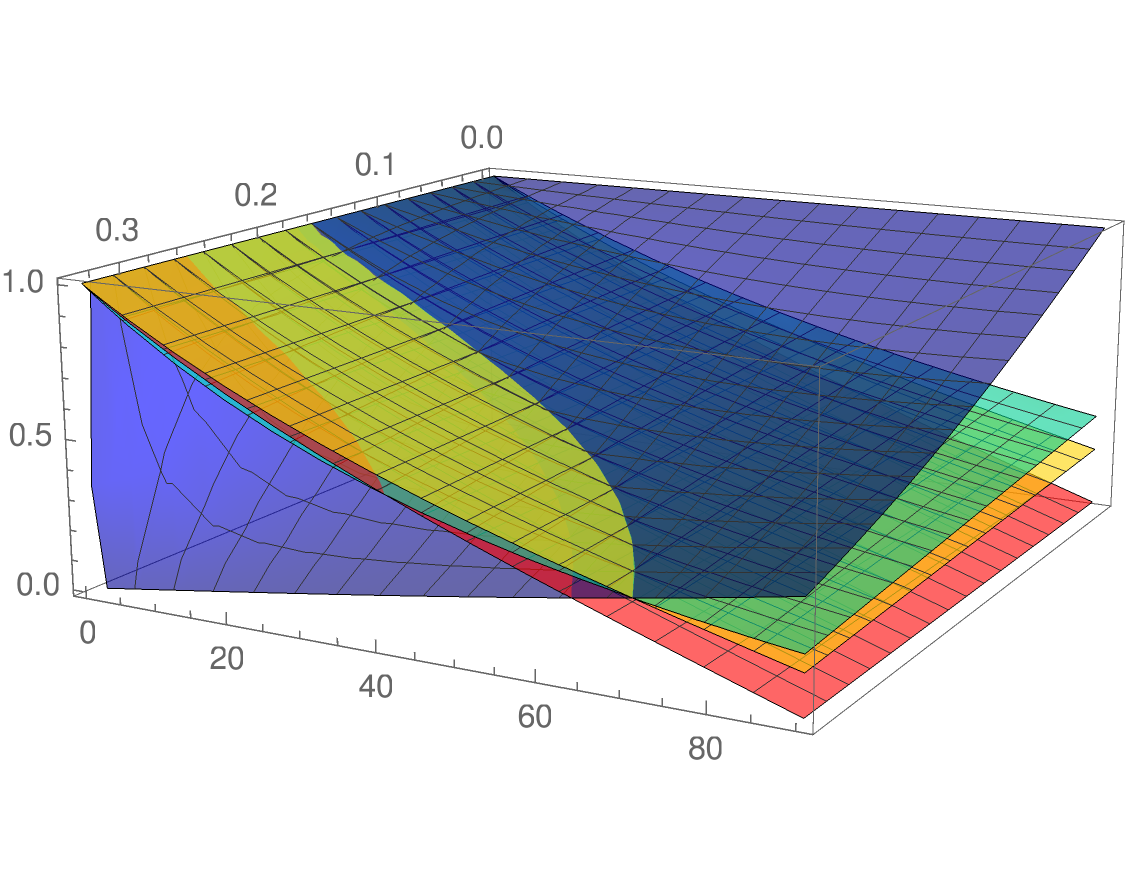}
   \caption{Singular values for the construction of Figure \ref{fig:sym_inter}.
     We numerically compute the smallest singular value of the transformation $u$ in
     various regions as functions of the angle $\theta \in (0,\frac{\pi}{2})$ and scaling
     factor $r \in (0,\frac{1}{3})$
     chosen in \eqref{eq:initialx3}. 
  }
  \label{fig:singularvaluesx3}
\end{figure}

In particular, we observe that there are no non-trivial choices of $r,\theta$
such that the singular values $\lambda$ agree in all regions. The necessary
condition for the $m$-well inclusion \eqref{eq:mwellx3} is thus never satisfied.

We remark that key obstacles of this three-dimensional construction are given by
the non-commutative structure of $SO(3)$ and the requirement to choose an axis
for the rotation of $T_2$. While in two dimensions all rotations commute and all
interpolation regions are given by triangles, in the present setting the
interpolations in the various regions instead behave qualitatively differently
and are for instance not anymore given by shears.  

\subsection{Rotations around an axis through a vertex}
\label{sec:vertex}

In this subsection we consider the construction depicted in Figure
\ref{fig:inner}, where
\begin{align}
  \label{eq:initial2}
  T_2=r R_{-\theta}^* (-T_1),
\end{align}
is instead rotated around the axis \eqref{eq:axis2} through the origin and one of
the corners of $T_1$.
As in the two-dimensional case, the determinant constraint and the resulting
volume preservation implies that an inner tetrahedron initially rotated by an angle
$-\theta$ compared to the dual tetrahedron of $T_1$ can only be ``flipped'' to an
angle $\theta$ (see Figure
\ref{fig:flip3D} for an illustration):
\begin{align}
  u(T_2)= r R_{-\theta}^* (-T_1).
\end{align}
We thus consider the mapping $u:\R^3\rightarrow \R^3$ which acts as the identity
outside the outer tetrahedron $T_1$, as a rotation by $2 \theta$ around the
symmetry axis inside $T_2$, and is given by the affine interpolation in any of the
(irregular) tetrahedra of the types depicted in Figure \ref{fig:inner}.
With this choice of construction the transformation $u$ is volume-preserving on
all interpolation regions for all choices of $r \in(0,\frac{1}{3})$ and $\theta
\in (-\frac{\pi}{2}, \frac{\pi}{2})$.
As a main difference to the construction of Section \ref{sec:x3} for the
$x_3$-axis case, we observe that under this transformation, the tetrahedron
obtained by interpolating between a surface of $T_2$ and $(1,1,1)$ (colored
yellow in Figure \ref{fig:inner}) is transformed by a rigid rotation and thus corresponds to austenite.
Furthermore, the region obtained by interpolating between a surface of $T_1$ and
the vertex $-r(1,1,1) \in T_2$ remains invariant under $u$ and thus also
corresponds to austenite.
For the remaining regions, we thus ask whether there is a choice of parameters
$r \in (0,\frac{1}{3}), \theta \in (-\frac{\pi}{2},\frac{\pi}{2})$ and a
suitable matrix $U$ and group $\mathcal{P}_n \subset SO(3)$ such that
\begin{align}
  \label{eq:mwell3D2}
  \nabla u \in \bigcup_{P \in \mathcal{P}_n} SO(3)P U P^T.
\end{align}
As in the setting of Subsection \ref{sec:x3} the singular values in these
regions are given by $\lambda, 1, \frac{1}{\lambda}$.

A plot of $\lambda$ in the various regions obtained by direct numerical
computation is given in Figure \ref{fig:singularvalues2}.
In particular, we again observe that there are no non-trivial choices of $r,\theta$
such that the singular values $\lambda$ agree in all regions. The necessary
condition for the $m$-well inclusion \eqref{eq:mwell3D2} is thus never satisfied.

\begin{figure}[h]
  \centering
   \includegraphics[width=0.5\linewidth,page=29]{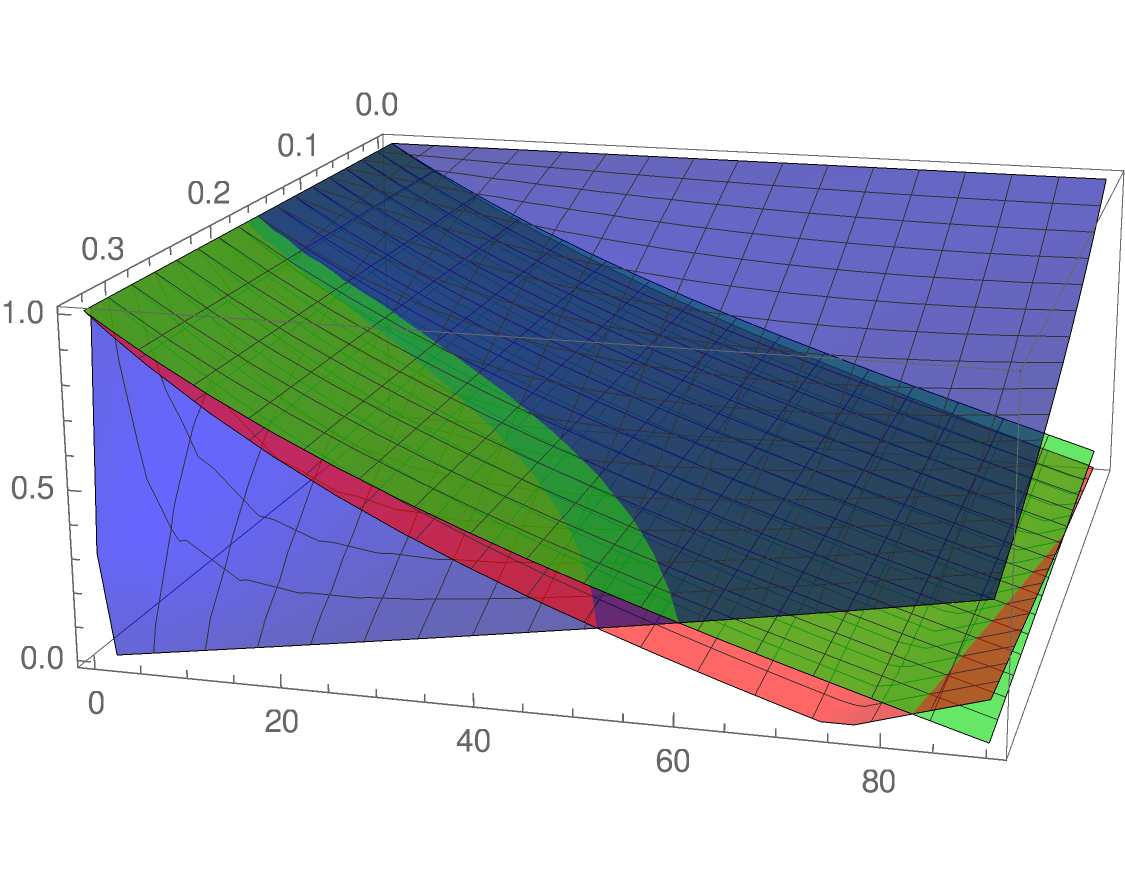}
   \caption{Singular values for the construction of Figure \ref{fig:inner}.
     We numerically compute the smallest singular value of the transformation $u$ in
     various regions as functions of the angle $\theta \in (0,\frac{\pi}{2})$ and scaling
     factor $r \in (0,\frac{1}{3})$
     chosen in \eqref{eq:initial2}.
  }
  \label{fig:singularvalues2}
\end{figure}

\appendix

\section{Necessary relation between the radius of the outer polygon and the radius of the inner polygon: the solutions to \eqref{LongEq}}
\label{LongEqSolve}
In this first part of the appendix, we provide the remainder of the argument from Proposition \ref{lem:nec}.

To this end, we solve 
\beq
\label{LongEq2}
\begin{split}
\Bigl(1+x^2-2x\cos( \frac{2\pi}{n}\alpha)\Bigr)\Bigl(1+x^2-2x\cos\Bigl(\frac{2\pi}{n}(1-\alpha) \Bigr)\Bigr)\cos\Bigl(\frac{(n-2)\pi}{2n }\Bigr)\\ = \Bigl(1+ x^2\cos\bigl(\frac{2\pi}{n}\bigr)-2x\cos\bigl(\frac\pi n\bigr)\cos\bigl(\frac{\pi}n(1-2\alpha)\bigr)\Bigr)^2
\end{split}
\eeq
which is \eqref{LongEq} squared. We get the following four solutions of the equation \eqref{LongEq} for $x$:
\begin{itemize}
\item $x=\frac{1}{\cos\frac{\pi}{n}}\Bigl(\cos\Bigl(\frac {\rho_n}2\Bigr) - \sqrt{\sin(\frac{2\pi}{n}\alpha)\sin\Bigl( \frac{2\pi}{n}(1-\alpha)\Bigr) } \Bigr)$,
\item $x=\frac{1}{\cos\frac{\pi}{n}}\Bigl(\cos\Bigl(\frac{\rho_n} 2\Bigr) + \sqrt{\sin(\frac{2\pi}{n}\alpha)\sin\Bigl( \frac{2\pi}{n}(1-\alpha)\Bigr) } \Bigr)$,
\item $x=\frac{1}{\cos\bigl(\frac{3\pi}{n}\bigr)}\Bigl(\cos\Bigl(\frac{2\pi}{n}\Bigr)\cos\Bigl(\frac{\rho_n} 2 \Bigr) -  \sqrt{\cos^2\Bigl(\frac{2\pi}{n}\Bigr)\cos^2\Bigl(\frac{\rho_n} 2\Bigr) - \cos\Bigl(\frac{3\pi}{n}\Bigr)\cos\Bigl(\frac{\pi}{n}\Bigr)} \Bigr)$,
\item $x=\frac{1}{\cos\bigl(\frac{3\pi}{n}\bigr)}\Bigl(\cos\Bigl(\frac{2\pi}{n}\Bigr)\cos\Bigl(\frac{\rho_n} 2\Bigr) +  \sqrt{\cos^2\Bigl(\frac{2\pi}{n}\Bigr)\cos^2\Bigl(\frac{\rho_n} 2\Bigr) - \cos\Bigl(\frac{3\pi}{n}\Bigr)\cos\Bigl(\frac{\pi}{n}\Bigr)} \Bigr)$,
\end{itemize}
where as in \ref{H4}, ${\rho_n}:=\frac{2\pi}{n}(1-2\alpha)$.
We now claim that just the first solution is admissible for us. Here and below we define a solution $x$ of \eqref{LongEq2} admissible if $x\in(0,1)$ and it satisfies \eqref{LongEq}. In order to prove our claim, we can assume without loss of generality that 
$$\sqrt{4\cos^2\Bigl(\frac{2\pi}{n}\Bigr)\cos^2\Bigl(\frac{\rho_n}2\Bigr) - 4\cos\Bigl(\frac{3\pi}{n}\Bigr)\cos\Bigl(\frac{\pi}{n}\Bigr)}
$$ 
is real, otherwise the third and fourth solutions are not admissible. The proof of the claim is as follows: 
\begin{itemize}
\item Second solution: We estimate
\begin{align*}
x \geq \frac{\cos\left( \frac{{\rho_n}}{2} \right)}{\cos\left( \frac{\pi}{n} \right)}.
\end{align*}
Since $\alpha\in \left( 0,1\right)$ it is clear that the second solution is such that $x\geq 1$ for any $\alpha\in(0,1)$, any $n\geq 3$.
\item Third solution: $x\geq 1$ if $n=3,4$ and $\alpha\in[0,1].$ We can hence restrict to the case $n>4$. We now claim that 
\beq
\label{IsNeg}
1+ x^2\cos\bigl(\frac{2\pi}{n}\bigr)-2x\cos\bigl(\frac\pi n\bigr)\cos\bigl(\frac{\rho_n}2\bigr)<0
\eeq
for any $\alpha\in(0,1)$, and any $n\geq 4$. Since the left-hand side of \eqref{LongEq} is always non-negative, the claim would imply that the third solution of \eqref{LongEq2} does not satisfy \eqref{LongEq}, and is hence not admissible. We plot $1+ x^2\cos\bigl(\frac{2\pi}{n}\bigr)-2x\cos\bigl(\frac\pi n\bigr)\cos\bigl(\frac{{\rho_n}}2\bigr)$ for $n\in\{5,\dots,50\}$ in Figure \ref{fig:VerificaIpotesi}. For large $n$, we have that 
$$x = 1-\frac{2\pi}{n}\sqrt{(\alpha-\alpha^2)}  +\frac{2\pi^2}{n^2}(-\alpha^2+\alpha +1) + O(n^{-3}),
$$
and, therefore,
$$
1+ x^2\cos\bigl(\frac{2\pi}{n}\bigr)-2x\cos\bigl(\frac\pi n\bigr)\cos\bigl(\frac{\rho_n}2\bigr) = -\frac{4\pi^3}{n^3}\sqrt{\alpha(1-\alpha)} + O(n^{-4}) < 0
$$
for any $\alpha\in(0,1),$ and for any $n$ large enough.
\begin{figure}[h]
  \centering
   \includegraphics[width=0.5\linewidth,page=30]{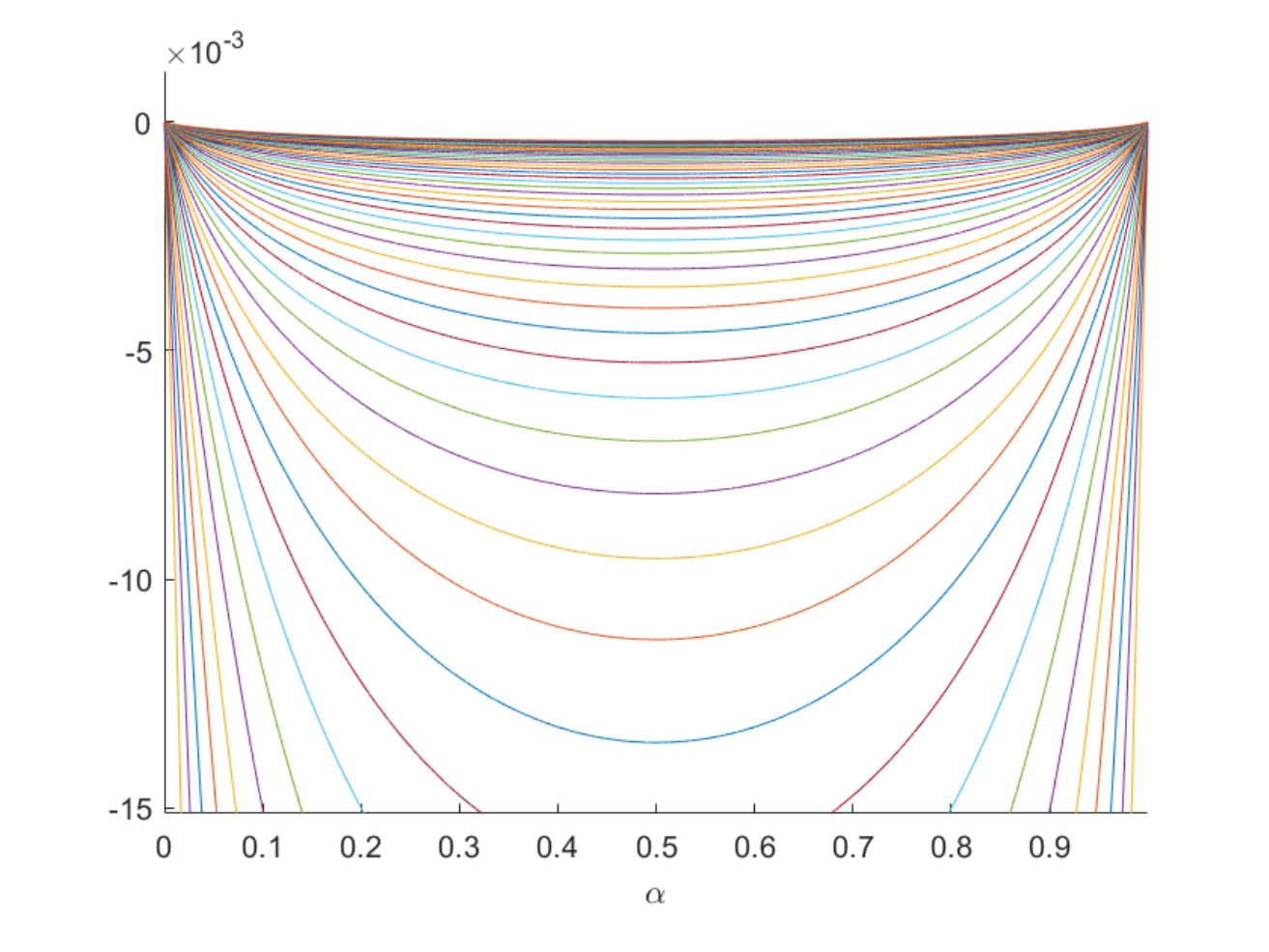}
   \caption{Numerical verification of the fact that, for $n\in\{1,\dots,50\}$ and $\alpha\in(0,1)$ we have \eqref{IsNeg}. The bigger $n$ is, the closer to zero the convex curves in the pictures are.}
  \label{fig:VerificaIpotesi}
\end{figure}
\item Fourth solution: It is easy to see that it is negative for any $\alpha\in[0,1]$ when $n=3,4,5$. Indeed, $\cos\frac{3\pi}{n}<0$. If $n=6$ we get $x=\infty,$ while for $n>6$ we have $x>1$. Indeed, in this case,
$$
x\geq \frac{2\cos\Bigl(\frac{2\pi}{n}\Bigr)\cos\Bigl(\frac{\rho_n}2\Bigr)}{2\cos\bigl(\frac{3\pi}{n}\bigr)}\geq \frac{\cos\Bigl(\frac{2\pi}{n}\Bigr)\cos\Bigl(\frac{\pi}{n}\Bigr)}{\cos\bigl(\frac{3\pi}{n}\bigr)}= \frac12\Bigl(1+\frac{1}{2\cos\bigl(\frac{2\pi}{n}\bigr)-1} \Bigr)>1.
$$
Therefore, for any $\alpha\in[0,1]$ and any $n\geq3$ the fourth solution is not admissible.
\end{itemize}

\section{Proof of Corollary \ref{cor:iterate}}
\label{AppCode}
In this part of the appendix we show that equation \eqref{eq:corollary26}
\begin{align*}
P_0 U P_0 = Q_{\alpha} U Q_{1-\alpha}^T
\end{align*}
is satisfied.
In order to simplify calculations, we express all matrices with respect to the
basis $(e_{11}, e_{11}^\perp)$ and thus have to show that
\begin{align}
  \begin{pmatrix}
    a  & - \frac{a^{-1}-a}{\tan(\phi)} \\ 0 & a^{-1}
  \end{pmatrix} Q_{1-\alpha} = Q_{\alpha}   \begin{pmatrix}
    a  & \frac{a^{-1}-a}{\tan(\phi)} \\ 0 & a^{-1}
  \end{pmatrix}.
\end{align}
We further recall that
\begin{align*}
  a^2 &= \frac{\sin(\frac{2\pi}{n}(1-\alpha))}{\sin(\frac{2\pi}{n}\alpha)}, \
  \frac{1}{\tan(\phi)}= \tan\left(\frac{\pi}{n}\right).
\end{align*}
In particular, since $\alpha \in (0,1)$, we may multiply the claimed equation
with $a \sin(\frac{2\pi}{n}\alpha) \neq 0 $ and for simplicity of notation introduce
$t:=\frac{2\pi}{n}\alpha $ and $s=
\frac{2\pi}{n}(1-\alpha) = \frac{2\pi}{n}-t$.
With this notation, we have to show that 
\begin{align*}
              & \quad \begin{pmatrix}
                \sin(s) & -(\sin(t)-\sin(s))\tan(\frac{\pi}{n}) \\ 0 & \sin(t)
              \end{pmatrix}
                \begin{pmatrix}
    \cos(s) & -\sin(s) \\
    \sin(s) & \cos(s)
  \end{pmatrix} \\
 & = \begin{pmatrix}
    \cos(t) & -\sin(t) \\
    \sin(t) & \cos(t)
  \end{pmatrix}
              \begin{pmatrix}
                \sin(s) & (\sin(t)-\sin(s))\tan(\frac{\pi}{n}) \\ 0 & \sin(t)
              \end{pmatrix}.
\end{align*}
We consider each matrix entry separately.
The claimed equality for the upper left entry is given by
\begin{align}
\label{eq:trig_ident}
\sin(s)\cos(s)- \sin(s)\tan\left(\frac{\pi}{n}\right) (\sin(t)-\sin(s))=\cos(t)\sin(s).
\end{align}
In order to show this, we may factor out $\sin(s)$ and use the angle addition formulas:
\begin{align*}
  \cos(s)&=\cos(t)\cos\left(\frac{2\pi}{n}\right)+ \sin(t)\sin\left(\frac{2\pi}{n}\right), \\
  \sin(s)&=\cos(t)\sin\left(\frac{2\pi}{n}\right)- \sin(t)\cos\left(\frac{2\pi}{n}\right).
\end{align*}
We then collect terms involving $\cos(t)$ and $\sin(t)$ as
\begin{align*}
  &\quad \cos(t)\cos\left(\frac{2\pi}{n}\right)+ \sin(t)\sin\left(\frac{2\pi}{n}\right)\\
  & \quad -\tan\left(\frac{\pi}{n}\right)
  \left(\sin(t)-\cos(t)\sin\left(\frac{2\pi}{n}\right)+ \sin(t)\cos\left(\frac{2\pi}{n}\right)\right) \\
  &= \cos(t)\left(\cos\left(\frac{2\pi}{n}\right)+\tan\left(\frac{\pi}{n}\right)\sin\left(\frac{2\pi}{n}\right)\right)\\
  & \quad
  + \sin(t)\left(\sin\left(\frac{2\pi}{n}\right)-\tan\left(\frac{\pi}{n}\right)(1+\cos\left(\frac{2\pi}{n}\right)\right)\\
  &= \cos(t) 1+ \sin(t) 0 = \cos(t),
\end{align*}
where we used the half angle identities for $\cos(2x)$ and $\sin(2x)$ in the
last equality.

The calculation for the bottom right-entry is analogous with the role of $s$ and
$t$ and the sign of $(\sin(t)-\sin(s))\tan(\frac{\pi}{n})$) interchanged.
The bottom left equality $\sin(t)\sin(s)=\sin(t)\sin(s)$ is always satisfied.
It thus only remains to verify equality of the upper right entry, which can be
simplified to read
\begin{align*}
  &-\sin^2(s)-\cos(s)(\sin(t)-\sin(s))\tan\left(\frac{\pi}{n}\right)\\
  & \quad  =
  -\sin^2(t)+ \cos(t)(\sin(t)-\sin(s))\tan\left(\frac{\pi}{n}\right) \\
  &\Leftrightarrow
  \sin^2(t)-\sin^2(s) - (\cos(t)+\cos(s))(\sin(t)-\sin(s))\tan\left(\frac{\pi}{n}\right)=0.
  \end{align*}
  Factoring out the factor $(\sin(t)-\sin(s))$, it suffices to prove
  \begin{align*}
\sin(t)+\sin(s) - (\cos(t)+\cos(s))\tan\left(\frac{\pi}{n}\right)=0.
\end{align*}
As above, the claimed equality
then again follows by using angle addition formulas.

\section{Reduction to Cauchy-Green Tensors used in the Proof of Proposition \ref{prop:non_iterable}}

Last but not least, we provide the argument (used in the proof of Proposition \ref{prop:non_iterable}) that it is possible to reduce the differential inclusion \eqref{eq:quest} to an inclusion for the associated Cauchy-Green tensors.

    \begin{lem}\label{lem:reduction}
    Suppose that $\det(M)=\det(U)>0$, then the inclusion
    \begin{align}
      \label{eq:inclusion}
      M \in \bigcup_{P \in \mathcal{P}_n}SO(2) P^T U P
    \end{align}
    is satisfied, if and only if
    \begin{align}
      \label{eq:inclusion2}
      M^TM \in \bigcup_{P \in \mathcal{P}_n} P^T U^T U P.
    \end{align}
  \end{lem}
  This characterisation follows from basic properties of the singular value
  decomposition.
  \begin{proof}
    We observe that \eqref{eq:inclusion} implies \eqref{eq:inclusion2}. Thus, we only 
    consider the converse and assume that
    \begin{align*}
      M^TM = (P^TU^TP) (P^T U P)=: M_1^T M_1.
    \end{align*}
    for some $P \in \mathcal{P}_n$.
    Since $M^TM$ is symmetric, there exists $Q \in SO(2)$ and a diagonal matrix
    $\diag(\lambda_1,\lambda_2)$, with $\lambda_1 \lambda_2=\det(M)^2\neq 0$, $\lambda_1,\lambda_2>0$,
    such that
    \begin{align*}
      M^T M= Q^T \diag(\lambda_1,\lambda_2)Q.
    \end{align*}
    It follows that
    \begin{align*}
      \tilde{M}:= M Q^T \diag(\frac{1}{\sqrt{\lambda_1}}, \frac{1}{\sqrt{\lambda_2}}), \\
      \tilde{M_1}:= M_1 Q^T \diag(\frac{1}{\sqrt{\lambda_1}}, \frac{1}{\sqrt{\lambda_2}}), 
    \end{align*}
    satisfy
    \begin{align*}
      \tilde{M}^T\tilde{M}= I= \tilde{M_1}^T\tilde{M_1}
    \end{align*}
    and thus $\tilde{M}, \tilde{M_1} \in SO(2)$. Here we used that
    $\det(M)=\det(U)>0$.
    In particular,
    \begin{align*}
      M_1&= \tilde{M_1} Q^T\diag(\sqrt{\lambda_1}, \sqrt{\lambda_2}), \\
      M& = \tilde{M} Q^T\diag(\sqrt{\lambda_1}, \sqrt{\lambda_2}) \\
      &= \tilde{M} \tilde{M}^T_1 M_1,
    \end{align*}
    where $\tilde{M}\tilde{M}^T_1 \in SO(2)$, which implies the result.
  \end{proof}

  \section*{Acknowledgements}
P.C. is supported by JSPS  Grant-in-Aid for Young Scientists (B) 16K21213
and partially by JSPS Innovative Area Grant 19H05131. P.C. holds an honorary appointment at La Trobe University and is a member of GNAMPA. C.Z. acknowledges a travel grant from the Simon's foundation. B.Z. would like to thank Sergio Conti for helpful discussions, and acknowledges support by the Berliner Chancengleichheitsprogramm and by the Deutsche Forschungsgemeinschaft through SFB 1060 “The Mathematics of Emergent Effects”. 

\bibliographystyle{alpha}
\bibliography{citations4}

\end{document}